\documentclass[13pt,reqno,a4paper,oneside,11pt]{amsart}%
\usepackage{amsfonts}
\usepackage{amsfonts}
\usepackage{amsfonts}
\usepackage{amsfonts}
\usepackage{amsfonts}
\usepackage{amsfonts}
\usepackage{amsfonts}
\usepackage{amsfonts}
\usepackage{amsfonts}
\usepackage{amsfonts}
\usepackage{amsfonts}
\usepackage{amsfonts}
\usepackage{amsfonts}
\usepackage{amsfonts}
\usepackage{amsfonts}
\usepackage{amsfonts}
\usepackage{amsfonts}
\usepackage{amsfonts}
\usepackage{amsfonts}
\usepackage{amsfonts}
\usepackage{mathrsfs}
\usepackage{mathrsfs}
\usepackage{amsfonts}
\usepackage{amssymb}
\usepackage{amsmath}
\usepackage{amsthm}
\usepackage{graphicx}
\usepackage{color}
\usepackage{extarrows}
\usepackage{cite}
\usepackage[cp1251]{inputenc}
\providecommand{\U}[1]{\protect\rule{.1in}{.1in}}
\newtheorem{theorem}{Theorem}[section]
\newtheorem{corollary}[theorem]{Corollary}
\newtheorem{lemma}[theorem]{Lemma}
\newtheorem{proposition}[theorem]{Proposition}

\newtheorem{definition}{Definition}[section]
\newtheorem{remark}[theorem]{Remark}
\newtheorem{algorithm}[theorem]{Algorithm}
\newtheorem{example}[theorem]{Example}

\theoremstyle{definition}
\theoremstyle{remark}
\numberwithin{equation}{section}

\ifx\pdfoutput\relax\let\pdfoutput=\undefined\fi
\newcount\msipdfoutput
\ifx\pdfoutput\undefined\else
\ifcase\pdfoutput\else
\ifx\paperwidth\undefined\else
\ifdim\paperheight=0pt\relax\else\pdfpageheight\paperheight\fi
\ifdim\paperwidth=0pt\relax\else\pdfpagewidth\paperwidth\fi
\fi\fi\fi
\begin{document}
\pagestyle{myheadings}
\begin{center}
{\Large \textbf{Algebraic conditions and general solution to a system of quaternion tensor equations with applications}}\footnote{This research was supported by
the grants from the National Natural
Science Foundation of China (11971294).
\par

{} First author: Mahmoud Saad Mehany (M.S. Mehany), mahmoud2006@shu.edu.cn.
\par  Second author: Qing-Wen Wang (Q.W. Wang), wqw@t.shu.edu.cn.}

\bigskip

{ \textbf{Mahmoud Saad Mehany$^{a,b}$, Qing-Wen Wang$^{a,c}$}}

{\small
\vspace{0.5cm}

$a.$ Department of Mathematics, Shanghai University, Shanghai 200444, P. R. China}\\
{\small
\vspace{0.1cm}
$c.$ Department of Mathematics, Ain Shams University, Cairo, 11566, A.R. Egypt }\\
{\small
\vspace{0.1cm}
$b.$ Collaborative Innovation Center for the Marine Artificial Intelligence, Shanghai 200444, P. R. China}
\end{center}
\vspace{0.2cm}
\begin{quotation}
\noindent\textbf{Abstract:}
 This paper investigates the necessary and sufficient algebraic
conditions to a constrained system of Sylvester-type quaternion  tensor equations. An explicit formula of the general solution regarding the Moore-Penrose inverses of some block given tensors is obtained.  As an application of a particular case, we establish the solvability conditions and the general solution to a system of Sylvester-type quaternion  tensor equations  involving $\eta$-Hermitian unknowns. An algorithm with a numerical example is proposed to compute the general solution of the main system.

\vspace{2mm}
\end{quotation}
\section{\textbf{Introduction}}
A tensor is a multidimensional array. Specifically, a tensor is a generalization of a vector or matrix to higher dimensions \cite{Qii,Qi3,QII,QII1,QII2,QII3,4}. Tensors have applications in diverse areas such as machine learning, signal processing, biology, applied mechanics, data mining, pattern recognition, and numerical approaches algorithms for computing some generalized tensor and matrix equations  \cite{6variables,10,c,11,add2,add3,add4,add5,add6,add7,add9,add10,add11,add12}. Hamilton \cite{12} was first presented the quaternion algebra over the real field $\mathbb{R}$
\begin{small}
\begin{align*}
&\mathbb{H}=\{d_{0}+d_{1}\mathbf{i}+d_{2}\mathbf{j}+d_{3}\mathbf{k} : \mathbf{i}^{2}=\mathbf{j}^{2}=\mathbf{k}^{2}=\mathbf{ijk}=-1,\ d_{0}, d_{1}, d_{2}, d_{3}\in \mathbb{R}\}.
\end{align*}
\end{small}
Quaternion algebra is considered a non-commutative division ring. Quaternion, quaternion matrices, and quaternion tensors have applications in signal processing, color image processing, control theory, computer science, statistics and probability, quantum computing \cite{13,15,22,5,Took3,Took4}. Regularization of singular systems, computation of restricted singular value decomposition, and  generalized systems of  Sylvester-type matrix and tensor equations over the complex field $\mathbb{C}$ and the quaternion algebra $\mathbb{H}$ have  been studied by many authors, see e.g. \cite{MAH,Chu10,Chu20,Chu30,a1,a2,a3,a4,a5,a12,a13,a18,a19,a20,a21}. Recently, Zhang and Kang  \cite{add1} propose the generalized modified Hermitian and skew-Hermitian splitting
approach for computing the generalized Lyapunov equation:
\begin{small}
\begin{align*}
AX+XA+\sum_{j=1}^{m}N_{j}XN_{j}^{T}+C=0,
\end{align*}
\end{small}
where $A,$ $ Nj \in \mathbb{C}^{n\times n}$ and $C=C^{T}$ are given matrices, $m \ll n$, $X \in \mathbb{C}^{n\times n}$ is the unknown matrix. However, here, we investigate the necessary and sufficient algebraic conditions for a two-sided four variable Sylvester-type linear tensor equation, and hence apply this equation to find the solvability conditions and the general solution to a constrained seven variables system of  coupled   tensor equations. The solvability conditions and  the general solution of the Sylvester-type tensor equation:
\begin{small}
\begin{align}
\label{1.1}
\begin{array}{rll}
&\mathcal{A}*_{N}\mathcal{X}*_{M}\mathcal{B}+\mathcal{C}*_{N}\mathcal{Y}*_{M}\mathcal{D}=\mathcal{E}
\end{array}
\end{align}
\end{small}
was established in \cite{18}, where $\mathcal{A}$, $\mathcal{B}$, $\mathcal{C}$, $\mathcal{D}$  and $\mathcal{E}$  are given tensors over $\mathbb{H}$. Equation $\eqref{1.1}$ has an application in the discretization  of higher dimension  linear partial differential equations \cite {19}. Here, we give a proper generalization of \eqref{1.1}, namely,
\begin{small}
\begin{equation}
\label{3.1}
\begin{gathered}
\mathcal{A}_{1}*_{N}\mathcal{X}_{1}*_{M}\mathcal{B}_{1}+\mathcal{A}_{2}*_{N}\mathcal{X}_{2}*_{M}\mathcal{B}_{2}\\
+\mathcal{A}_{2}*_{N}(\mathcal{C}_{3}*_{N}\mathcal{X}_{3}*_{M}\mathcal{D}_{3}
+\mathcal{C}_{4}*_{N}\mathcal{W}*_{M}\mathcal{D}_{4})*_{M}\mathcal{B}_{1}=\mathcal{E}_{1}.
\end{gathered}
\end{equation}
\end{small}
 Wang et al. \cite{21} gave a comprehensive discussion to the following system of  coupled two-sided Sylvester-type tensor equations:
\begin{small}
\begin{align}
\label{1.4}
\left\{
\begin{array}{rll}
\begin{gathered}
\mathcal{A}_{1}*_{N}\mathcal{X}*_{M}\mathcal{B}_{1}=\mathcal{E}_{1},\
\mathcal{A}_{2}*_{N}\mathcal{Y}*_{M}\mathcal{B}_{2}=\mathcal{E}_{2},\\
\mathcal{A}_{3}*_{N}\mathcal{Z}=\mathcal{E}_{3},\ \mathcal{Z}*_{M}\mathcal{B}_{3}=\mathcal{E}_{4}\\
\mathcal{A}_{4}*_{N}\mathcal{X}*_{M}\mathcal{B}_{4}+\mathcal{C}_{4}*_{N}\mathcal{Z}*_{M}\mathcal{D}_{4}=\mathcal{P},\\
\mathcal{A}_{5}*_{N}\mathcal{Y}*_{M}\mathcal{B}_{5}+\mathcal{C}_{5}*_{N}\mathcal{Z}*_{M}\mathcal{D}_{5}=\mathcal{Q}.
\end{gathered}
\end{array}
  \right.
\end{align}
\end{small}
They carried out the solvability conditions and the general solution in the Moore-Penrose inverses of some block given tensors. The quaternion system \eqref{1.4} considers as a proper extension of the tensor equation \eqref{1.1}.
We are motivated by wide applications of quaternion, quaternion matrices, quaternion tensors, even quaternion systems of Sylvester-type  tensor equations \cite{Chu11,Chu22,Chu33,TAO1,TAO2,TAO3,Zheng1,Zheng2},  we, in this paper, investigate the algebraic solvability conditions and an  expression of the general solution to the following constrained system of Sylvester-type quaternion  tensor equations:
\begin{small}
\begin{align}
\label{1.5}
\left\{
\begin{array}{rll}
\begin{gathered}
\mathcal{A}_{1}*_{N}\mathcal{X}_{3}*_{M}\mathcal{B}_{1}=\mathcal{E}_{1},\
 \mathcal{A}_{2}*_{N}\mathcal{Y}_{3}*_{M}\mathcal{B}_{2}=\mathcal{E}_{2},\\
\mathcal{A}_{4}*_{N}\mathcal{X}_{1}=\mathcal{E}_{5},\ \mathcal{X}_{2}*_{M}\mathcal{B}_{4}=\mathcal{E}_{6},\\
\mathcal{A}_{5}*_{N}\mathcal{Y}_{1}=\mathcal{E}_{7},\ \mathcal{Y}_{2}*_{M}\mathcal{B}_{5}=\mathcal{E}_{8},\\
\mathcal{A}_{3}*_{N}\mathcal{W}=\mathcal{E}_{3},\ \mathcal{W}*_{M}\mathcal{B}_{3}=\mathcal{E}_{4},\\
\mathcal{A}_{6}*_{N}\mathcal{X}_{1}*_{M}\mathcal{B}_{6}+\mathcal{A}_{7}*_{N}\mathcal{X}_{2}*_{M}\mathcal{B}_{7}\\
+\mathcal{A}_{7}*_{N}(\mathcal{C}_{3}*_{N}\mathcal{X}_{3}*_{M}\mathcal{D}_{3}
+\mathcal{C}_{4}*_{N}\mathcal{W}*_{M}\mathcal{D}_{4})*_{M}\mathcal{B}_{6}=\mathcal{E}_{9},\\
\mathcal{A}_{8}*_{N}\mathcal{Y}_{1}*_{M}\mathcal{B}_{8}+\mathcal{A}_{9}*_{N}\mathcal{Y}_{2}*_{M}\mathcal{B}_{9}\\
+\mathcal{A}_{9}*_{N}(\mathcal{H}_{3}*_{N}\mathcal{Y}_{3}*_{M}\mathcal{J}_{3}
+\mathcal{H}_{4}*_{N}\mathcal{W}*_{M}\mathcal{J}_{4})*_{M}\mathcal{B}_{8}=\mathcal{E}_{10}.
\end{gathered}
\end{array}
  \right.
\end{align}
\end{small}
As a particular case of \eqref{1.5}, we derive the solvability conditions and the general  solution to the system of Sylvester-type  tensor equations:
\begin{small}
\begin{align}
\label{1.6}
\left\{
\begin{array}{rll}
\begin{gathered}
\mathcal{A}_{1}*_{N}\mathcal{X}_{3}*_{M}\mathcal{B}_{1}=\mathcal{E}_{1},\
\mathcal{A}_{2}*_{N}\mathcal{Y}_{3}*_{M}\mathcal{B}_{2}=\mathcal{E}_{2},\\
\mathcal{A}_{3}*_{N}\mathcal{W}=\mathcal{E}_{3},\ \mathcal{W}*_{M}\mathcal{B}_{3}=\mathcal{E}_{4},\\
\mathcal{A}_{6}*_{N}\mathcal{X}_{1}+\mathcal{X}_{2}*_{M}\mathcal{B}_{7}
+\mathcal{C}_{3}*_{N}\mathcal{X}_{3}*_{M}\mathcal{D}_{3}
+\mathcal{C}_{4}*_{N}\mathcal{W}*_{M}\mathcal{D}_{4}=\mathcal{E}_{9},\\
\mathcal{A}_{8}*_{N}\mathcal{Y}_{1}+\mathcal{Y}_{2}*_{M}\mathcal{B}_{9}
+\mathcal{H}_{3}*_{N}\mathcal{Y}_{3}*_{M}\mathcal{J}_{3}
+\mathcal{H}_{4}*_{N}\mathcal{W}*_{M}\mathcal{J}_{4}=\mathcal{E}_{10}.
\end{gathered}
\end{array}
  \right.
\end{align}
\end{small}
Took et al. \cite{Took3} define  an ${\eta}$-Hermitian matrix, for ${\eta} \in \{\mathbf{i},\mathbf{j},\mathbf{k}\}$, a quaternion square matrix $B$ over $\mathbb{H}$ is said to be an ${\eta}$-Hermitian matrix if $B^{\eta^{*}} = B$, where $B^{\eta^{*}} = -{\eta}B^{*}{\eta}$. ${\eta}$-Hermitian matrices have applications in statistical signal processing and linear modeling \cite{Took3,Took4,Took1,Took2}. As a direct implementation of the particular case \eqref{1.6}, this study investigates the necessary and sufficient algebriac conditions for the existence of a general solution to the following  system of Sylvester-type  tensor equations:
\begin{small}
\begin{align}
\label{1.7a}
\left\{
\begin{array}{rll}
\begin{gathered}
\mathcal{A}_{1}*_{N}\mathcal{X}_{3}*_{M}\mathcal{A}_{1}^{\eta^{*}}=\mathcal{E}_{1},\
\mathcal{A}_{2}*_{N}\mathcal{Y}_{3}*_{M}\mathcal{A}_{2}^{\eta^{*}}=\mathcal{E}_{2},\
\mathcal{A}_{3}*_{N}\mathcal{W}=\mathcal{E}_{3},\\
\mathcal{A}_{6}*_{N}\mathcal{X}_{1}+(\mathcal{A}_{6}*_{N}\mathcal{X}_{1})^{\eta^{*}}
+\mathcal{C}_{3}*_{N}\mathcal{X}_{3}*_{M}\mathcal{C}_{3}^{\eta^{*}}
+\mathcal{C}_{4}*_{N}\mathcal{W}*_{M}\mathcal{C}_{4}^{\eta^{*}}=\mathcal{E}_{9},\\
\mathcal{A}_{8}*_{N}\mathcal{Y}_{1}+(\mathcal{A}_{8}*_{N}\mathcal{Y}_{1})^{\eta^{*}}
+\mathcal{H}_{3}*_{N}\mathcal{Y}_{3}*_{M}\mathcal{H}_{3}^{\eta^{*}}
+\mathcal{H}_{4}*_{N}\mathcal{W}*_{M}\mathcal{H}_{4}^{\eta^{*}}=\mathcal{E}_{10}
\end{gathered}
\end{array}
  \right.
\end{align}
\end{small}
where $\mathcal{X}_{3}$, $\mathcal{Y}_{3}$, and $\mathcal{W}$ are $\eta$-Hermitian unknowns and $\mathcal{E}_{i}=\mathcal{E}_{i}^{\eta^{*}}$  $i \in \{1,2,9,10\}$.

 This paper is organized as follows. In Section 2, we recall some basic definitions and well-known results. Section  3  continues the algebraic solvability  conditions and the general solution to \eqref{1.5}. In Section  4,  we investigate the particular case \eqref{1.6}. Consequently, we carry out the solvability conditions and the general solution to \eqref{1.7a}. In Section 5, we summarized the results in giving the main conclusions.
\section{\textbf{Preliminaries}}
Throughout this paper, consider all tensors to be quaternion tensors. For convenience, we utilize the symbol $I(M)$ instead of $I_{1}\times I_{2}\times...\times I_{M}$, for some positive integers $I_{1},$ ..., $I_{M}$, $M$. A tensor $\mathcal{P}\in \mathbb{H}^{I_{1}\times I_{2}\times...\times I_{M}\times J_{1}\times J_{2}\times...\times J_{N}}$ can be written in the more straightforward form $\mathcal{P}\in \mathbb{H}^{I(M)\times J(N)}$. A tensor $\mathcal{P}\in \mathbb{H}^{I(N)\times J(N)}$ is called an even-order tensor. An even-order tensor $\mathcal{Q}\in \mathbb{H}^{I(N)\times I(N)}$ is called an even-order square tensor. For a fixed element $q \in \mathbb{H}$, the symbol $\overline{q}$ stands for the conjugate of $q$. A quaternion tensor $\mathcal{P}^{*}$= $(\overline{p}_{j_{1}..j_{M}i_{1}..i_{N}}) \in \mathbb{H}^{J(M)\times I(N)}$ calls the conjugate transpose of the tensor $\mathcal{P}$= $(p_{i_{1}..i_{N}j_{1}..j_{M}}) \in \mathbb{H}^{I(N)\times J(M)}$. If $\mathcal{P}=\mathcal{P}^{*}$, then $\mathcal{P}$  is called a quaternion Hermitian tensor.
\begin{definition}\ \cite{22}
Let $\mathcal{A} \in \mathbb{H}^{I(N)\times J(N)}$ and $\mathcal{B} \in \mathbb{H}^{J(N)\times K(M)}$, then the Einstein product of  $\mathcal{A}$ and $\mathcal{B}$ is denoted by $\mathcal{A}*_{N}\mathcal{B} \in \mathbb{H}^{I(N)\times K(M)}$, where \[(\mathcal{A}*_{N}\mathcal{B})_{i_{1}..i_{N}k_{1}..k_{M}}=
	\sum_{j_{1}...j_{N}}a_{i_{1}...i_{N}j_{1}...j_{N}}b_{j_{1}...j_{N}k_{1}...k_{M}}.\] Moreover, $*_{N}$ is associative over the set of all tensors with qualified order.
\end{definition}
\begin{definition}\cite{23}
An even order square tensor $\mathcal{D} = (d_{i_{1}...i_{N}i_{1}...i_{N}})\in \mathbb{H}^{I(N)\times I(N)}$ is called a diagonal tensor if  $d_{i_{1}...i_{N}i_{1}...i_{N}}\neq 0$ and all its entries are zero. A diagonal tensor is said to be a unit tensor if $d_{i_{1}...i_{N}i_{1}...i_{N}}= 1$, which denotes by $\mathcal{I}$.
\end{definition}
\begin{definition}\cite{23}
Let $\mathcal{A} = (a_{i_{1}...i_{N}j_{1}...j_{M}})\in \mathbb{H}^{I(N)\times J(M)}$, $\mathcal{B} = (b_{i_{1}...i_{N}k_{1}...k_{M}})\in \mathbb{H}^{I(N)\times K(M)}$, then the  ``row block tenso'' of $\mathcal{A}$ and $\mathcal{B}$ is denoted by
\begin{small}
\begin{align}
\label{2.1}
\begin{array}{rll}
&\begin{pmatrix}\mathcal{A}&\mathcal{B}\end{pmatrix} \in \mathbb{H}^{I(N)\times L(M)},
\end{array}
\end{align}
\end{small}
where $L_{s}=J_{s}+K_{s}$, $s=1, ...,M$ define as
\begin{small}
\begin{align*}
&\begin{pmatrix}\mathcal{A}&\mathcal{B}\end{pmatrix}_{i_{1}...i_{N}l_{1}...l_{M}}=
\left\{
\begin{array}{rll}
&a_{i_{1}...i_{N}l_{1}...l_{M}},\ if\ i_{1}...i_{N} \in [I_{1}]... [I_{N}],\ l_{1}...l_{M} \in [J_{1}]... [J_{M}],\\
&b_{i_{1}...i_{N}l_{1}...l_{M}},\ if\ i_{1}...i_{N} \in [I_{1}]... [I_{N}],\ l_{1}...l_{M} \in \Gamma_{1}... \Gamma_{M},\\
&0,\ \ \ \ \ otherwise,
\end{array}
  \right.
\end{align*}
\end{small}
where $\Gamma_{s}=\{J_{s}+1, ..., J_{s}+K_{s}\}$, $s=1, ...,M$. For a given tensors $\mathcal{C} = (c_{j_{1}...j_{M}i_{1}...i_{N}})\in \mathbb{H}^{J(M)\times I(N)}$, $\mathcal{D} = (d_{k_{1}...k_{M}i_{1}...i_{N}})\in \mathbb{H}^{K(M)\times I(N)}$. The``column block tensor'' of $\mathcal{C}$ and $\mathcal{D}$ is denoted by
\begin{small}
\begin{align}
\label{2.2}
\begin{array}{rll}
&\begin{pmatrix}\mathcal{C}\\ \mathcal{D}\end{pmatrix} \in \mathbb{H}^{L(M)\times I(N)},
\end{array}
\end{align}
\end{small}
where $L_{s}=J_{s}+K_{s}$, $s=1, ...,M$ define as
\begin{small}
\begin{align*}
&\begin{pmatrix}\mathcal{C}\\ \mathcal{D}\end{pmatrix}_{l_{1}...l_{M}i_{1}...i_{N}}=
\left\{
\begin{array}{rll}
&c_{l_{1}...l_{M}i_{1}...i_{N}},\ if\ l_{1}...l_{M} \in [J_{1}]...[J_{M}],\ i_{1}...i_{N} \in [I_{1}]... [I_{N}],\\
&d_{l_{1}...l_{M}i_{1}...i_{N}},\ if\ l_{1}...l_{M} \in \Gamma_{1}... \Gamma_{M},\ i_{1}...i_{N} \in [I_{1}]...[I_{N}],\\
&0,\ \ \ \ \ otherwise,
\end{array}
  \right.
\end{align*}
\end{small}
where $\Gamma_{s}=\{J_{s}+1, ..., J_{s}+K_{s}\}$, $s=1, ...,M$.
\end{definition}
\begin{proposition}\cite{23} Let $\mathcal{A} \in \mathbb{H}^{I(P)\times K(N)}$ and $\mathcal{B} \in \mathbb{H}^{K(N)\times J(M)}$, then
\begin{small}
\begin{enumerate}
\item $(\mathcal{A}*_{N}\mathcal{B})^{*}=\mathcal{B}^{*}*_{N}\mathcal{A}^{*}$,
\item $\mathcal{I_{N}}*_{N}\mathcal{B}=\mathcal{B}$,\ $\mathcal{B}*_{M}\mathcal{I}_{M}=\mathcal{B}$,
where $\mathcal{I_{N}} \in \mathbb{H}^{K(N)\times K(N)}$ and $\mathcal{I_{M}} \in \mathbb{H}^{J(M)\times J(M)}$ are unit tensors.
\end{enumerate}
\end{small}
\end{proposition}
\begin{proposition}\cite{23} Consider the tensors $\begin{pmatrix}\mathcal{A}& \mathcal{B}\end{pmatrix}$ and $\begin{pmatrix}\mathcal{C}\\ \mathcal{D}\end{pmatrix}$ be given in \eqref{2.1} and \eqref{2.2}, respectively. For a given  tensor $\mathcal{G} \in \mathbb{H}^{I(N)\times I(N)}$, we have that
\begin{small}
\begin{enumerate}
  \item $\mathcal{G}*_{N}\begin{pmatrix}\mathcal{A}& \mathcal{B}\end{pmatrix}=
\begin{pmatrix}\mathcal{G}*_{N}\mathcal{A}& \mathcal{G}*_{N}\mathcal{B}\end{pmatrix}\in \mathbb{H}^{I(N)\times L(M)},$
  \item $\begin{pmatrix}\mathcal{C}\\ \mathcal{D}\end{pmatrix}*_{N}\mathcal{G}=
\begin{pmatrix}\mathcal{C}*_{N}\mathcal{G}\\ \mathcal{D}*_{N}\mathcal{G}\end{pmatrix}\in \mathbb{H}^{L(M)\times I(N)},$
  \item $\begin{pmatrix}\mathcal{A}& \mathcal{B}\end{pmatrix}*_{M}\begin{pmatrix}\mathcal{C}\\ \mathcal{D}\end{pmatrix}=
\mathcal{A}*_{M}\mathcal{C}+\mathcal{B}*_{M}\mathcal{D}\in \mathbb{H}^{I(N)\times I(N)}.$
\end{enumerate}
\end{small}
\end{proposition}

\begin{definition}\cite{18}
Let $\mathcal{D} \in \mathbb{H}^{I(N)\times J(N)}$, then the Moore-Penrose inverse of $\mathcal{D}$  is the unique  tensor $\mathcal{D}^{\dagger}\in \mathbb{H}^{J(N)\times I(N)}$  satisfies:
\begin{small}
\begin{enumerate}
 \item $\mathcal{D}*_{N}\mathcal{D}^{\dagger}*_{N}\mathcal{D}=\mathcal{D}$,
 \item $\mathcal{D}^{\dagger}*_{N}\mathcal{D}*_{N}\mathcal{D}^{\dagger}=\mathcal{D}^{\dagger}$,
 \item $(\mathcal{D}*_{N}\mathcal{D}^{\dagger})^{*}=\mathcal{D}*_{N}\mathcal{D}^{\dagger}$,
 \item $(\mathcal{D}^{\dagger}*_{N}\mathcal{D})^{*}=\mathcal{D}^{\dagger}*_{N}\mathcal{D}$.
 \end{enumerate}
 \end{small}
where $\mathcal{R}_{\mathcal{D}}=I-\mathcal{D}*_{N}\mathcal{D}^{\dagger}$ and  $\mathcal{L}_{\mathcal{D}}=I-\mathcal{D}^{\dagger}*_{N}\mathcal{D}$ denote the projections along $\mathcal{D}$.
\end{definition}
\begin{definition}\cite{18}
Let $\eta$ be an element in the quaternion algebra basis $\{\mathbf{i},\mathbf{j},\mathbf{k}\}$. A tensor
$\mathcal{D} \in \mathbb{H}^{I(N)\times I(N)}$ is said to be $\eta$-Hermitian if $\mathcal{D}=\mathcal{D}^{\eta^{*}}$, where
$\mathcal{D}^{\eta^{*}}=-\eta \mathcal{D}^{*}\eta$.
\end{definition}
\begin{proposition}\cite{18}
Let $\mathcal{D} \in \mathbb{H}^{I(N)\times I(N)}$, then we have that
\begin{small}
\begin{enumerate}
\item $\mathcal{L}_{\mathcal{D}}*_{N}\mathcal{D}^{\dagger}=\mathcal{D}*_{N}\mathcal{L}_{\mathcal{D}}=0,\
\mathcal{R}_{\mathcal{D}}*_{N}\mathcal{D}=\mathcal{D}^{\dagger}*_{N}\mathcal{R}_{\mathcal{D}}=0$,
\item $(\mathcal{D}^{*})^{\dagger}=(\mathcal{D}^{\dagger})^{*},\ (\mathcal{D}^{\eta^{*}})^{\dagger}=(\mathcal{D}^{\dagger})^{\eta^{*}},$
\item $(\mathcal{L}_{\mathcal{D}})^{\eta^{*}}=\mathcal{R}_{\mathcal{D}^{\eta^{*}}},\ (\mathcal{R}_{\mathcal{D}})^{\eta^{*}}=\mathcal{L}_{\mathcal{D}^{\eta^{*}}},$
\item $(\mathcal{D}^{*}*_{N}\mathcal{D})^{\dagger}=\mathcal{D}^{\dagger}*_{N}(\mathcal{D}^{*})^{\dagger},\
(\mathcal{D}*_{N}\mathcal{D}^{*})^{\dagger}=(\mathcal{D}^{*})^{\dagger}*_{N}\mathcal{D}^{\dagger}.$
\end{enumerate}
\end{small}
\end{proposition}
\begin{lemma}\label{lma 2.3}\cite{18} Let $\mathcal{A} \in \mathbb{H}^{I(N)\times J(N)}$,\ $\mathcal{B} \in \mathbb{H}^{K(M)\times L(M)}$,\
$\mathcal{C} \in \mathbb{H}^{I(N)\times G(N)}$,\ $\mathcal{D} \in \mathbb{H}^{H(M)\times L(M)}$\ and
$\mathcal{E} \in \mathbb{H}^{I(N)\times L(M)}$. Set
\begin{small}
\begin{align*}
&\mathcal{P}=\mathcal{R}_{\mathcal{A}}*_{N}\mathcal{C},\ \mathcal{Q}=\mathcal{D}*_{M}\mathcal{L}_{\mathcal{B}},\ \mathcal{S}=\mathcal{C}*_{N}\mathcal{L}_{\mathcal{P}}.
\end{align*}
\end{small}
Then \eqref{1.1} is solvable if and only if
\begin{small}
\begin{align*}
&\mathcal{R}_{\mathcal{P}}*_{N}\mathcal{R}_{\mathcal{A}}*_{N}\mathcal{E}=0,\ \mathcal{E}*_{M}\mathcal{L}_{\mathcal{B}}*_{M}\mathcal{L}_{\mathcal{Q}}=0,\\
&\mathcal{R}_{\mathcal{A}}*_{N}\mathcal{E}*_{M}\mathcal{L}_{\mathcal{D}}=0,\ \mathcal{R}_{\mathcal{C}}*_{N}\mathcal{E}*_{M}\mathcal{L}_{\mathcal{B}}=0.
\end{align*}
\end{small}
In that case, the general solution to \eqref{1.1} can be expressed as follows:
\begin{small}
\begin{align*}
&\mathcal{X}=\mathcal{A}^{\dagger}*_{N}\mathcal{E}*_{M}\mathcal{B}^{\dagger}
-\mathcal{A}^{\dagger}*_{N}\mathcal{C}*_{N}\mathcal{P}^{\dagger}*_{N}\mathcal{E}*_{M}\mathcal{B}^{\dagger}
-\mathcal{A}^{\dagger}*_{N}\mathcal{S}*_{N}\mathcal{C}^{\dagger}*_{N}\mathcal{E}*_{M}\mathcal{Q}^{\dagger}\\
&\ \ \ \ \ \ *_{M}\mathcal{D}*_{M}\mathcal{B}^{\dagger}-\mathcal{A}^{\dagger}*_{N}\mathcal{S}*_{N}\mathcal{U}_{2}*_{M}\mathcal{R}_{\mathcal{Q}}*_{M}\mathcal{D}*_{M}\mathcal{B}^{\dagger}
+\mathcal{L}_{\mathcal{A}}*_{N}\mathcal{U}_{4}+\mathcal{U}_{5}*_{M}\mathcal{R}_{\mathcal{B}},\\
&\mathcal{Y}=\mathcal{P}^{\dagger}*_{N}\mathcal{E}*_{M}\mathcal{D}^{\dagger}+\mathcal{S}^{\dagger}*_{N}\mathcal{S}*_{N}\mathcal{C}^{\dagger}
*_{N}\mathcal{E}*_{M}\mathcal{Q}^{\dagger}+\mathcal{L}_{\mathcal{P}}*_{N}\mathcal{L}_{\mathcal{S}}*_{N}\mathcal{U}_{1}+
\mathcal{L}_{\mathcal{P}}*_{N}\mathcal{U}_{2}\\
&\ \ \ \ \ \ *_{M}\mathcal{R}_{\mathcal{Q}}+\mathcal{U}_{3}*_{M}\mathcal{R}_{\mathcal{D}},
\end{align*}
\end{small}
where $\mathcal{U}_{1},$ $\mathcal{U}_{2},$ $\mathcal{U}_{3},$ $\mathcal{U}_{4}$ and $\mathcal{U}_{5}$ are arbitrary  tensors with suitable orders.
\end{lemma}
\begin{lemma}\label{lma 2.3a}\cite{21}
Consider the  system of tensor equations \eqref{1.4}, where
\begin{small}
\begin{align*}
&\mathcal{A}_{1} \in \mathbb{H}^{I(N)\times J(N)},\ \mathcal{A}_{2} \in \mathbb{H}^{I(N)\times Q(N)},\ \mathcal{A}_{3} \in \mathbb{H}^{I(N)\times P(N)},\
\mathcal{A}_{4} \in \mathbb{H}^{I(N)\times J(N)},\\
&\mathcal{A}_{5} \in \mathbb{H}^{I(N)\times Q(N)},\ \mathcal{B}_{1} \in \mathbb{H}^{L(M)\times K(M)},\ \mathcal{B}_{2} \in \mathbb{H}^{S(M)\times K(M)},\ \mathcal{B}_{3} \in \mathbb{H}^{T(M)\times K(M)},\\
&\mathcal{B}_{4} \in \mathbb{H}^{L(M)\times K(M)},\ \mathcal{B}_{5} \in \mathbb{H}^{S(M)\times K(M)},\ \mathcal{C}_{4} \in \mathbb{H}^{I(N)\times P(N)},\ \mathcal{C}_{5} \in \mathbb{H}^{I(N)\times P(N)},\\
&\mathcal{D}_{4} \in \mathbb{H}^{K(M)\times K(M)},\ \mathcal{D}_{5} \in \mathbb{H}^{K(M)\times K(M)},\
\mathcal{P} \in \mathbb{H}^{I(N)\times K(M)},\ \mathcal{E}_{i} \in \mathbb{H}^{I(N)\times K(M)}
\end{align*}
\end{small}
$(i=\overline{1,4})$ are given tensors over $\mathbb{H}$. Set
\begin{small}
\begin{align*}
&\mathcal{A}_{6}=\mathcal{C}_{4}*_{N}\mathcal{L}_{\mathcal{A}_{3}},\ \mathcal{B}_{6}=\mathcal{R}_{\mathcal{B}_{3}}*_{M}\mathcal{D}_{4},\
\mathcal{A}_{7}=\mathcal{C}_{5}*_{N}\mathcal{L}_{\mathcal{A}_{3}},\ \mathcal{B}_{7}=\mathcal{R}_{\mathcal{B}_{3}}*_{M}\mathcal{D}_{5},\\
&\mathcal{G}=\mathcal{P}-\mathcal{C}_{4}*_{N}\mathcal{A}_{3}^{\dagger}*_{N}\mathcal{E}_{3}*_{M}\mathcal{D}_{4}
-\mathcal{C}_{4}*_{N}\mathcal{L}_{\mathcal{A}_{3}}*_{N}\mathcal{E}_{4}*_{M}\mathcal{B}_{3}^{\dagger}*_{M}\mathcal{D}_{4},\\
&\mathcal{F}=\mathcal{Q}-\mathcal{C}_{5}*_{N}\mathcal{A}_{3}^{\dagger}*_{N}\mathcal{E}_{3}*_{M}\mathcal{D}_{5}
-\mathcal{C}_{5}*_{N}\mathcal{L}_{\mathcal{A}_{3}}*_{N}\mathcal{E}_{4}*_{M}\mathcal{B}_{3}^{\dagger}*_{M}\mathcal{D}_{5},\\
&\mathcal{M}_{1}=\mathcal{R}_{\mathcal{A}_{4}}*_{M}\mathcal{A}_{6},\ \mathcal{N}_{1}=\mathcal{B}_{6}*_{M}\mathcal{L}_{\mathcal{B}_{4}},\
\mathcal{S}_{1}=\mathcal{A}_{6}*_{N}\mathcal{L}_{\mathcal{M}_{1}},\\
&\mathcal{M}_{2}=\mathcal{R}_{\mathcal{A}_{5}}*_{M}\mathcal{A}_{7},\ \mathcal{N}_{2}=\mathcal{B}_{7}*_{M}\mathcal{L}_{\mathcal{B}_{5}},\
\mathcal{S}_{2}=\mathcal{A}_{7}*_{N}\mathcal{L}_{\mathcal{M}_{2}},\\
&\mathcal{A}_{11}=\begin{pmatrix}\mathcal{L}_{\mathcal{M}_{1}}*_{N}\mathcal{L}_{\mathcal{S}_{1}}& \mathcal{L}_{\mathcal{M}_{2}}*_{N}\mathcal{L}_{\mathcal{S}_{2}}\end{pmatrix},\
\mathcal{B}_{11}=\begin{pmatrix}\mathcal{R}_{\mathcal{B}_{6}}\\ \mathcal{R}_{\mathcal{B}_{7}}\end{pmatrix},\
\mathcal{A}=\mathcal{R}_{\mathcal{A}_{11}}*_{N}\mathcal{L}_{\mathcal{M}_{1}},\\
&
\mathcal{B}=\mathcal{R}_{\mathcal{N}_{1}}*_{M}\mathcal{L}_{\mathcal{B}_{11}},\
\mathcal{C}=\mathcal{R}_{\mathcal{A}_{11}}*_{N}\mathcal{L}_{\mathcal{M}_{2}},\
\mathcal{D}=\mathcal{R}_{\mathcal{N}_{2}}*_{M}\mathcal{L}_{\mathcal{B}_{11}},\ \mathcal{S}=\mathcal{C}*_{N}\mathcal{L}_{\mathcal{M}},\\
&\mathcal{E}=\mathcal{R}_{\mathcal{A}_{11}}*_{N}\mathcal{E}_{11}*_{M}\mathcal{L}_{\mathcal{B}_{11}},\
\mathcal{M}=\mathcal{R}_{\mathcal{A}}*_{N}\mathcal{C},\
\mathcal{N}=\mathcal{D}*_{M}\mathcal{L}_{\mathcal{B}},\ \mathcal{C}_{22}=\mathcal{A}_{4}^{\dagger}*_{N}\mathcal{S}_{1}\\
&\mathcal{A}_{22}=\begin{pmatrix}\mathcal{L}_{\mathcal{A}_{1}}& \mathcal{L}_{\mathcal{A}_{4}}\end{pmatrix},\
\mathcal{B}_{22}=\begin{pmatrix}\mathcal{R}_{\mathcal{B}_{1}}\\ \mathcal{R}_{\mathcal{B}_{4}}\end{pmatrix},\
\mathcal{D}_{22}=\mathcal{R}_{\mathcal{N}_{1}}*_{M}\mathcal{B}_{6}*_{N}\mathcal{B}_{4}^{\dagger},\\
&\mathcal{A}_{33}=\mathcal{R}_{\mathcal{A}_{22}}*_{N}\mathcal{C}_{22},\
\mathcal{B}_{33}=\mathcal{D}_{22}*_{M}\mathcal{L}_{\mathcal{B}_{22}},\
\mathcal{E}_{33}=\mathcal{R}_{\mathcal{A}_{22}}*_{N}\mathcal{E}_{22}*_{M}\mathcal{L}_{\mathcal{B}_{22}},\\
&\mathcal{A}_{44}=\begin{pmatrix}\mathcal{L}_{\mathcal{A}_{2}}& \mathcal{L}_{\mathcal{A}_{5}}\end{pmatrix},\
\mathcal{B}_{44}=\begin{pmatrix}\mathcal{R}_{\mathcal{B}_{2}}\\ \mathcal{R}_{\mathcal{B}_{5}}\end{pmatrix},\
\mathcal{C}_{44}=\mathcal{A}_{5}^{\dagger}*_{N}\mathcal{S}_{2},\\
&\begin{array}{l}
\mathcal{D}_{44}=\mathcal{R}_{\mathcal{N}_{2}}*_{M}\mathcal{B}_{7}*_{N}\mathcal{B}_{5}^{\dagger},\ \mathcal{E}_{11}=\mathcal{M}_{2}^{\dagger}*_{N}\mathcal{F}*_{M}\mathcal{B}_{7}^{\dagger}
+\mathcal{S}_{2}^{\dagger}*_{N}\mathcal{S}_{2}*_{N}\mathcal{A}_{7}^{\dagger}
*_{N}\mathcal{F}\\
\ \ \ \ \ \ *_{M}\mathcal{N}_{2}^{\dagger}
-\mathcal{M}_{1}^{\dagger}*_{N}\mathcal{G}*_{M}\mathcal{A}_{6}^{\dagger}-\mathcal{S}_{1}^{\dagger}*_{N}\mathcal{S}_{1}*_{N}\mathcal{A}_{6}^{\dagger}*_{N}\mathcal{G}*_{M}\mathcal{N}_{1}^{\dagger},
\end{array}\\
&\begin{array}{l}
\mathcal{E}_{22}=\mathcal{A}_{4}^{\dagger}*_{N}\mathcal{G}*_{M}\mathcal{B}_{4}^{\dagger}
-\mathcal{A}_{1}^{\dagger}*_{N}\mathcal{E}_{1}*_{M}\mathcal{B}_{1}^{\dagger}
-\mathcal{A}_{4}^{\dagger}*_{N}\mathcal{S}_{1}*_{N}\mathcal{A}_{6}^{\dagger}*_{N}\mathcal{G}*_{M}\mathcal{N}_{1}\\
\ \ \ \ \ \ *_{M}\mathcal{B}_{6}*_{M}\mathcal{B}_{4}^{\dagger}-\mathcal{A}_{4}^{\dagger}*_{N}\mathcal{A}_{6}*_{N}\mathcal{M}_{1}^{\dagger}
*_{N}\mathcal{G}*_{M}\mathcal{B}_{4}^{\dagger},
\end{array}
\end{align*}
\begin{align*}
&\begin{array}{l}
\mathcal{E}_{44}=\mathcal{A}_{5}^{\dagger}*_{N}\mathcal{F}*_{M}\mathcal{B}_{5}^{\dagger}
-\mathcal{A}_{2}^{\dagger}*_{N}\mathcal{E}_{2}*_{M}\mathcal{B}_{2}^{\dagger}
-\mathcal{A}_{5}^{\dagger}*_{N}\mathcal{S}_{2}*_{N}\mathcal{A}_{7}^{\dagger}*_{N}\mathcal{F}*_{M}\mathcal{N}_{2}\\
\ \ \ \ \ \ *_{M}\mathcal{B}_{7}*_{M}\mathcal{B}_{5}^{\dagger}-\mathcal{A}_{5}^{\dagger}*_{N}\mathcal{A}_{7}*_{N}\mathcal{M}_{2}^{\dagger}*_{N}\mathcal{F}
*_{M}\mathcal{B}_{5}^{\dagger},
\end{array}\\
&\begin{array}{l}
\mathcal{E}_{66}=\mathcal{A}_{33}^{\dagger}*_{N}\mathcal{E}_{33}*_{M}\mathcal{B}_{33}^{\dagger}
-\mathcal{A}^{\dagger}*_{N}\mathcal{E}*_{M}\mathcal{B}^{\dagger}
+\mathcal{A}^{\dagger}*_{N}\mathcal{S}*_{N}\mathcal{C}^{\dagger}*_{N}\mathcal{E}*_{M}\mathcal{N}^{\dagger}\\
\ \ \ \ \ \ *_{M}\mathcal{D}*_{N}\mathcal{B}^{\dagger}+\mathcal{A}^{\dagger}*_{N}\mathcal{C}*_{N}\mathcal{M}^{\dagger}*_{N}\mathcal{E}
*_{M}\mathcal{B}^{\dagger},\ \mathcal{C}_{66}=\mathcal{A}^{\dagger}*_{N}\mathcal{S},\
\end{array}\\
&\mathcal{A}_{55}=\mathcal{R}_{\mathcal{A}_{44}}*_{N}\mathcal{C}_{44},\
\mathcal{B}_{55}=\mathcal{D}_{44}*_{M}\mathcal{L}_{\mathcal{B}_{44}},\
\mathcal{E}_{55}=\mathcal{R}_{\mathcal{A}_{44}}*_{N}\mathcal{E}_{44}*_{M}\mathcal{L}_{\mathcal{B}_{44}},\\
&\mathcal{A}_{66}=\begin{pmatrix}\mathcal{L}_{\mathcal{A}}& \mathcal{L}_{\mathcal{A}_{33}}\end{pmatrix},\
\mathcal{B}_{66}=\begin{pmatrix}\mathcal{R}_{\mathcal{B}}\\ \mathcal{R}_{\mathcal{B}_{33}}\end{pmatrix},\
\mathcal{D}_{66}=\mathcal{R}_{\mathcal{N}}*_{M}\mathcal{D}*_{M}\mathcal{B}^{\dagger},\\
&\mathcal{A}_{77}=\mathcal{R}_{\mathcal{A}_{66}}*_{N}\mathcal{C}_{66},\
\mathcal{B}_{77}=\mathcal{D}_{66}*_{M}\mathcal{L}_{\mathcal{B}_{66}},\
\mathcal{E}_{77}=\mathcal{R}_{\mathcal{A}_{66}}*_{N}\mathcal{E}_{66}*_{M}\mathcal{L}_{\mathcal{B}_{66}},\\
&\mathcal{A}_{88}=\begin{pmatrix}\mathcal{L}_{\mathcal{M}}*_{N}\mathcal{L}_{\mathcal{S}}& \mathcal{L}_{\mathcal{A}_{55}}\end{pmatrix},\
\mathcal{B}_{88}=\begin{pmatrix}\mathcal{R}_{\mathcal{D}}\\ \mathcal{R}_{\mathcal{B}_{55}}\end{pmatrix},\
\mathcal{C}_{88}=\mathcal{L}_{\mathcal{M}},\
\mathcal{D}_{88}=\mathcal{R}_{\mathcal{N}},\\
&\mathcal{E}_{88}=\mathcal{A}_{55}^{\dagger}*_{N}\mathcal{E}_{55}*_{M}\mathcal{B}_{55}^{\dagger}
-\mathcal{M}^{\dagger}*_{N}\mathcal{E}*_{M}\mathcal{B}^{\dagger}
-\mathcal{S}^{\dagger}*_{N}\mathcal{S}*_{N}\mathcal{C}^{\dagger}*_{N}\mathcal{E}*_{M}\mathcal{N}^{\dagger},\\
&\mathcal{A}_{99}=\mathcal{R}_{\mathcal{A}_{88}}*_{N}\mathcal{C}_{88},\
\mathcal{B}_{99}=\mathcal{D}_{88}*_{M}\mathcal{L}_{\mathcal{B}_{88}},\
\mathcal{E}_{99}=\mathcal{R}_{\mathcal{A}_{88}}*_{N}\mathcal{E}_{88}*_{M}\mathcal{L}_{\mathcal{B}_{88}},\\
&\mathcal{\widetilde{A}}=\begin{pmatrix}\mathcal{L}_{\mathcal{A}_{77}}& -\mathcal{L}_{\mathcal{A}_{99}}\end{pmatrix},\
\mathcal{\widetilde{B}}=\begin{pmatrix}\mathcal{R}_{\mathcal{B}_{77}}\\ -\mathcal{R}_{\mathcal{B}_{99}}\end{pmatrix},\\
&\mathcal{\widetilde{E}}=\mathcal{A}_{77}^{\dagger}*_{N}\mathcal{E}_{77}*_{M}\mathcal{B}_{77}^{\dagger}
-\mathcal{A}_{99}^{\dagger}*_{N}\mathcal{E}_{99}*_{M}\mathcal{B}_{99}^{\dagger}.\
\end{align*}
\end{small}
Then  \eqref{1.4} is solvable  if and only if
\begin{small}
\begin{align*}
&\mathcal{R}_{\mathcal{A}_{3}}*_{N}\mathcal{E}_{3}=0,\ \mathcal{E}_{4}*_{M}\mathcal{L}_{\mathcal{B}_{3}}=0,\
\mathcal{A}_{3}*_{N}\mathcal{E}_{3}=\mathcal{E}_{4}*_{M}\mathcal{B}_{3},\\
&\mathcal{R}_{\mathcal{A}_{4}}*_{N}\mathcal{G}*_{M}\mathcal{L}_{\mathcal{B}_{6}}=0,\
\mathcal{R}_{\mathcal{A}_{6}}*_{N}\mathcal{G}*_{M}\mathcal{L}_{\mathcal{B}_{4}}=0,\
\mathcal{R}_{\mathcal{S}_{1}}*_{N}\mathcal{R}_{\mathcal{A}_{4}}*_{N}\mathcal{G}=0,\\
&\mathcal{G}*_{M}\mathcal{L}_{\mathcal{B}_{4}}*_{M}\mathcal{L}_{\mathcal{N}_{1}}=0,\
\mathcal{R}_{\mathcal{A}_{5}}*_{N}\mathcal{F}*_{M}\mathcal{L}_{\mathcal{B}_{7}}=0,\
\mathcal{R}_{\mathcal{A}_{7}}*_{N}\mathcal{F}*_{M}\mathcal{L}_{\mathcal{B}_{5}}=0,\\
&\mathcal{R}_{\mathcal{S}_{2}}*_{N}\mathcal{R}_{\mathcal{A}_{5}}*_{N}\mathcal{F}=0,\
\mathcal{F}*_{M}\mathcal{L}_{\mathcal{B}_{5}}*_{M}\mathcal{L}_{\mathcal{N}_{2}}=0,\
\mathcal{R}_{\mathcal{M}}*_{N}\mathcal{R}_{\mathcal{A}}*_{N}\mathcal{E}=0,\\
&\mathcal{R}_{\mathcal{A}}*_{N}\mathcal{E}*_{M}\mathcal{L}_{\mathcal{D}}=0,\
\mathcal{E}*_{M}\mathcal{L}_{\mathcal{B}}*_{M}\mathcal{L}_{\mathcal{N}}=0,\
\mathcal{R}_{\mathcal{C}}*_{N}\mathcal{E}*_{M}\mathcal{L}_{\mathcal{B}}=0,\\
&\mathcal{R}_{\mathcal{A}_{1}}*_{N}\mathcal{E}_{1}=0,\ \mathcal{E}_{1}*_{M}\mathcal{L}_{\mathcal{B}_{1}}=0,\
\mathcal{R}_{\mathcal{A}_{2}}*_{N}\mathcal{E}_{2}=0,\ \mathcal{E}_{2}*_{M}\mathcal{L}_{\mathcal{B}_{2}}=0,\\
&\mathcal{R}_{\mathcal{A}_{33}}*_{N}\mathcal{E}_{33}=0,\ \mathcal{E}_{33}*_{M}\mathcal{L}_{\mathcal{B}_{33}}=0,\
\mathcal{R}{\mathcal{A}_{55}}*_{N}\mathcal{E}_{55}=0,\ \mathcal{E}_{55}*_{M}\mathcal{L}_{\mathcal{B}_{55}}=0,\\
&\mathcal{R}_{\mathcal{A}_{77}}*_{N}\mathcal{E}_{77}=0,\ \mathcal{E}_{77}*_{M}\mathcal{L}_{\mathcal{B}_{77}}=0,\
\mathcal{R}_{\mathcal{A}_{99}}*_{N}\mathcal{E}_{99}=0,\ \mathcal{E}_{99}*_{M}\mathcal{L}_{\mathcal{B}_{99}}=0,\\
&\mathcal{R}_{\mathcal{\widetilde{A}}}*_{N}\mathcal{\widetilde{E}}*_{M}\mathcal{L}_{\widetilde{B}}=0.
\end{align*}
\end{small}
Under these circumstances, the general solution to \eqref{1.4} can be expressed as follows:
\begin{small}
\begin{align*}
&\mathcal{X}=\mathcal{A}_{1}^{\dagger}*_{N}\mathcal{E}_{1}*_{M}\mathcal{B}_{1}^{\dagger}
+\mathcal{L}_{\mathcal{A}_{1}}*_{N}\mathcal{U}_{1}+\mathcal{U}_{2}*_{M}\mathcal{R}_{\mathcal{B}_{1}},\\
&\mathcal{Y}=\mathcal{A}_{2}^{\dagger}*_{N}\mathcal{E}_{2}*_{M}\mathcal{B}_{2}^{\dagger}
+\mathcal{L}_{\mathcal{A}_{2}}*_{N}\mathcal{U}_{3}+\mathcal{U}_{4}*_{M}\mathcal{R}_{\mathcal{B}_{2}},\\
&\mathcal{Z}=\mathcal{A}_{3}^{\dagger}*_{N}\mathcal{E}_{3}+\mathcal{L}_{\mathcal{A}_{3}}*_{N}\mathcal{E}_{4}*_{M}\mathcal{B}_{3}^{\dagger}
+\mathcal{L}_{\mathcal{A}_{3}}*_{N}\mathcal{W}*_{M}\mathcal{R}_{\mathcal{B}_{3}},
\end{align*}
\end{small}
where
\begin{small}
\begin{subequations}
\begin{align*}
&\begin{array}{l}
 \mathcal{U}_{1}=\begin{pmatrix}\mathcal{I}\ \  0\end{pmatrix}*_{N}(\mathcal{A}_{22}^{\dagger}*_{N}
(\mathcal{E}_{22}-\mathcal{C}_{22}*_{N}\mathcal{V}_{2}*_{M}\mathcal{D}_{22})
-\mathcal{A}_{22}^{\dagger}*_{N}\mathcal{H}_{12}*_{M}\mathcal{B}_{22}\\
 \ \ \ \ \ \ +\mathcal{L}_{\mathcal{A}_{22}}*_{N}\mathcal{H}_{11}),
 \end{array}\\
&\begin{array}{l}
 \mathcal{U}_{2}=(\mathcal{R}_{\mathcal{A}_{22}}*_{N}(\mathcal{E}_{22}-\mathcal{C}_{22}*_{N}\mathcal{V}_{2}*_{M}\mathcal{D}_{22})
*_{M}\mathcal{B}_{22}^{\dagger}+\mathcal{A}_{22}*_{N}\mathcal{A}_{22}^{\dagger}*_{N}\mathcal{H}_{12}\\
\ \ \ \ \ \  \mathcal{H}_{13}*_{M}\mathcal{R}_{\mathcal{B}_{22}})*_{M}\begin{pmatrix}\mathcal{I}\\ 0\end{pmatrix},
 \end{array}
 \end{align*}
\begin{align*}
 &\begin{array}{l}
 \mathcal{U}_{3}=\begin{pmatrix}\mathcal{I}& 0\end{pmatrix}*_{N}(\mathcal{A}_{44}^{\dagger}*_{N}
(\mathcal{E}_{44}-\mathcal{C}_{44}*_{N}\mathcal{T}_{2}*_{M}\mathcal{D}_{44})
-\mathcal{A}_{44}^{\dagger}*_{N}\mathcal{H}_{22}*_{M}\mathcal{B}_{44}\\
\ \ \ \ \ \ +\mathcal{L}_{\mathcal{A}_{44}}*_{N}\mathcal{H}_{21}),
 \end{array}\\
&\begin{array}{l}
 \mathcal{U}_{4}=(\mathcal{R}_{\mathcal{A}_{44}}*_{N}(\mathcal{E}_{44}-\mathcal{C}_{44}*_{N}\mathcal{T}_{2}*_{M}\mathcal{D}_{44})
*_{M}\mathcal{B}_{44}^{\dagger}+\mathcal{A}_{44}*_{N}\mathcal{A}_{44}^{\dagger}*_{N}\mathcal{H}_{22}\\
\ \ \ \ \ \  \mathcal{H}_{23}*_{M}\mathcal{R}_{\mathcal{B}_{44}})*_{M}\begin{pmatrix}\mathcal{I}\\ 0\end{pmatrix},
 \end{array}\\
&\begin{array}{l}
 \mathcal{W}=\mathcal{M}_{1}^{\dagger}*_{N}\mathcal{G}*_{M}\mathcal{B}_{6}^{\dagger}+\mathcal{S}_{1}^{\dagger}*_{N}\mathcal{S}_{1}*_{N}\mathcal{A}_{6}^{\dagger}
*_{N}\mathcal{G}*_{M}\mathcal{N}_{1}^{\dagger}-\mathcal{L}_{\mathcal{M}_{1}}*_{N}\mathcal{L}_{\mathcal{S}_{1}}\\
\ \ \ \ \ \ *_{N}\mathcal{V}_{1}+\mathcal{L}_{\mathcal{M}_{1}}*_{N}\mathcal{V}_{2}*_{M}\mathcal{R}_{\mathcal{N}_{1}}
+\mathcal{V}_{3}*_{M}\mathcal{R}_{\mathcal{B}_{6}},
 \end{array}\\
&\begin{array}{l}
 \mathcal{V}_{1}=\begin{pmatrix}\mathcal{I}& 0\end{pmatrix}*_{N}(\mathcal{A}_{11}^{\dagger}*_{N}
(\mathcal{E}_{11}-\mathcal{L}_{\mathcal{M}_{1}}*_{N}\mathcal{\nu}_{2}*_{M}\mathcal{R}_{\mathcal{N}_{1}}-
\mathcal{L}_{\mathcal{M}_{2}}*_{N}\mathcal{T}_{2}*_{M}\mathcal{R}_{\mathcal{N}_{2}})\\
\ \ \ \ \ \ +\mathcal{W}_{1}*_{M}\mathcal{B}_{11}
+\mathcal{L}_{\mathcal{A}_{11}}*_{N}\mathcal{W}_{2}),
 \end{array}\\
 &\begin{array}{l}
\mathcal{V}_{2}=\mathcal{A}^{\dagger}*_{N}\mathcal{E}*_{M}\mathcal{B}^{\dagger}
-\mathcal{A}^{\dagger}*_{N}\mathcal{S}*_{N}\mathcal{C}^{\dagger}*_{N}\mathcal{E}*_{M}\mathcal{N}^{\dagger}*_{M}\mathcal{D}*_{M}\mathcal{B}^{\dagger}
-\mathcal{A}^{\dagger}*_{N}\\
\ \ \ \ \ \  \mathcal{C}*_{N}\mathcal{M}^{\dagger}*_{N}\mathcal{E}*_{M}\mathcal{B}^{\dagger}
 +\mathcal{A}^{\dagger}*_{N}\mathcal{S}*_{N}\mathcal{W}_{4}*_{M}\mathcal{R}_{\mathcal{N}}*_{M}\mathcal{D}*_{M}\mathcal{B}^{\dagger}\\
\ \ \ \ \ \ +\mathcal{L}_{\mathcal{A}}*_{N}\mathcal{W}_{5}+\mathcal{W}_{6}*_{M}\mathcal{R}_{\mathcal{B}},
 \end{array}\\
 &\begin{array}{l}
 \mathcal{V}_{3}=(\mathcal{R}_{\mathcal{A}_{11}}*_{N}(\mathcal{E}_{11}-\mathcal{L}_{\mathcal{M}_{1}}*_{N}\mathcal{\nu}_{2}*_{M}\mathcal{R}_{\mathcal{N}_{1}}
-\mathcal{L}_{\mathcal{M}_{2}}*_{N}\mathcal{T}_{2}*_{M}\mathcal{R}_{\mathcal{N}_{2}})
\\
\ \ \ \ \ \ *_{M}\mathcal{B}_{11}^{\dagger}-\mathcal{A}_{11}*_{N}\mathcal{W}_{1}-\mathcal{W}_{3}*_{M}\mathcal{R}_{\mathcal{B}_{11}})*_{M}\begin{pmatrix}\mathcal{I}\\ 0\end{pmatrix},
 \end{array}\\
 &\begin{array}{l}
\mathcal{T}_{2}=\mathcal{M}^{\dagger}*_{N}\mathcal{E}*_{M}\mathcal{D}^{\dagger}+\mathcal{S}^{\dagger}*_{N}\mathcal{S}*_{N}\mathcal{C}^{\dagger}
*_{N}\mathcal{E}*_{M}\mathcal{N}^{\dagger}+\mathcal{L}_{\mathcal{M}}*_{N}\mathcal{L}_{\mathcal{S}}*_{N}\mathcal{W}_{7}\\
\ \ \ \ \ \ +\mathcal{L}_{\mathcal{M}}*_{N}\mathcal{W}_{4}*_{M}\mathcal{R}_{\mathcal{N}}
+\mathcal{W}_{8}*_{M}\mathcal{R}_{\mathcal{D}},
 \end{array}\\
&\mathcal{W}_{4}=\mathcal{A}_{77}^{\dagger}*_{N}\mathcal{E}_{77}*_{M}\mathcal{B}_{77}^{\dagger}
-\mathcal{L}_{\mathcal{A}_{77}}*_{N}\mathcal{Q}_{1}-\mathcal{Q}_{2}*_{M}\mathcal{R}_{\mathcal{B}_{77}},\\
&\begin{array}{l}
\mathcal{W}_{5}=\begin{pmatrix}\mathcal{I}& 0\end{pmatrix}*_{N}(\mathcal{A}_{66}^{\dagger}*_{N}
(\mathcal{E}_{66}-\mathcal{C}_{66}*_{N}\mathcal{W}_{4}*_{M}\mathcal{D}_{66})
-\mathcal{A}_{66}^{\dagger}*_{N}\mathcal{H}_{32}*_{M}\mathcal{B}_{66}\\
\ \ \ \ \ \ +\mathcal{L}_{\mathcal{A}_{66}}*_{N}\mathcal{H}_{31}),
 \end{array}\\
&\begin{array}{l}
\mathcal{W}_{6}=(\mathcal{R}_{\mathcal{A}_{66}}*_{N}(\mathcal{E}_{66}-\mathcal{C}_{66}*_{N}\mathcal{W}_{4}*_{M}\mathcal{D}_{66})
*_{M}\mathcal{B}_{66}^{\dagger}+\mathcal{A}_{66}*_{N}\mathcal{A}_{66}^{\dagger}*_{N}\mathcal{H}_{32}\\
\ \ \ \ \ \ \mathcal{H}_{33}*_{M}\mathcal{R}_{\mathcal{B}_{66}})*_{M}\begin{pmatrix}\mathcal{I}\\ 0\end{pmatrix},
 \end{array}\\
&\begin{array}{l}
\mathcal{W}_{7}=\begin{pmatrix}\mathcal{I}& 0\end{pmatrix}*_{N}(\mathcal{A}_{88}^{\dagger}*_{N}
(\mathcal{E}_{88}-\mathcal{C}_{88}*_{N}\mathcal{W}_{4}*_{M}\mathcal{D}_{88})
-\mathcal{A}_{88}^{\dagger}*_{N}\mathcal{H}_{42}\\
\ \ \ \ \ \ *_{M}\mathcal{B}_{88}+\mathcal{L}_{\mathcal{A}_{88}}*_{N}\mathcal{H}_{41}),
 \end{array}\\
&\begin{array}{l}
\mathcal{W}_{8}=(\mathcal{R}_{\mathcal{A}_{88}}*_{N}(\mathcal{E}_{88}-\mathcal{C}_{88}*_{N}\mathcal{W}_{4}*_{M}\mathcal{D}_{88})
*_{M}\mathcal{B}_{88}^{\dagger}+\mathcal{A}_{88}*_{N}\mathcal{A}_{88}^{\dagger}*_{N}\mathcal{H}_{42}\\
\ \ \ \ \ \ \mathcal{H}_{43}*_{M}\mathcal{R}_{\mathcal{B}_{88}})*_{M}\begin{pmatrix}\mathcal{I}\\ 0\end{pmatrix},
 \end{array}\\
&\mathcal{Q}_{1}=\begin{pmatrix}\mathcal{I}& 0\end{pmatrix}*_{N}(\mathcal{\widetilde{A}}^{\dagger}*_{N}
\mathcal{\widetilde{E}}-\mathcal{\widetilde{A}}^{\dagger}*_{N}\mathcal{K}_{2}*_{M}\mathcal{\widetilde{B}}+\mathcal{L}_{\mathcal{A}}*_{N}\mathcal{K}_{1}),\\
&\mathcal{Q}_{2}=(\mathcal{R}_{\mathcal{\widetilde{A}}}*_{N}\mathcal{\widetilde{E}}*_{M}\mathcal{B}^{\dagger}
+\mathcal{\widetilde{A}}*_{N}\mathcal{\widetilde{A}}^{\dagger}*_{N}\mathcal{K}_{2}+\mathcal{K}_{3}*_{M}\mathcal{R}_{\mathcal{\widetilde{B}}})
*_{M}\begin{pmatrix}\mathcal{I}\\ 0\end{pmatrix},
\end{align*}
\end{subequations}
\end{small}
\end{lemma}
where $\mathcal{W}_{i}$, $\mathcal{K}_{i}$\, $\mathcal{H}_{jk}$ $(i,k=\overline{1,3},\ j=\overline{1,4})$  are arbitrary with suitable orders.

\section{\textbf{Algebraic solvability  conditions and  general solution to $\mathbf{(1.5)}$}}
In the following Proposition, we provide a proper extension of the  tensor equation \eqref{1.1}, which plays an essential role in the proof-findings process. Precisely, we derive the solvability conditions and the general solution to  \eqref{3.1}
\begin{proposition}\label{lma 2.3b} Let $\mathcal{A}_{1} \in \mathbb{H}^{I(N)\times J(N)}$,\ $\mathcal{A}_{2} \in \mathbb{H}^{I(N)\times G(N)}$,\ $\mathcal{B}_{1} \in \mathbb{H}^{K(M)\times L(M)}$,\ $\mathcal{B}_{2} \in \mathbb{H}^{H(M)\times L(M)}$,\ $\mathcal{C}_{3} \in \mathbb{H}^{G(N)\times Q(N)}$,\ $\mathcal{C}_{4} \in \mathbb{H}^{G(N)\times T(N)}$,\ $\mathcal{D}_{3} \in \mathbb{H}^{S(M)\times K(M)}$,\ $\mathcal{D}_{4} \in \mathbb{H}^{P(M)\times K(M)}$\ and $\mathcal{E}_{1} \in \mathbb{H}^{I(N)\times L(M)}$ be given. Set
\begin{small}
\begin{align}
\label{system 2221}
&\mathcal{M}_{1}=\mathcal{R}_{\mathcal{A}_{1}}*_{N}\mathcal{A}_{2},\ \mathcal{N}_{1}=\mathcal{B}_{2}*_{M}\mathcal{L}_{\mathcal{B}_{1}},\ \mathcal{S}_{1}=\mathcal{A}_{2}*_{N}\mathcal{L}_{\mathcal{M}_{1}},\
\mathcal{\widehat{A}}_{1}=\mathcal{M}_{1}*_{N}\mathcal{C}_{3},\\
\label{system 2222}
&\mathcal{\widehat{A}}_{2}=\mathcal{M}_{1}*_{N}\mathcal{C}_{4},\
\mathcal{\widehat{B}}_{1}=\mathcal{D}_{3}*_{M}\mathcal{B}_{1}*_{M}\mathcal{L}_{\mathcal{B}_{2}},\
\mathcal{\widehat{B}}_{2}=\mathcal{D}_{4}*_{M}\mathcal{B}_{1}*_{M}\mathcal{L}_{\mathcal{B}_{2}},\\
\label{system 2223}
&\mathcal{\widehat{M}}_{1}=\mathcal{R}_{\mathcal{\widehat{A}}_{1}}*_{N}\mathcal{\widehat{A}}_{2},\ \mathcal{\widehat{N}}_{1}=\mathcal{\widehat{B}}_{2}*_{M}\mathcal{L}_{\mathcal{\widehat{B}}_{1}},\  \mathcal{\widehat{S}}_{1}=\mathcal{\widehat{A}}_{2}*_{N}\mathcal{L}_{\mathcal{\widehat{M}}_{1}},\
\mathcal{\widehat{E}}_{1}=\mathcal{R}_{\mathcal{A}_{1}}*_{N}\mathcal{E}_{1}\\
\label{system 22241}
&*_{M}\mathcal{L}_{\mathcal{B}_{2}},\ \mathcal{\grave{E}}_{1}=\mathcal{E}_{1}-\mathcal{A}_{2}*_{N}(\mathcal{C}_{3}*_{N}\mathcal{X}_{3}*_{M}\mathcal{D}_{3}
+\mathcal{C}_{4}*_{N}\mathcal{W}*_{M}\mathcal{D}_{4})*_{M}\mathcal{B}_{1}.
\end{align}
\end{small}
Then the following statements are equivalent:\\
(1) \eqref{3.1} is solvable.\\
(2) The conditions
\begin{small}
\begin{align*}
\begin{gathered}
\mathcal{R}_{\mathcal{M}_{1}}*_{N}\mathcal{R}_{\mathcal{A}_{1}}*_{N}\mathcal{E}_{1}=0,\ \mathcal{E}_{1}*_{M}\mathcal{L}_{\mathcal{B}_{1}}*_{M}\mathcal{L}_{\mathcal{N}_{1}}=0,\
\mathcal{R}_{\mathcal{A}_{2}}*_{N}\mathcal{E}_{1}*_{M}\mathcal{L}_{\mathcal{B}_{1}}=0\
\end{gathered}
\end{align*}
\end{small}
 are satisfying  and there exist quaternion tensors $\mathcal{X}_{3}$ and $\mathcal{W}$ satisfy
\begin{small}
\begin{align*}
\mathcal{\widehat{A}}_{1}*_{N}\mathcal{X}_{3}*_{M}\mathcal{\widehat{B}}_{1}
+\mathcal{\widehat{A}}_{2}*_{N}\mathcal{W}*_{M}\mathcal{\widehat{B}}_{2}=\mathcal{\widehat{E}}_{1}.
\end{align*}
\end{small}
(3)
\begin{small}
\begin{align*}
\begin{gathered}
\mathcal{R}_{\mathcal{M}_{1}}*_{N}\mathcal{R}_{\mathcal{A}_{1}}*_{N}\mathcal{E}_{1}=0,\ \mathcal{E}_{1}*_{M}\mathcal{L}_{\mathcal{B}_{1}}*_{M}\mathcal{L}_{\mathcal{N}_{1}}=0,\
\mathcal{R}_{\mathcal{A}_{2}}*_{N}\mathcal{E}_{1}*_{M}\mathcal{L}_{\mathcal{B}_{1}}=0,\\
\mathcal{R}_{\mathcal{\widehat{M}}_{1}}*_{N}\mathcal{R}_{\mathcal{\widehat{A}}_{1}}*_{N}\mathcal{\widehat{E}}_{1}=0,\ \mathcal{\widehat{E}}_{1}*_{M}\mathcal{L}_{\mathcal{\widehat{B}}_{1}}*_{M}\mathcal{L}_{\mathcal{\widehat{N}}_{1}}=0,\\ \mathcal{R}_{\mathcal{\widehat{A}}_{1}}*_{N}\mathcal{\widehat{E}}_{1}*_{M}\mathcal{L}_{\mathcal{\widehat{B}}_{2}}=0,\
\mathcal{R}_{\mathcal{\widehat{A}}_{2}}*_{N}\mathcal{\widehat{E}}_{1}*_{M}\mathcal{L}_{\mathcal{\widehat{B}}_{1}}=0.
\end{gathered}
\end{align*}
\end{small}
In that case, the general solution to  \eqref{3.1} can be expressed as follows:
\begin{small}
\begin{subequations}
\begin{align*}
&\begin{array}{l}
 \mathcal{X}_{1}=\mathcal{A}_{1}^{\dagger}*_{N}\mathcal{\grave{E}}_{1}*_{M}\mathcal{B}_{1}^{\dagger}
-\mathcal{A}_{1}^{\dagger}*_{N}\mathcal{A}_{2}*_{N}\mathcal{M}_{1}^{\dagger}*_{N}\mathcal{\grave{E}}_{1}*_{M}\mathcal{B}_{1}^{\dagger}
-\mathcal{A}_{1}^{\dagger}*_{N}\mathcal{S}_{1}*_{N}\mathcal{A}_{2}^{\dagger}\\
\ \ \ \ \ *_{N}\mathcal{\grave{E}}_{1}*_{M}\mathcal{N}_{1}^{\dagger}*_{M}\mathcal{B}_{2}*_{M}\mathcal{B}_{1}^{\dagger}-\mathcal{A}_{1}^{\dagger}*_{N}\mathcal{S}_{1}
*_{N}\mathcal{U}_{2}*_{M}\mathcal{R}_{\mathcal{N}_{1}}*_{M}\mathcal{B}_{2}*_{M}\mathcal{B}_{1}^{\dagger}
\\
\ \ \ \ \ +\mathcal{L}_{\mathcal{A}_{1}}*_{N}\mathcal{U}_{4}+\mathcal{U}_{5}*_{M}\mathcal{R}_{\mathcal{B}_{1}},
 \end{array}  \\
&\begin{array}{l}
 \mathcal{X}_{2}=\mathcal{M}_{1}^{\dagger}*_{N}\mathcal{\grave{E}}_{1}*_{M}\mathcal{B}_{2}^{\dagger}
 +\mathcal{S}_{1}^{\dagger}*_{N}\mathcal{S}_{1}*_{N}\mathcal{A}_{2}^{\dagger}
*_{N}\mathcal{\grave{E}}*_{M}\mathcal{N}_{1}^{\dagger}+\mathcal{L}_{\mathcal{M}_{1}}*_{N}\mathcal{L}_{\mathcal{S}_{1}}\\
\ \ \ \ \ \ *_{N}\mathcal{U}_{1}+\mathcal{L}_{\mathcal{M}_{1}}*_{N}\mathcal{U}_{2}*_{M}\mathcal{R}_{\mathcal{N}_{1}}+\mathcal{U}_{3}*_{M}\mathcal{R}_{\mathcal{B}_{2}},
 \end{array}\\
 &\begin{array}{l}
 \mathcal{X}_{3}=\mathcal{\widehat{A}}_{1}^{\dagger}*_{N}\mathcal{\widehat{E}}_{1}*_{M}\mathcal{\widehat{B}}_{1}^{\dagger}
-\mathcal{\widehat{A}}_{1}^{\dagger}*_{N}\mathcal{\widehat{A}}_{2}*_{N}\mathcal{\widehat{M}}_{1}^{\dagger}*_{N}\mathcal{\widehat{E}}_{1}
*_{M}\mathcal{\widehat{B}}_{1}^{\dagger}
-\mathcal{\widehat{A}}_{1}^{\dagger}*_{N}\mathcal{\widehat{S}}_{1}*_{N}\mathcal{\widehat{A}}_{2}^{\dagger}\\
\ \ \ \ \*_{N}\mathcal{\widehat{E}}_{1} *_{M}\mathcal{\widehat{N}}_{1}^{\dagger}*_{M}\mathcal{\widehat{B}}_{2}*_{M}\mathcal{\widehat{B}}_{1}^{\dagger}-\mathcal{\widehat{A}}_{1}^{\dagger}
*_{N}\mathcal{\widehat{S}}_{1}
*_{N}\mathcal{\widehat{U}}_{2}*_{M}\mathcal{R}_{\mathcal{\widehat{N}}_{1}}*_{M}\mathcal{\widehat{B}}_{2}*_{M}\mathcal{\widehat{B}}_{1}^{\dagger}
\\
\ \ \ \ \ +\mathcal{L}_{\mathcal{\widehat{A}}_{1}}*_{N}\mathcal{\widehat{U}}_{4}+\mathcal{\widehat{U}}_{5}*_{M}\mathcal{R}_{\mathcal{\widehat{B}}_{1}},
 \end{array}\\
&\begin{array}{l}
 \mathcal{W}=\mathcal{\widehat{M}}_{1}^{\dagger}*_{N}\mathcal{\widehat{E}}_{1}*_{M}\mathcal{\widehat{B}}_{2}^{\dagger}
 +\mathcal{\widehat{S}}_{1}^{\dagger}*_{N}\mathcal{\widehat{S}}_{1}*_{N}\mathcal{\widehat{A}}_{2}^{\dagger}
*_{N}\mathcal{\widehat{E}}*_{M}\mathcal{\widehat{N}}_{1}^{\dagger}+\mathcal{L}_{\mathcal{\widehat{M}}_{1}}
*_{N}\mathcal{L}_{\mathcal{\widehat{S}}_{1}}\\
\ \ \ \ \ \ *_{N}\mathcal{\widehat{U}}_{1}+\mathcal{L}_{\mathcal{\widehat{M}}_{1}}*_{N}\mathcal{\widehat{U}}_{2}*_{M}\mathcal{R}_{\mathcal{\widehat{N}}_{1}}
+\mathcal{\widehat{U}}_{3}*_{M}\mathcal{R}_{\mathcal{\widehat{B}}_{2}},
 \end{array}
\end{align*}
\end{subequations}
\end{small}
where $\mathcal{U}_{i},$ $\mathcal{\widehat{U}}_{i}$ $(i=\overline{1,5})$ are arbitrary  tensors with suitable orders.
\end{proposition}
\begin{proof}$(1)\Leftrightarrow(2)$\ \ \  We, first, rewrite  the  tensor equation \eqref{3.1}  in the form
\begin{small}
\begin{align}
\label{3.81}
\begin{array}{rll}
&\mathcal{A}_{1}*_{N}\mathcal{X}_{1}*_{M}\mathcal{B}_{1}+\mathcal{A}_{2}*_{N}\mathcal{X}_{2}*_{M}\mathcal{B}_{2}=\mathcal{\grave{E}}_{1},
\end{array}
\end{align}
\end{small}
where $\mathcal{\grave{E}}_{1}$ gives by \eqref{system 22241}. By utilizing $Lemma$ $\ref{lma 2.3}$, we have that  \eqref{3.81} is solvable if and only if there exist quaternion tensors $\mathcal{X}_{3}$ and $\mathcal{W}$ satisfy the following conditions:
\begin{small}
\begin{align}
\label{3.91}
&\mathcal{R}_{\mathcal{M}_{1}}*_{N}\mathcal{R}_{\mathcal{A}_{1}}*_{N}\mathcal{\grave{E}}_{1}=0,\ \mathcal{\grave{E}}_{1}*_{M}\mathcal{L}_{\mathcal{B}_{1}}*_{M}\mathcal{L}_{\mathcal{N}_{1}}=0,\
\mathcal{R}_{\mathcal{A}_{2}}*_{N}\mathcal{\grave{E}}_{1}*_{M}\mathcal{L}_{\mathcal{B}_{1}}=0,
\end{align}
\begin{align}
\label{3.92}
&\mathcal{R}_{\mathcal{A}_{1}}*_{N}\mathcal{\grave{E}}_{1}*_{M}\mathcal{L}_{\mathcal{B}_{2}}=0.
\end{align}
\end{small}
In that case, the general solution to \eqref{3.81} can be expressed as
\begin{small}
\begin{subequations}
\begin{align*}
&\begin{array}{l}
 \mathcal{X}_{1}=\mathcal{A}_{1}^{\dagger}*_{N}\mathcal{\grave{E}}_{1}*_{M}\mathcal{B}_{1}^{\dagger}
-\mathcal{A}_{1}^{\dagger}*_{N}\mathcal{A}_{2}*_{N}\mathcal{M}_{1}^{\dagger}*_{N}\mathcal{\grave{E}}_{1}*_{M}\mathcal{B}_{1}^{\dagger}
-\mathcal{A}_{1}^{\dagger}*_{N}\mathcal{S}_{1}*_{N}\mathcal{A}_{2}^{\dagger}\\
\ \ \ \ \ *_{N}\mathcal{\grave{E}}_{1}*_{M}\mathcal{N}_{1}^{\dagger}*_{M}\mathcal{B}_{2}*_{M}\mathcal{B}_{1}^{\dagger}-\mathcal{A}_{1}^{\dagger}*_{N}\mathcal{S}_{1}
*_{N}\mathcal{U}_{2}*_{M}\mathcal{R}_{\mathcal{N}_{1}}*_{M}\mathcal{B}_{2}*_{M}\mathcal{B}_{1}^{\dagger}
\\
\ \ \ \ \ +\mathcal{L}_{\mathcal{A}_{1}}*_{N}\mathcal{U}_{4}+\mathcal{U}_{5}*_{M}\mathcal{R}_{\mathcal{B}_{1}},
 \end{array}\\
&\begin{array}{l}
 \mathcal{X}_{2}=\mathcal{M}_{1}^{\dagger}*_{N}\mathcal{\grave{E}}_{1}*_{M}\mathcal{B}_{2}^{\dagger}
 +\mathcal{S}_{1}^{\dagger}*_{N}\mathcal{S}_{1}*_{N}\mathcal{A}_{2}^{\dagger}
*_{N}\mathcal{\grave{E}}*_{M}\mathcal{N}_{1}^{\dagger}+\mathcal{L}_{\mathcal{M}_{1}}*_{N}\mathcal{L}_{\mathcal{S}_{1}}\\
\ \ \ \ \ \ *_{N}\mathcal{U}_{1}+\mathcal{L}_{\mathcal{M}_{1}}*_{N}\mathcal{U}_{2}*_{M}\mathcal{R}_{\mathcal{N}_{1}}+\mathcal{U}_{3}*_{M}\mathcal{R}_{\mathcal{B}_{2}}.
 \end{array}
\end{align*}
\end{subequations}
\end{small}
The conditions \eqref{3.91} are satisfied if and only if the  conditions
\begin{small}
\begin{align}
\label{3.93}
&\mathcal{R}_{\mathcal{M}_{1}}*_{N}\mathcal{R}_{\mathcal{A}_{1}}*_{N}\mathcal{E}_{1}=0,\ \mathcal{E}_{1}*_{M}\mathcal{L}_{\mathcal{B}_{1}}*_{M}\mathcal{L}_{\mathcal{N}_{1}}=0,\
\mathcal{R}_{\mathcal{A}_{2}}*_{N}\mathcal{E}_{1}*_{M}\mathcal{L}_{\mathcal{B}_{1}}=0
\end{align}
\end{small}
are satisfied, respectively. It is evident that the condition \eqref{3.92} satisfies if and only if there exist quaternion tensors $\mathcal{X}_{3}$ and $\mathcal{W}$ satisfy
\begin{small}
\begin{align}
\label{3.94}
\mathcal{\widehat{A}}_{1}*_{N}\mathcal{X}_{3}*_{M}\mathcal{\widehat{B}}_{1}
+\mathcal{\widehat{A}}_{2}*_{N}\mathcal{W}*_{M}\mathcal{\widehat{B}}_{2}=\mathcal{\widehat{E}}_{1}.
\end{align}
\end{small}
$(2)\Leftrightarrow(3)$\ \ \ By applying  $Lemma$ $\ref{lma 2.3}$, we have that   \eqref{3.94}  is solvable if and only if
\begin{small}
\begin{align}
\begin{gathered}
\mathcal{R}_{\mathcal{\widehat{M}}_{1}}*_{N}\mathcal{R}_{\mathcal{\widehat{A}}_{1}}*_{N}\mathcal{\widehat{E}}_{1}=0,\ \mathcal{\widehat{E}}_{1}*_{M}\mathcal{L}_{\mathcal{\widehat{B}}_{1}}*_{M}\mathcal{L}_{\mathcal{\widehat{N}}_{1}}=0,\\ \mathcal{R}_{\mathcal{\widehat{A}}_{1}}*_{N}\mathcal{\widehat{E}}_{1}*_{M}\mathcal{L}_{\mathcal{\widehat{B}}_{2}}=0,\
\mathcal{R}_{\mathcal{\widehat{A}}_{2}}*_{N}\mathcal{\widehat{E}}_{1}*_{M}\mathcal{L}_{\mathcal{\widehat{B}}_{1}}=0.
\end{gathered}
\end{align}
\end{small}
In that case, the general solution to \eqref{3.94} can be expressed as
\begin{small}
\begin{subequations}
\begin{align}
&\begin{array}{l}
 \mathcal{X}_{3}=\mathcal{\widehat{A}}_{1}^{\dagger}*_{N}\mathcal{\widehat{E}}_{1}*_{M}\mathcal{\widehat{B}}_{1}^{\dagger}
-\mathcal{\widehat{A}}_{1}^{\dagger}*_{N}\mathcal{\widehat{A}}_{2}*_{N}\mathcal{\widehat{M}}_{1}^{\dagger}*_{N}\mathcal{\widehat{E}}_{1}
*_{M}\mathcal{\widehat{B}}_{1}^{\dagger}
-\mathcal{\widehat{A}}_{1}^{\dagger}*_{N}\mathcal{\widehat{S}}_{1}*_{N}\mathcal{\widehat{A}}_{2}^{\dagger}\\
\ \ \ \ \ *_{N}\mathcal{\widehat{E}}_{1}*_{M}\mathcal{\widehat{N}}_{1}^{\dagger}*_{M}\mathcal{\widehat{B}}_{2}*_{M}\mathcal{\widehat{B}}_{1}^{\dagger}-\mathcal{\widehat{A}}_{1}^{\dagger}
*_{N}\mathcal{\widehat{S}}_{1}
*_{N}\mathcal{\widehat{U}}_{2}*_{M}\mathcal{R}_{\mathcal{\widehat{N}}_{1}}*_{M}\mathcal{\widehat{B}}_{2}*_{M}\mathcal{\widehat{B}}_{1}^{\dagger}
\\
\ \ \ \ \ +\mathcal{L}_{\mathcal{\widehat{A}}_{1}}*_{N}\mathcal{\widehat{U}}_{4}+\mathcal{\widehat{U}}_{5}*_{M}\mathcal{R}_{\mathcal{\widehat{B}}_{1}},
 \end{array}  \\
&\begin{array}{l}
 \mathcal{W}=\mathcal{\widehat{M}}_{1}^{\dagger}*_{N}\mathcal{\widehat{E}}_{1}*_{M}\mathcal{\widehat{B}}_{2}^{\dagger}
 +\mathcal{\widehat{S}}_{1}^{\dagger}*_{N}\mathcal{\widehat{S}}_{1}*_{N}\mathcal{\widehat{A}}_{2}^{\dagger}
*_{N}\mathcal{\widehat{E}}*_{M}\mathcal{\widehat{N}}_{1}^{\dagger}+\mathcal{L}_{\mathcal{\widehat{M}}_{1}}
*_{N}\mathcal{L}_{\mathcal{\widehat{S}}_{1}}\\
\ \ \ \ \ \ *_{N}\mathcal{\widehat{U}}_{1}+\mathcal{L}_{\mathcal{\widehat{M}}_{1}}*_{N}\mathcal{\widehat{U}}_{2}*_{M}\mathcal{R}_{\mathcal{\widehat{N}}_{1}}
+\mathcal{\widehat{U}}_{3}*_{M}\mathcal{R}_{\mathcal{\widehat{B}}_{2}},
 \end{array}
\end{align}
\end{subequations}
\end{small}
where $\mathcal{U}_{i}$ $\mathcal{\widehat{U}}_{i},$ $(i=\overline{1,5})$ are arbitrary tensors with suitable orders.
\end{proof}
\begin{corollary}\label{saaad1}
Set $\mathcal{A}_{2}=\mathcal{B}_{1}=I,$ in  \eqref{3.1}, we can obtain the solvability conditions and the general solution to the following tensor equation:
\begin{small}
\begin{align*}
\mathcal{A}_{1}*_{N}\mathcal{X}_{1}*_{M}+\mathcal{X}_{2}*_{M}\mathcal{B}_{2}
+\mathcal{C}_{3}*_{N}\mathcal{X}_{3}*_{M}\mathcal{D}_{3}
+\mathcal{C}_{4}*_{N}\mathcal{W}*_{M}\mathcal{D}_{4}=\mathcal{E}_{1}.
\end{align*}
\end{small}
\end{corollary}
\begin{theorem}\label{system 3.33AAt}
Consider the quaternion system of tensor equations  \eqref{1.5}, where
\begin{small}
\begin{align*}
&\mathcal{A}_{1} \in \mathbb{H}^{I(N)\times J(N)},\ \mathcal{A}_{2} \in \mathbb{H}^{I(N)\times Q(N)},\ \mathcal{A}_{3} \in \mathbb{H}^{I(N)\times P(N)},\
\mathcal{A}_{4} \in \mathbb{H}^{A(N)\times E(N)},\\
&\mathcal{A}_{5} \in \mathbb{H}^{C(N)\times V(N)},\ \mathcal{A}_{6} \in \mathbb{H}^{A(N)\times E(N)},\ \mathcal{A}_{7} \in \mathbb{H}^{A(N)\times G(N)},\ \mathcal{A}_{8} \in \mathbb{H}^{C(N)\times V(N)},\\
&\mathcal{A}_{9} \in \mathbb{H}^{C(N)\times G(N)},\
 \mathcal{B}_{1} \in \mathbb{H}^{L(M)\times K(M)},\ \mathcal{B}_{2} \in \mathbb{H}^{S(M)\times K(M)},\ \mathcal{B}_{3} \in \mathbb{H}^{T(M)\times K(M)},\\
&\mathcal{B}_{4} \in \mathbb{H}^{H(M)\times B(M)},\ \mathcal{B}_{5} \in \mathbb{H}^{U(M)\times D(M)},\ \mathcal{B}_{6} \in \mathbb{H}^{F(M)\times B(M)},\ \mathcal{B}_{7} \in \mathbb{H}^{H(M)\times B(M)},\\
&\mathcal{B}_{8} \in \mathbb{H}^{D(M)\times S(M)},\ \mathcal{B}_{9} \in \mathbb{H}^{U(M)\times R(M)},\
\mathcal{C}_{3} \in \mathbb{H}^{G(N)\times J(N)},\ \mathcal{C}_{4} \in \mathbb{H}^{G(N)\times P(N)},\\
&\mathcal{D}_{3} \in \mathbb{H}^{L(M)\times F(M)},\ \mathcal{D}_{4} \in \mathbb{H}^{T(M)\times F(M)},\
\mathcal{H}_{3} \in \mathbb{H}^{G(N)\times Q(N)},\ \mathcal{H}_{4} \in \mathbb{H}^{C(N)\times P(N)},\\
&\mathcal{J}_{3} \in \mathbb{H}^{S(M)\times D(M)},\ \mathcal{J}_{4} \in \mathbb{H}^{T(M)\times D(M)},\
\mathcal{E}_{1} \in \mathbb{H}^{I(N)\times K(M)},\ \mathcal{E}_{2} \in \mathbb{H}^{I(N)\times K(M)},\\
&\mathcal{E}_{3} \in \mathbb{H}^{I(N)\times T(M)},\ \mathcal{E}_{4} \in \mathbb{H}^{P(N)\times K(M)},\
\mathcal{E}_{5} \in \mathbb{H}^{A(N)\times F(M)},\ \mathcal{E}_{6} \in \mathbb{H}^{G(N)\times B(M)},\\
&\mathcal{E}_{7} \in \mathbb{H}^{C(N)\times D(M)},\ \mathcal{E}_{8} \in \mathbb{H}^{C(N)\times D(M)},\
\mathcal{E}_{9} \in \mathbb{H}^{A(N)\times B(M)},\ \mathcal{E}_{10} \in \mathbb{H}^{C(N)\times R(M)}
\end{align*}
\end{small}
 are given tensors over $\mathbb{H}$. Set
\begin{small}
\begin{subequations}
\begin{align}
\label{system 3.33A}
&\mathcal{\widehat{A}}_{6}=\mathcal{A}_{6}*_{N}\mathcal{L}_{\mathcal{A}_{4}},\ \mathcal{\widehat{B}}_{7}=\mathcal{R}_{\mathcal{B}_{4}}*_{M}\mathcal{B}_{7},\\
\label{system 3.33B}
&\mathcal{\widehat{E}}_{9}=\mathcal{E}_{9}-\mathcal{A}_{6}*_{N}\mathcal{A}_{4}^{\dagger}*_{N}\mathcal{E}_{5}*_{M}\mathcal{B}_{6}-\mathcal{A}_{7}*_{N}\mathcal{E}_{6}
*_{M}\mathcal{B}_{4}^{\dagger}*_{M}\mathcal{B}_{7},\\
\label{system 3.33C}
&\mathcal{M}_{11}=\mathcal{R}_{\mathcal{\widehat{A}}_{6}}*_{N}\mathcal{A}_{7},\ \mathcal{N}_{11}=\mathcal{\widehat{B}}_{7}*_{M}\mathcal{L}_{\mathcal{B}_{6}},\ \mathcal{S}_{11}=\mathcal{A}_{7}*_{N}\mathcal{L}_{\mathcal{M}_{11}},\\
\label{system 3.33D}
&\mathcal{\grave{E}}_{1}=\mathcal{\widehat{E}}_{9}-\mathcal{A}_{7}*_{N}(\mathcal{C}_{3}*_{N}\mathcal{X}_{3}*_{M}\mathcal{D}_{3}
+\mathcal{C}_{4}*_{N}\mathcal{W}*_{M}\mathcal{D}_{4})*_{M}\mathcal{B}_{6},\\
\label{system 3.33EE}
&\mathcal{\widehat{A}}_{4}=\mathcal{M}_{11}*_{N}\mathcal{C}_{3},\ \mathcal{\widehat{C}}_{4}=\mathcal{M}_{11}*_{N}\mathcal{C}_{4},\
\mathcal{\widehat{B}}_{4}=\mathcal{D}_{3}*_{M}\mathcal{B}_{6}*_{M}\mathcal{L}_{\mathcal{\widehat{B}}_{7}},\\
\label{system 3.33F}
&\mathcal{\widehat{D}}_{4}=\mathcal{D}_{4}*_{M}\mathcal{B}_{6}*_{M}\mathcal{L}_{\mathcal{\widehat{B}}_{7}},\
\mathcal{\widehat{P}}=\mathcal{R}_{\mathcal{\widehat{A}}_{6}}*_{N}\mathcal{\widehat{E}}_{9}*_{M}\mathcal{L}_{\mathcal{\widehat{B}}_{7}},\\
\label{system 3.33G}
&\mathcal{\widehat{A}}_{8}=\mathcal{A}_{8}*_{N}\mathcal{L}_{\mathcal{A}_{5}},\ \mathcal{\widehat{B}}_{9}=\mathcal{R}_{\mathcal{B}_{5}}*_{M}\mathcal{B}_{9},\\
\label{system 3.33H}
&\mathcal{\widehat{E}}_{10}=\mathcal{E}_{10}-\mathcal{A}_{8}*_{N}\mathcal{A}_{5}^{\dagger}*_{N}\mathcal{E}_{7}*_{M}\mathcal{B}_{8}-\mathcal{A}_{9}*_{N}\mathcal{E}_{8}
*_{M}\mathcal{B}_{5}^{\dagger}*_{M}\mathcal{B}_{9},\\
\label{system 3.33I}
&\mathcal{M}_{22}=\mathcal{R}_{\mathcal{\widehat{A}}_{8}}*_{N}\mathcal{A}_{9},\ \mathcal{N}_{22}=\mathcal{\widehat{B}}_{9}*_{M}\mathcal{L}_{\mathcal{B}_{8}},\ \mathcal{S}_{22}=\mathcal{A}_{9}*_{N}\mathcal{L}_{\mathcal{M}_{22}},\\
\label{system 3.33J}
&\mathcal{\grave{E}}_{2}=\mathcal{\widehat{E}}_{10}-\mathcal{A}_{9}*_{N}(\mathcal{H}_{3}*_{N}\mathcal{Y}_{3}*_{M}\mathcal{J}_{3}
+\mathcal{H}_{4}*_{N}\mathcal{W}*_{M}\mathcal{J}_{4})*_{M}\mathcal{B}_{8},\\
\label{system 3.33KKK}
&\mathcal{\widehat{A}}_{5}=\mathcal{M}_{22}*_{N}\mathcal{H}_{3},\ \mathcal{\widehat{C}}_{5}=\mathcal{M}_{22}*_{N}\mathcal{H}_{4},\
\mathcal{\widehat{B}}_{5}=\mathcal{J}_{3}*_{M}\mathcal{B}_{8}*_{M}\mathcal{L}_{\mathcal{\widehat{B}}_{9}},\\
\label{system 3.33L}
&\mathcal{\widehat{D}}_{5}=\mathcal{J}_{4}*_{M}\mathcal{B}_{8}*_{M}\mathcal{L}_{\mathcal{\widehat{B}}_{9}},\
\mathcal{\widehat{Q}}=\mathcal{R}_{\mathcal{\widehat{A}}_{8}}*_{N}\mathcal{\widehat{E}}_{10}*_{M}\mathcal{L}_{\mathcal{\widehat{B}}_{9}},\\
\label{system 3.33M}
&\mathcal{C}_{6}=\mathcal{\widehat{C}}_{4}*_{N}\mathcal{L}_{\mathcal{A}_{3}},\ \mathcal{D}_{6}=\mathcal{R}_{\mathcal{B}_{3}}*_{M}\mathcal{\widehat{D}}_{4},\
\mathcal{C}_{7}=\mathcal{\widehat{C}}_{5}*_{N}\mathcal{L}_{\mathcal{A}_{3}},\ \mathcal{D}_{7}=\mathcal{R}_{\mathcal{B}_{3}}*_{M}\mathcal{\widehat{D}}_{5},\\
\label{system 3.33N}
&\mathcal{G}=\mathcal{\widehat{P}}-\mathcal{\widehat{C}}_{4}*_{N}\mathcal{A}_{3}^{\dagger}*_{N}\mathcal{E}_{3}*_{M}\mathcal{\widehat{D}}_{4}
-\mathcal{\widehat{C}}_{4}*_{N}\mathcal{L}_{\mathcal{A}_{3}}*_{N}\mathcal{E}_{4}*_{M}\mathcal{B}_{3}^{\dagger}*_{M}\mathcal{\widehat{D}}_{4},\\
\label{system 3.33O}
&\mathcal{F}=\mathcal{\widehat{Q}}-\mathcal{\widehat{C}}_{5}*_{N}\mathcal{A}_{3}^{\dagger}*_{N}\mathcal{E}_{3}*_{M}\mathcal{\widehat{D}}_{5}
-\mathcal{\widehat{C}}_{5}*_{N}\mathcal{L}_{\mathcal{A}_{3}}*_{N}\mathcal{E}_{4}*_{M}\mathcal{B}_{3}^{\dagger}*_{M}\mathcal{\widehat{D}}_{5},\\
\label{system 3.33P}
&\mathcal{M}_{1}=\mathcal{R}_{\mathcal{\widehat{A}}_{4}}*_{M}\mathcal{C}_{6},\ \mathcal{N}_{1}=\mathcal{D}_{6}*_{M}\mathcal{L}_{\mathcal{\widehat{B}}_{4}},\
\mathcal{S}_{1}=\mathcal{C}_{6}*_{N}\mathcal{L}_{\mathcal{M}_{1}},\\
\label{system 3.33Q}
&\mathcal{M}_{2}=\mathcal{R}_{\mathcal{\widehat{A}}_{5}}*_{M}\mathcal{C}_{7},\ \mathcal{N}_{2}=\mathcal{D}_{7}*_{M}\mathcal{L}_{\mathcal{\widehat{B}}_{5}},\
\mathcal{S}_{2}=\mathcal{C}_{7}*_{N}\mathcal{L}_{\mathcal{M}_{2}},\\
\label{system 3.33R}
&\mathcal{A}_{11}=\begin{pmatrix}\mathcal{L}_{\mathcal{M}_{1}}*_{N}\mathcal{L}_{\mathcal{S}_{1}}& \mathcal{L}_{\mathcal{M}_{2}}*_{N}\mathcal{L}_{\mathcal{S}_{2}}\end{pmatrix},\
\mathcal{B}_{11}=\begin{pmatrix}\mathcal{R}_{\mathcal{D}_{6}}\\ \mathcal{R}_{\mathcal{D}_{7}}\end{pmatrix},\\
\label{system 3.33S}
&\begin{array}{l}
\mathcal{E}_{11}=\mathcal{M}_{2}^{\dagger}*_{N}\mathcal{F}*_{M}\mathcal{D}_{7}^{\dagger}
+\mathcal{S}_{2}^{\dagger}*_{N}\mathcal{S}_{2}*_{N}\mathcal{C}_{7}^{\dagger}
*_{N}\mathcal{F}*_{M}\mathcal{N}_{2}^{\dagger}
-\mathcal{M}_{1}^{\dagger}*_{N}\mathcal{G}*_{M}\mathcal{D}_{6}^{\dagger}\\
\ \ \ \ \ \ \ -\mathcal{S}_{1}^{\dagger}*_{N}\mathcal{S}_{1}*_{N}\mathcal{C}_{6}^{\dagger}*_{N}\mathcal{G}*_{M}\mathcal{N}_{1}^{\dagger},\
\mathcal{A}=\mathcal{R}_{\mathcal{A}_{11}}*_{N}\mathcal{L}_{\mathcal{M}_{1}},
\end{array}\\
\label{system 3.33T}
&
\mathcal{B}=\mathcal{R}_{\mathcal{N}_{1}}*_{M}\mathcal{L}_{\mathcal{B}_{11}},\
\mathcal{C}=\mathcal{R}_{\mathcal{A}_{11}}*_{N}\mathcal{L}_{\mathcal{M}_{2}},\
\mathcal{D}=\mathcal{R}_{\mathcal{N}_{2}}*_{M}\mathcal{L}_{\mathcal{B}_{11}},\\
\label{system 3.33U}
&\mathcal{E}=\mathcal{R}_{\mathcal{A}_{11}}*_{N}\mathcal{E}_{11}*_{M}\mathcal{L}_{\mathcal{B}_{11}},\
\mathcal{M}=\mathcal{R}_{\mathcal{A}}*_{N}\mathcal{C},\
\mathcal{N}=\mathcal{D}*_{M}\mathcal{L}_{\mathcal{B}},\
\mathcal{S}=\mathcal{C}*_{N}\mathcal{L}_{\mathcal{M}},\\
\label{system 3.33V}
&\mathcal{A}_{22}=\begin{pmatrix}\mathcal{L}_{\mathcal{A}_{1}}& \mathcal{L}_{\mathcal{\widehat{A}}_{4}}\end{pmatrix},\
\mathcal{B}_{22}=\begin{pmatrix}\mathcal{R}_{\mathcal{B}_{1}}\\ \mathcal{R}_{\mathcal{\widehat{B}}_{4}}\end{pmatrix},\
\mathcal{C}_{22}=\mathcal{\widehat{A}}_{4}^{\dagger}*_{N}\mathcal{S}_{1},\\
\label{system 3.33aa}
&\begin{array}{l}
\mathcal{D}_{22}=\mathcal{R}_{\mathcal{N}_{1}}*_{M}\mathcal{D}_{6}*_{N}\mathcal{\widehat{B}}_{4}^{\dagger},\
\mathcal{E}_{22}=\mathcal{\widehat{A}}_{4}^{\dagger}*_{N}\mathcal{G}*_{M}\mathcal{\widehat{B}}_{4}^{\dagger}
-\mathcal{A}_{1}^{\dagger}*_{N}\mathcal{E}_{1}*_{M}\mathcal{B}_{1}^{\dagger}
-\mathcal{\widehat{A}}_{4}^{\dagger}\\
 \ \ *_{N}\mathcal{S}_{1}*_{N}\mathcal{C}_{6}^{\dagger}*_{N}\mathcal{G}*_{M}\mathcal{N}_{1}*_{M}\mathcal{D}_{6}*_{M}\mathcal{\widehat{B}}_{4}^{\dagger}-\mathcal{\widehat{A}}_{4}^{\dagger}*_{N}\mathcal{C}_{6}*_{N}\mathcal{M}_{1}^{\dagger}
*_{N}\mathcal{G}*_{M}\mathcal{\widehat{B}}_{4}^{\dagger},
\end{array}\\
\label{system 3.33bb}
&\mathcal{A}_{33}=\mathcal{R}_{\mathcal{A}_{22}}*_{N}\mathcal{C}_{22},\
\mathcal{B}_{33}=\mathcal{D}_{22}*_{M}\mathcal{L}_{\mathcal{B}_{22}},\
\mathcal{E}_{33}=\mathcal{R}_{\mathcal{A}_{22}}*_{N}\mathcal{E}_{22}*_{M}\mathcal{L}_{\mathcal{B}_{22}},\\
\label{system 3.33cc}
&\mathcal{A}_{44}=\begin{pmatrix}\mathcal{L}_{\mathcal{A}_{2}}& \mathcal{L}_{\mathcal{\widehat{A}}_{5}}\end{pmatrix},\
\mathcal{B}_{44}=\begin{pmatrix}\mathcal{R}_{\mathcal{B}_{2}}\\ \mathcal{R}_{\mathcal{\widehat{B}}_{5}}\end{pmatrix},\
\mathcal{C}_{44}=\mathcal{\widehat{A}}_{5}^{\dagger}*_{N}\mathcal{S}_{2},\\
&\label{system 3.33dd}
\begin{array}{l}
\mathcal{D}_{44}=\mathcal{R}_{\mathcal{N}_{2}}*_{M}\mathcal{D}_{7}*_{N}\mathcal{\widehat{B}}_{5}^{\dagger},\ \mathcal{E}_{44}=\mathcal{\widehat{A}}_{5}^{\dagger}*_{N}\mathcal{F}*_{M}\mathcal{\widehat{B}}_{5}^{\dagger}
-\mathcal{A}_{2}^{\dagger}*_{N}\mathcal{E}_{2}*_{M}\mathcal{B}_{2}^{\dagger} -\mathcal{\widehat{A}}_{5}^{\dagger}\\
*_{N}\mathcal{S}_{2}*_{N}\mathcal{C}_{7}^{\dagger}*_{N}\mathcal{F}*_{M}\mathcal{N}_{2}*_{M}\mathcal{D}_{7}*_{M}\mathcal{\widehat{B}}_{5}^{\dagger}-\mathcal{\widehat{A}}_{5}^{\dagger}*_{N}\mathcal{C}_{7}*_{N}\mathcal{M}_{2}^{\dagger}*_{N}\mathcal{F}
*_{M}\mathcal{\widehat{B}}_{5}^{\dagger},
\end{array}
\end{align}
\end{subequations}
\vspace*{-\baselineskip}
\begin{subequations}
\begin{align}
\label{system 3.33ee}
&\mathcal{A}_{55}=\mathcal{R}_{\mathcal{A}_{44}}*_{N}\mathcal{C}_{44},\
\mathcal{B}_{55}=\mathcal{D}_{44}*_{M}\mathcal{L}_{\mathcal{B}_{44}},\
\mathcal{E}_{55}=\mathcal{R}_{\mathcal{A}_{44}}*_{N}\mathcal{E}_{44}*_{M}\mathcal{L}_{\mathcal{B}_{44}},\\
\label{system 3.33eee}
&\mathcal{A}_{66}=\begin{pmatrix}\mathcal{L}_{\mathcal{A}}& \mathcal{L}_{\mathcal{A}_{33}}\end{pmatrix},\
\mathcal{B}_{66}=\begin{pmatrix}\mathcal{R}_{\mathcal{B}}\\ \mathcal{R}_{\mathcal{B}_{33}}\end{pmatrix},\
\mathcal{C}_{66}=\mathcal{A}^{\dagger}*_{N}\mathcal{S},
\end{align}
\begin{align}
\label{system 3.33ff}
&\begin{array}{l}
\mathcal{D}_{66}=\mathcal{R}_{\mathcal{N}}*_{M}\mathcal{D}*_{M}\mathcal{B}^{\dagger},\ \mathcal{E}_{66}=\mathcal{A}_{33}^{\dagger}*_{N}\mathcal{E}_{33}*_{M}\mathcal{B}_{33}^{\dagger}
-\mathcal{A}^{\dagger}*_{N}\mathcal{E}*_{M}\mathcal{B}^{\dagger} +\mathcal{A}^{\dagger}\\
\ \ \ \*_{N}\mathcal{S}*_{N}\mathcal{C}^{\dagger}*_{N}\mathcal{E}*_{M}\mathcal{N}^{\dagger}*_{M}\mathcal{D}*_{N} \mathcal{B}^{\dagger}+\mathcal{A}^{\dagger}*_{N}\mathcal{C}*_{N}\mathcal{M}^{\dagger}*_{N}\mathcal{E}
*_{M}\mathcal{B}^{\dagger},
\end{array}\\
\label{system 3.33gg}
&\mathcal{A}_{77}=\mathcal{R}_{\mathcal{A}_{66}}*_{N}\mathcal{C}_{66},\
\mathcal{B}_{77}=\mathcal{D}_{66}*_{M}\mathcal{L}_{\mathcal{B}_{66}},\
\mathcal{E}_{77}=\mathcal{R}_{\mathcal{A}_{66}}*_{N}\mathcal{E}_{66}*_{M}\mathcal{L}_{\mathcal{B}_{66}},\\
\label{system 3.33hh}
&\mathcal{A}_{88}=\begin{pmatrix}\mathcal{L}_{\mathcal{M}}*_{N}\mathcal{L}_{\mathcal{S}}& \mathcal{L}_{\mathcal{A}_{55}}\end{pmatrix},\
\mathcal{B}_{88}=\begin{pmatrix}\mathcal{R}_{\mathcal{D}}\\ \mathcal{R}_{\mathcal{B}_{55}}\end{pmatrix},\
\mathcal{C}_{88}=\mathcal{L}_{\mathcal{M}},\
\mathcal{D}_{88}=\mathcal{R}_{\mathcal{N}},\\
\label{system 3.33II}
&\mathcal{E}_{88}=\mathcal{A}_{55}^{\dagger}*_{N}\mathcal{E}_{55}*_{M}\mathcal{B}_{55}^{\dagger}
-\mathcal{M}^{\dagger}*_{N}\mathcal{E}*_{M}\mathcal{B}^{\dagger}
-\mathcal{S}^{\dagger}*_{N}\mathcal{S}*_{N}\mathcal{C}^{\dagger}*_{N}\mathcal{E}*_{M}\mathcal{N}^{\dagger},\\
\label{system 3.33JJ}
&\mathcal{A}_{99}=\mathcal{R}_{\mathcal{A}_{88}}*_{N}\mathcal{C}_{88},\
\mathcal{B}_{99}=\mathcal{D}_{88}*_{M}\mathcal{L}_{\mathcal{B}_{88}},\
\mathcal{E}_{99}=\mathcal{R}_{\mathcal{A}_{88}}*_{N}\mathcal{E}_{88}*_{M}\mathcal{L}_{\mathcal{B}_{88}},\\
\label{system 3.33KK}
&\mathcal{\widetilde{A}}=\begin{pmatrix}\mathcal{L}_{\mathcal{A}_{77}}& -\mathcal{L}_{\mathcal{A}_{99}}\end{pmatrix},\
\mathcal{\widetilde{B}}=\begin{pmatrix}\mathcal{R}_{\mathcal{B}_{77}}\\ -\mathcal{R}_{\mathcal{B}_{99}}\end{pmatrix},\\
&\mathcal{\widetilde{E}}=\mathcal{A}_{77}^{\dagger}*_{N}\mathcal{E}_{77}*_{M}\mathcal{B}_{77}^{\dagger}
-\mathcal{A}_{99}^{\dagger}*_{N}\mathcal{E}_{99}*_{M}\mathcal{B}_{99}^{\dagger},\
\end{align}
\end{subequations}
\end{small}
then the system \eqref{1.5} is solvable  if and only if
\begin{small}
\begin{align}
\label{system 222a}
&\mathcal{R}_{\mathcal{A}_{4}}*_{N}\mathcal{E}_{5}=0,\ \mathcal{E}_{6}*_{M}\mathcal{L}_{\mathcal{B}_{4}}=0,\
\mathcal{R}_{\mathcal{A}_{5}}*_{N}\mathcal{E}_{7}=0,\ \mathcal{E}_{8}*_{M}\mathcal{L}_{\mathcal{B}_{5}}=0,\\
\label{system 222b}
&\mathcal{R}_{\mathcal{M}_{11}}*_{N}\mathcal{R}_{\mathcal{\widehat{A}}_{6}}*_{N}\mathcal{\widehat{E}}_{9}=0,\ \mathcal{\widehat{E}}_{9}*_{M}\mathcal{L}_{\mathcal{B}_{6}}*_{M}\mathcal{L}_{\mathcal{N}_{11}}=0,\
 \mathcal{R}_{\mathcal{A}_{7}}*_{N}\mathcal{\widehat{E}}_{9},\ *_{M}\mathcal{L}_{\mathcal{B}_{6}}\\
\label{system 222c}
&=0,\ \mathcal{R}_{\mathcal{M}_{22}}*_{N}\mathcal{R}_{\mathcal{\widehat{A}}_{8}}*_{N}\mathcal{\widehat{E}}_{10}=0,\ \mathcal{\widehat{E}}_{10}*_{M}\mathcal{L}_{\mathcal{B}_{8}}*_{M}\mathcal{L}_{\mathcal{N}_{22}}=0,\
 \mathcal{R}_{\mathcal{A}_{9}}*_{N}\mathcal{\widehat{E}}_{10}\\
\label{system 222ddd}
&*_{M}\mathcal{L}_{\mathcal{B}_{8}}=0,\ \mathcal{R}_{\mathcal{A}_{3}}*_{N}\mathcal{E}_{3}=0,\ \mathcal{E}_{4}*_{M}\mathcal{L}_{\mathcal{B}_{3}}=0,\
\mathcal{A}_{3}*_{N}\mathcal{E}_{3}=\mathcal{E}_{4}*_{M}\mathcal{B}_{3},\\
&\mathcal{R}_{\mathcal{\widehat{A}}_{4}}*_{N}\mathcal{G}*_{M}\mathcal{L}_{\mathcal{D}_{6}}=0,\
\mathcal{R}_{\mathcal{C}_{6}}*_{N}\mathcal{G}*_{M}\mathcal{L}_{\mathcal{\widehat{B}}_{4}}=0,\
\mathcal{R}_{\mathcal{S}_{1}}*_{N}\mathcal{R}_{\mathcal{\widehat{A}}_{4}}*_{N}\mathcal{G}=0,\\
&\mathcal{G}*_{M}\mathcal{L}_{\mathcal{\widehat{B}}_{4}}*_{M}\mathcal{L}_{\mathcal{N}_{1}}=0,\
\mathcal{R}_{\mathcal{\widehat{A}}_{5}}*_{N}\mathcal{F}*_{M}\mathcal{L}_{\mathcal{D}_{7}}=0,\
\mathcal{R}_{\mathcal{C}_{7}}*_{N}\mathcal{F}*_{M}\mathcal{L}_{\mathcal{\widehat{B}}_{5}}=0,\\
&\mathcal{R}_{\mathcal{S}_{2}}*_{N}\mathcal{R}_{\mathcal{\widehat{A}}_{5}}*_{N}\mathcal{F}=0,\
\mathcal{F}*_{M}\mathcal{L}_{\mathcal{\widehat{B}}_{5}}*_{M}\mathcal{L}_{\mathcal{N}_{2}}=0,\
\mathcal{R}_{\mathcal{M}}*_{N}\mathcal{R}_{\mathcal{A}}*_{N}\mathcal{E}=0,\\
&\mathcal{R}_{\mathcal{A}}*_{N}\mathcal{E}*_{M}\mathcal{L}_{\mathcal{D}}=0,\
\mathcal{E}*_{M}\mathcal{L}_{\mathcal{B}}*_{M}\mathcal{L}_{\mathcal{N}}=0,\
\mathcal{R}_{\mathcal{C}}*_{N}\mathcal{E}*_{M}\mathcal{L}_{\mathcal{B}}=0,\\
&\mathcal{R}_{\mathcal{A}_{1}}*_{N}\mathcal{E}_{1}=0,\ \mathcal{E}_{1}*_{M}\mathcal{L}_{\mathcal{B}_{1}}=0,\
\mathcal{R}{\mathcal{A}_{2}}*_{N}\mathcal{E}_{2}=0,\ \mathcal{E}_{2}*_{M}\mathcal{B}_{2}=0,\\
&\mathcal{R}_{\mathcal{A}_{33}}*_{N}\mathcal{E}_{33}=0,\ \mathcal{E}_{33}*_{M}\mathcal{L}_{\mathcal{B}_{33}}=0,\
\mathcal{R}{\mathcal{A}_{55}}*_{N}\mathcal{E}_{55}=0,\ \mathcal{E}_{55}*_{M}\mathcal{L}_{\mathcal{B}_{55}}=0,\\
&\mathcal{R}_{\mathcal{A}_{77}}*_{N}\mathcal{E}_{77}=0,\ \mathcal{E}_{77}*_{M}\mathcal{L}_{\mathcal{B}_{77}}=0,\
\mathcal{R}_{\mathcal{A}_{99}}*_{N}\mathcal{E}_{99}=0,\ \mathcal{E}_{99}*_{M}\mathcal{L}_{\mathcal{B}_{99}}=0,\\
\label{system 222eee}
&\mathcal{R}_{\mathcal{\widetilde{A}}}*_{N}\mathcal{\widetilde{E}}*_{M}\mathcal{L}_{\widetilde{B}}=0.
\end{align}
\end{small}
Under these conditions, the general solution to  \eqref{1.5} can be expressed as:
\begin{small}
\begin{align}
\label{system 2225}
&\mathcal{X}_{1}=\mathcal{A}_{4}^{\dagger}*_{N}\mathcal{E}_{5}+\mathcal{L}_{\mathcal{A}_{4}}*_{N}\mathcal{V}_{11},\\
\label{system 2226}
&\mathcal{X}_{2}=\mathcal{E}_{6}*_{M}\mathcal{B}_{4}^{\dagger}+\mathcal{V}_{22}*_{M}\mathcal{R}_{\mathcal{B}_{4}},\\
\label{system 2227}
&\mathcal{Y}_{1}=\mathcal{A}_{5}^{\dagger}*_{N}\mathcal{E}_{7}+\mathcal{L}_{\mathcal{A}_{5}}*_{N}\mathcal{V}_{33},\\
\label{system 2228}
& \mathcal{Y}_{2}=\mathcal{E}_{8}*_{M}\mathcal{B}_{5}^{\dagger}+\mathcal{V}_{44}*_{M}\mathcal{R}_{\mathcal{B}_{5}},\\
\label{system 2229}
&\mathcal{X}_{3}=\mathcal{A}_{1}^{\dagger}*_{N}\mathcal{E}_{1}*_{M}\mathcal{B}_{1}^{\dagger}
+\mathcal{L}_{\mathcal{A}_{1}}*_{N}\mathcal{U}_{1}+\mathcal{U}_{2}*_{M}\mathcal{R}_{\mathcal{B}_{1}},\\
\label{system 2230}
&\mathcal{Y}_{3}=\mathcal{A}_{2}^{\dagger}*_{N}\mathcal{E}_{2}*_{M}\mathcal{B}_{2}^{\dagger}
+\mathcal{L}_{\mathcal{A}_{2}}*_{N}\mathcal{U}_{3}+\mathcal{U}_{4}*_{M}\mathcal{R}_{\mathcal{B}_{2}},\\
\label{system 2231}
&\mathcal{W}=\mathcal{A}_{3}^{\dagger}*_{N}\mathcal{E}_{3}+\mathcal{L}_{\mathcal{A}_{3}}*_{N}\mathcal{E}_{4}*_{M}\mathcal{B}_{3}^{\dagger}
+\mathcal{L}_{\mathcal{A}_{3}}*_{N}\mathcal{U}_{5}*_{M}\mathcal{R}_{\mathcal{B}_{3}},
\end{align}
\end{small}
where
\begin{small}
\begin{subequations}
\begin{align}
\label{system 22a}
&\begin{array}{l}
 \mathcal{V}_{11}=\mathcal{\widehat{A}}_{6}^{\dagger}*_{N}\mathcal{\grave{E}}_{1}*_{M}\mathcal{B}_{6}^{\dagger}
-\mathcal{\widehat{A}}_{6}^{\dagger}*_{N}\mathcal{A}_{7}*_{N}\mathcal{M}_{11}^{\dagger}*_{N}\mathcal{\grave{E}}_{1}*_{M}\mathcal{B}_{6}^{\dagger}
-\mathcal{\widehat{A}}_{6}^{\dagger}*_{N}\mathcal{S}_{11}*_{N}\\
\ \ \ \ \ \mathcal{A}_{7}^{\dagger}*_{N}\mathcal{\grave{E}}_{1}*_{M}\mathcal{N}_{11}^{\dagger}*_{M}\mathcal{\widehat{B}}_{7}*_{M}\mathcal{B}_{6}^{\dagger}-\mathcal{\widehat{A}}_{6}^{\dagger}*_{N}\mathcal{S}_{11}
*_{N}\mathcal{T}_{21}*_{M}\mathcal{R}_{\mathcal{N}_{11}}*_{M}\mathcal{\widehat{B}}_{7}\\
\ \ \ \ \ *_{M}\mathcal{B}_{6}^{\dagger}+\mathcal{L}_{\mathcal{\widehat{A}}_{6}}*_{N}\mathcal{T}_{41}+\mathcal{T}_{51}*_{M}\mathcal{R}_{\mathcal{B}_{6}},
 \end{array}
 \end{align}
\begin{align}
\label{system 22b}
&\begin{array}{l}
 \mathcal{V}_{22}=\mathcal{M}_{11}^{\dagger}*_{N}\mathcal{\grave{E}}_{1}*_{M}\mathcal{\widehat{B}}_{7}^{\dagger}
 +\mathcal{S}_{11}^{\dagger}*_{N}\mathcal{S}_{11}*_{N}\mathcal{A}_{7}^{\dagger}
*_{N}\mathcal{\grave{E}}_{1}*_{M}\mathcal{N}_{11}^{\dagger}+\mathcal{L}_{\mathcal{M}_{11}}*_{N}\\
\ \ \ \ \ \ \mathcal{L}_{\mathcal{S}_{11}}*_{N}\mathcal{T}_{11}+\mathcal{L}_{\mathcal{M}_{11}}*_{N}\mathcal{T}_{21}*_{M}\mathcal{R}_{\mathcal{N}_{11}}+\mathcal{T}_{31}*_{M}\mathcal{R}_{\mathcal{\widehat{B}}_{7}},
 \end{array}\\
&\begin{array}{l}
\label{system 22c}
 \mathcal{V}_{33}=\mathcal{\widehat{A}}_{8}^{\dagger}*_{N}\mathcal{\grave{E}}_{2}*_{M}\mathcal{B}_{8}^{\dagger}
-\mathcal{\widehat{A}}_{8}^{\dagger}*_{N}\mathcal{A}_{9}*_{N}\mathcal{M}_{22}^{\dagger}*_{N}\mathcal{\grave{E}}_{2}*_{M}\mathcal{B}_{8}^{\dagger}
-\mathcal{\widehat{A}}_{8}^{\dagger}*_{N}\mathcal{S}_{22}\\
\ \ \ \ \ *_{N}\mathcal{A}_{9}^{\dagger}*_{N}\mathcal{\grave{E}}_{2}*_{M}\mathcal{N}_{22}^{\dagger}*_{M}\mathcal{\widehat{B}}_{9}*_{M}\mathcal{B}_{8}^{\dagger}-\mathcal{\widehat{A}}_{8}^{\dagger}*_{N}\mathcal{S}_{22}
*_{N}\mathcal{J}_{21}*_{M}\mathcal{R}_{\mathcal{N}_{22}}*_{M}\\
\ \ \ \ \ \mathcal{\widehat{B}}_{9}*_{M}\mathcal{B}_{8}^{\dagger}
+\mathcal{L}_{\mathcal{\widehat{A}}_{8}}*_{N}\mathcal{J}_{41}+\mathcal{J}_{51}*_{M}\mathcal{R}_{\mathcal{B}_{8}},
 \end{array}\\
\label{system 22d}
&\begin{array}{l}
 \mathcal{V}_{44}=\mathcal{M}_{22}^{\dagger}*_{N}\mathcal{\grave{E}}_{2}*_{M}\mathcal{\widehat{B}}_{9}^{\dagger}
 +\mathcal{S}_{22}^{\dagger}*_{N}\mathcal{S}_{22}*_{N}\mathcal{A}_{9}^{\dagger}
*_{N}\mathcal{\grave{E}}_{2}*_{M}\mathcal{N}_{22}^{\dagger}+\mathcal{L}_{\mathcal{M}_{22}}*_{N}\\
\ \ \ \ \ \ \mathcal{L}_{\mathcal{S}_{22}}*_{N}\mathcal{J}_{11}+\mathcal{L}_{\mathcal{M}_{22}}*_{N}\mathcal{J}_{21}*_{M}\mathcal{R}_{\mathcal{N}_{22}}+\mathcal{J}_{31}*_{M}\mathcal{R}_{\mathcal{\widehat{B}}_{9}},
 \end{array}\\
\label{system 2222eeee}
&\begin{array}{l}
 \mathcal{U}_{1}=\begin{pmatrix}\mathcal{I}\ \  0\end{pmatrix}*_{N}(\mathcal{A}_{22}^{\dagger}*_{N}
(\mathcal{E}_{22}-\mathcal{C}_{22}*_{N}\mathcal{V}_{2}*_{M}\mathcal{D}_{22})
-\mathcal{A}_{22}^{\dagger}*_{N}\mathcal{H}_{12}*_{M}\mathcal{B}_{22}\\
 \ \ \ \ \ \ +\mathcal{L}_{\mathcal{A}_{22}}*_{N}\mathcal{H}_{11}),
 \end{array}  \\
\label{system 22f}
&\begin{array}{l}
 \mathcal{U}_{2}=(\mathcal{R}_{\mathcal{A}_{22}}*_{N}(\mathcal{E}_{22}-\mathcal{C}_{22}*_{N}\mathcal{V}_{2}*_{M}\mathcal{D}_{22})
*_{M}\mathcal{B}_{22}^{\dagger}+\mathcal{A}_{22}*_{N}\mathcal{A}_{22}^{\dagger}*_{N}\mathcal{H}_{12}\\
\ \ \ \ \ \  \mathcal{H}_{13}*_{M}\mathcal{R}_{\mathcal{B}_{22}})*_{M}\begin{pmatrix}\mathcal{I}\\ 0\end{pmatrix},
 \end{array}\\
 \label{system 22g}
 &\begin{array}{l}
 \mathcal{U}_{3}=\begin{pmatrix}\mathcal{I}& 0\end{pmatrix}*_{N}(\mathcal{A}_{44}^{\dagger}*_{N}
(\mathcal{E}_{44}-\mathcal{C}_{44}*_{N}\mathcal{T}_{2}*_{M}\mathcal{D}_{44})
-\mathcal{A}_{44}^{\dagger}*_{N}\mathcal{H}_{22}*_{M}\mathcal{B}_{44}\\
\ \ \ \ \ \ +\mathcal{L}_{\mathcal{A}_{44}}*_{N}\mathcal{H}_{21}),
 \end{array}\\
 \label{system 22h}
&\begin{array}{l}
 \mathcal{U}_{4}=(\mathcal{R}_{\mathcal{A}_{44}}*_{N}(\mathcal{E}_{44}-\mathcal{C}_{44}*_{N}\mathcal{T}_{2}*_{M}\mathcal{D}_{44})
*_{M}\mathcal{B}_{44}^{\dagger}+\mathcal{A}_{44}*_{N}\mathcal{A}_{44}^{\dagger}*_{N}\mathcal{H}_{22}\\
\ \ \ \ \ \  \mathcal{H}_{23}*_{M}\mathcal{R}_{\mathcal{B}_{44}})*_{M}\begin{pmatrix}\mathcal{I}\\ 0\end{pmatrix},
 \end{array}\\
 \label{system 2222iiii}
&\begin{array}{l}
 \mathcal{U}_{5}=\mathcal{M}_{1}^{\dagger}*_{N}\mathcal{G}*_{M}\mathcal{D}_{6}^{\dagger}+\mathcal{S}_{1}^{\dagger}*_{N}\mathcal{S}_{1}*_{N}\mathcal{C}_{6}^{\dagger}
*_{N}\mathcal{G}*_{M}\mathcal{N}_{1}^{\dagger}-\mathcal{L}_{\mathcal{M}_{1}}*_{N}\mathcal{L}_{\mathcal{S}_{1}}\\
\ \ \ \ \ \ *_{N}\mathcal{V}_{1}+\mathcal{L}_{\mathcal{M}_{1}}*_{N}\mathcal{V}_{2}*_{M}\mathcal{R}_{\mathcal{N}_{1}}
+\mathcal{V}_{3}*_{M}\mathcal{R}_{\mathcal{D}_{6}},
 \end{array}\\
 \label{system 22j}
&\begin{array}{l}
 \mathcal{V}_{1}=\begin{pmatrix}\mathcal{I}& 0\end{pmatrix}*_{N}(\mathcal{A}_{11}^{\dagger}*_{N}
(\mathcal{E}_{11}-\mathcal{L}_{\mathcal{M}_{1}}*_{N}\mathcal{\nu}_{2}*_{M}\mathcal{R}_{\mathcal{N}_{1}}-
\mathcal{L}_{\mathcal{M}_{2}}*_{N}\mathcal{T}_{2}*_{M}\mathcal{R}_{\mathcal{N}_{2}})\\
\ \ \ \ \ \ +\mathcal{W}_{1}*_{M}\mathcal{B}_{11}
+\mathcal{L}_{\mathcal{A}_{11}}*_{N}\mathcal{W}_{2}),
 \end{array}\\
 \label{system 22k}
 &\begin{array}{l}
\mathcal{V}_{2}=\mathcal{A}^{\dagger}*_{N}\mathcal{E}*_{M}\mathcal{B}^{\dagger}
-\mathcal{A}^{\dagger}*_{N}\mathcal{S}*_{N}\mathcal{C}^{\dagger}*_{N}\mathcal{E}*_{M}\mathcal{N}^{\dagger}*_{M}\mathcal{D}*_{M}\mathcal{B}^{\dagger}
-\mathcal{A}^{\dagger}*_{N}\\
\ \ \ \ \ \ \mathcal{C}*_{N}\mathcal{M}^{\dagger}*_{N}\mathcal{E}*_{M}\mathcal{B}^{\dagger}
 +\mathcal{A}^{\dagger}*_{N}\mathcal{S}*_{N}\mathcal{W}_{4}*_{M}\mathcal{R}_{\mathcal{N}}*_{M}\mathcal{D}*_{M}\mathcal{B}^{\dagger}\\
\ \ \ \ \ \ +\mathcal{L}_{\mathcal{A}}*_{N}\mathcal{W}_{5}+\mathcal{W}_{6}*_{M}\mathcal{R}_{\mathcal{B}},
 \end{array}\\
 \label{system 22l}
&\begin{array}{l}
 \mathcal{V}_{3}=(\mathcal{R}_{\mathcal{A}_{11}}*_{N}(\mathcal{E}_{11}-\mathcal{L}_{\mathcal{M}_{1}}*_{N}\mathcal{\nu}_{2}*_{M}\mathcal{R}_{\mathcal{N}_{1}}
-\mathcal{L}_{\mathcal{M}_{2}}*_{N}\mathcal{T}_{2}*_{M}\mathcal{R}_{\mathcal{N}_{2}})
*_{M}\mathcal{B}_{11}^{\dagger}\\
\ \ \ \ \ \ -\mathcal{A}_{11}*_{N}\mathcal{W}_{1}-\mathcal{W}_{3}*_{M}\mathcal{R}_{\mathcal{B}_{11}})*_{M}\begin{pmatrix}\mathcal{I}\\ 0\end{pmatrix},
 \end{array}\\
 \label{system 22m}
 &\begin{array}{l}
\mathcal{T}_{2}=\mathcal{M}^{\dagger}*_{N}\mathcal{E}*_{M}\mathcal{D}^{\dagger}+\mathcal{S}^{\dagger}*_{N}\mathcal{S}*_{N}\mathcal{C}^{\dagger}
*_{N}\mathcal{E}*_{M}\mathcal{N}^{\dagger}+\mathcal{L}_{\mathcal{M}}*_{N}\mathcal{L}_{\mathcal{S}}*_{N}\mathcal{W}_{7}\\
\ \ \ \ \ \ +\mathcal{L}_{\mathcal{M}}*_{N}\mathcal{W}_{4}*_{M}\mathcal{R}_{\mathcal{N}}
+\mathcal{W}_{8}*_{M}\mathcal{R}_{\mathcal{D}},
 \end{array}\\
 \label{system 22n}
&\mathcal{W}_{4}=\mathcal{A}_{77}^{\dagger}*_{N}\mathcal{E}_{77}*_{M}\mathcal{B}_{77}^{\dagger}
-\mathcal{L}_{\mathcal{A}_{77}}*_{N}\mathcal{Q}_{1}-\mathcal{Q}_{2}*_{M}\mathcal{R}_{\mathcal{B}_{77}},\\
\label{system 22o}
&\begin{array}{l}
\mathcal{W}_{5}=\begin{pmatrix}\mathcal{I}& 0\end{pmatrix}*_{N}(\mathcal{A}_{66}^{\dagger}*_{N}
(\mathcal{E}_{66}-\mathcal{C}_{66}*_{N}\mathcal{W}_{4}*_{M}\mathcal{D}_{66})
-\mathcal{A}_{66}^{\dagger}*_{N}\mathcal{H}_{32}*_{M}\mathcal{B}_{66}\\
\ \ \ \ \ \ +\mathcal{L}_{\mathcal{A}_{66}}*_{N}\mathcal{H}_{31}),
 \end{array}\\
\label{system 22p}
&\begin{array}{l}
\mathcal{W}_{6}=(\mathcal{R}_{\mathcal{A}_{66}}*_{N}(\mathcal{E}_{66}-\mathcal{C}_{66}*_{N}\mathcal{W}_{4}*_{M}\mathcal{D}_{66})
*_{M}\mathcal{B}_{66}^{\dagger}+\mathcal{A}_{66}*_{N}\mathcal{A}_{66}^{\dagger}*_{N}\mathcal{H}_{32}\\
\ \ \ \ \ \ \mathcal{H}_{33}*_{M}\mathcal{R}_{\mathcal{B}_{66}})*_{M}\begin{pmatrix}\mathcal{I}\\ 0\end{pmatrix},
 \end{array}\\
 \label{system 22q}
&\begin{array}{l}
\mathcal{W}_{7}=\begin{pmatrix}\mathcal{I}& 0\end{pmatrix}*_{N}(\mathcal{A}_{88}^{\dagger}*_{N}
(\mathcal{E}_{88}-\mathcal{C}_{88}*_{N}\mathcal{W}_{4}*_{M}\mathcal{D}_{88})
-\mathcal{A}_{88}^{\dagger}*_{N}\mathcal{H}_{42}*_{M}\mathcal{B}_{88}\\
\ \ \ \ \ \ +\mathcal{L}_{\mathcal{A}_{88}}*_{N}\mathcal{H}_{41}),
 \end{array}\\
 \label{system 22r}
&\begin{array}{l}
\mathcal{W}_{8}=(\mathcal{R}_{\mathcal{A}_{88}}*_{N}(\mathcal{E}_{88}-\mathcal{C}_{88}*_{N}\mathcal{W}_{4}*_{M}\mathcal{D}_{88})
*_{M}\mathcal{B}_{88}^{\dagger}+\mathcal{A}_{88}*_{N}\mathcal{A}_{88}^{\dagger}*_{N}\mathcal{H}_{42}\\
\ \ \ \ \ \ \mathcal{H}_{43}*_{M}\mathcal{R}_{\mathcal{B}_{88}})*_{M}\begin{pmatrix}\mathcal{I}\\ 0\end{pmatrix},
 \end{array}
 \end{align}
 \begin{align}
 \label{system 22s}
&\mathcal{Q}_{1}=\begin{pmatrix}\mathcal{I}& 0\end{pmatrix}*_{N}(\mathcal{\widetilde{A}}^{\dagger}*_{N}
\mathcal{\widetilde{E}}-\mathcal{\widetilde{A}}^{\dagger}*_{N}\mathcal{K}_{2}*_{M}\mathcal{\widetilde{B}}+\mathcal{L}_{\mathcal{A}}*_{N}\mathcal{K}_{1}),\\
\label{system 22t}
&\mathcal{Q}_{2}=(\mathcal{R}_{\mathcal{\widetilde{A}}}*_{N}\mathcal{\widetilde{E}}*_{M}\mathcal{B}^{\dagger}
+\mathcal{\widetilde{A}}*_{N}\mathcal{\widetilde{A}}^{\dagger}*_{N}\mathcal{K}_{2}+\mathcal{K}_{3}*_{M}\mathcal{R}_{\mathcal{\widetilde{B}}})
*_{M}\begin{pmatrix}\mathcal{I}\\ 0\end{pmatrix},
\end{align}
\end{subequations}
\end{small}
\end{theorem}
where $\mathcal{W}_{i}$, $\mathcal{K}_{i}$\, $\mathcal{H}_{jk}$ $(i,k=\overline{1,3},\ j=\overline{1,4}),$ $\mathcal{T}_{l1}$ and  $\mathcal{J}_{l1}$  $(l=\overline{1,5})$ are arbitrary tensors over $\mathbb{H}$.
\begin{proof}
We, first, Separate the  system of tensor equations \eqref{1.5} into eight blocks:
\begin{small}
\begin{align}
\label{system 2221}
&\mathcal{A}_{4}*_{N}\mathcal{X}_{1}=\mathcal{E}_{5},\\
\label{system 2222}
&\mathcal{X}_{2}*_{M}\mathcal{B}_{4}=\mathcal{E}_{6},\\
\label{system 2223}
&\mathcal{A}_{5}*_{N}\mathcal{Y}_{1}=\mathcal{E}_{7},\\
\label{system 2224}
& \mathcal{Y}_{2}*_{M}\mathcal{B}_{5}=\mathcal{E}_{8},
\end{align}
\begin{equation}
\label{3.1113}
\mathcal{A}_{1}*_{N}\mathcal{X}_{3}*_{M}\mathcal{B}_{1}=\mathcal{E}_{1},\ \ \
\mathcal{A}_{2}*_{N}\mathcal{Y}_{3}*_{M}\mathcal{B}_{2}=\mathcal{E}_{2},
\end{equation}
\begin{equation}
\label{3.1116}
\mathcal{A}_{3}*_{N}\mathcal{W}=\mathcal{E}_{3},\ \mathcal{W}*_{M}\mathcal{B}_{3}=\mathcal{E}_{4},
\end{equation}
\begin{equation}
\label{3.111}
\begin{gathered}
\mathcal{A}_{6}*_{N}\mathcal{X}_{1}*_{M}\mathcal{B}_{6}+\mathcal{A}_{7}*_{N}\mathcal{X}_{2}*_{M}\mathcal{B}_{7}\\
+\mathcal{A}_{7}*_{N}(\mathcal{C}_{3}*_{N}\mathcal{X}_{3}*_{M}\mathcal{D}_{3}
+\mathcal{C}_{4}*_{N}\mathcal{W}*_{M}\mathcal{D}_{4})*_{M}\mathcal{B}_{6}=\mathcal{E}_{9},
\end{gathered}
\end{equation}
\end{small}
and
\begin{small}
\begin{equation}
\label{3.1112}
\begin{gathered}
\mathcal{A}_{8}*_{N}\mathcal{Y}_{1}*_{M}\mathcal{B}_{8}+\mathcal{A}_{9}*_{N}\mathcal{Y}_{2}*_{M}\mathcal{B}_{9}\\
+\mathcal{A}_{9}*_{N}(\mathcal{H}_{3}*_{N}\mathcal{Y}_{3}*_{M}\mathcal{J}_{3}
+\mathcal{H}_{4}*_{N}\mathcal{W}*_{M}\mathcal{J}_{4})*_{M}\mathcal{B}_{8}=\mathcal{E}_{10},
\end{gathered}
\end{equation}
\end{small}
Our goal is to investigate the solvability conditions that make these eight groups have a solution, and hence we investigate an expression of this solution. Applying $Lemma$ $\ref{lma 2.3}$, we have that the quaternion tensor equations \eqref{system 2221}, \eqref{system 2222}, \eqref{system 2223} and \eqref{system 2224} are solvable respectively if and only if the conditions  \eqref{system 222a} are satisfied, respectively. In that case, the general solution  expresses in the form
\begin{small}
\begin{align}
\label{system 22251}
&\mathcal{X}_{1}=\mathcal{A}_{4}^{\dagger}*_{N}\mathcal{E}_{5}+\mathcal{L}_{\mathcal{A}_{4}}*_{N}\mathcal{V}_{11},\\
\label{system 22261}
&\mathcal{X}_{2}=\mathcal{E}_{6}*_{M}\mathcal{B}_{4}^{\dagger}+\mathcal{V}_{22}*_{M}\mathcal{R}_{\mathcal{B}_{4}},\\
\label{system 2227}
&\mathcal{Y}_{1}=\mathcal{A}_{5}^{\dagger}*_{N}\mathcal{E}_{7}+\mathcal{L}_{\mathcal{A}_{5}}*_{N}\mathcal{V}_{33},\\
\label{system 2228}
& \mathcal{Y}_{2}=\mathcal{E}_{8}*_{M}\mathcal{B}_{5}^{\dagger}+\mathcal{V}_{44}*_{M}\mathcal{R}_{\mathcal{B}_{5}},
\end{align}
\end{small}
where $\mathcal{V}_{ii}$  $(i=\overline{1,4})$ are arbitrary  tensors with qualified orders. Substitute expressions  \eqref{system 22251} and \eqref{system 22261} into  \eqref{3.111} yields:
\begin{small}
\begin{equation}
\label{3.111aa}
\begin{gathered}
\mathcal{\widehat{A}}_{6}*_{N}\mathcal{V}_{11}*_{M}\mathcal{B}_{6}+\mathcal{A}_{7}*_{N}\mathcal{V}_{22}*_{M}\mathcal{\widehat{B}}_{7}\\
+\mathcal{A}_{7}*_{N}(\mathcal{C}_{3}*_{N}\mathcal{X}_{3}*_{M}\mathcal{D}_{3}
+\mathcal{C}_{4}*_{N}\mathcal{W}*_{M}\mathcal{D}_{4})*_{M}\mathcal{B}_{6}=\mathcal{\widehat{E}}_{9},
\end{gathered}
\end{equation}
\end{small}
where $\mathcal{\widehat{A}}_{6}$, $\mathcal{\widehat{B}}_{7}$ and $\mathcal{\widehat{E}}_{9}$ are given by  \eqref{system 3.33A} and \eqref{system 3.33B}.
Utilizing $Proposition$ $\ref{lma 2.3b}$, we have that \eqref{3.111aa} is solvable if and only if the conditions \eqref{system 222b} are satisfied and there exist quaternion tensors $\mathcal{X}_{3}$ and $\mathcal{W}$ satisfy the following equation:
\begin{small}
\begin{equation}
\label{3.111r}
\begin{gathered}
\mathcal{\widehat{A}}_{4}*_{N}\mathcal{X}_{3}*_{M}\mathcal{\widehat{B}}_{4}+\mathcal{\widehat{C}}_{4}*_{N}\mathcal{W}*_{M}\mathcal{\widehat{D}}_{4}
=\mathcal{\widehat{P}},
\end{gathered}
\end{equation}
\end{small}
where $\mathcal{\widehat{A}}_{4},$ $\mathcal{\widehat{B}}_{4},$ $\mathcal{\widehat{C}}_{4}$, $\mathcal{\widehat{D}}_{4}$, and $\mathcal{\widehat{P}}$ are defined in \eqref{system 3.33EE}-\eqref{system 3.33F}. In that case,  the quaternion tensors $\mathcal{V}_{11}$ and $\mathcal{V}_{22}$ can be expressed as
\begin{small}
\begin{subequations}
\begin{align}
&\begin{array}{l}
 \mathcal{V}_{11}=\mathcal{\widehat{A}}_{6}^{\dagger}*_{N}\mathcal{\grave{E}}_{1}*_{M}\mathcal{B}_{6}^{\dagger}
-\mathcal{\widehat{A}}_{6}^{\dagger}*_{N}\mathcal{A}_{7}*_{N}\mathcal{M}_{11}^{\dagger}*_{N}\mathcal{\grave{E}}_{1}*_{M}\mathcal{B}_{6}^{\dagger}
-\mathcal{\widehat{A}}_{6}^{\dagger}*_{N}\mathcal{S}_{11}*_{N}\\
\ \ \ \ \ \mathcal{A}_{7}^{\dagger}*_{N}\mathcal{\grave{E}}_{1}*_{M}\mathcal{N}_{11}^{\dagger}*_{M}\mathcal{\widehat{B}}_{7}*_{M}\mathcal{B}_{6}^{\dagger}-\mathcal{\widehat{A}}_{6}^{\dagger}*_{N}\mathcal{S}_{11}
*_{N}\mathcal{T}_{21}*_{M}\mathcal{R}_{\mathcal{N}_{11}}*_{M}\mathcal{\widehat{B}}_{7}\\
\ \ \ \ \ *_{M}\mathcal{B}_{6}^{\dagger}
+\mathcal{L}_{\mathcal{\widehat{A}}_{6}}*_{N}\mathcal{T}_{41}+\mathcal{T}_{51}*_{M}\mathcal{R}_{\mathcal{B}_{6}},
 \end{array}\\
&\begin{array}{l}
 \mathcal{V}_{22}=\mathcal{M}_{11}^{\dagger}*_{N}\mathcal{\grave{E}}_{1}*_{M}\mathcal{\widehat{B}}_{7}^{\dagger}
 +\mathcal{S}_{11}^{\dagger}*_{N}\mathcal{S}_{11}*_{N}\mathcal{A}_{7}^{\dagger}
*_{N}\mathcal{\grave{E}}_{1}*_{M}\mathcal{N}_{11}^{\dagger}+\mathcal{L}_{\mathcal{M}_{11}}*_{N}\\
\ \ \ \ \ \ \mathcal{L}_{\mathcal{S}_{11}}*_{N}\mathcal{T}_{11}+\mathcal{L}_{\mathcal{M}_{11}}*_{N}\mathcal{T}_{21}*_{M}\mathcal{R}_{\mathcal{N}_{11}}+\mathcal{T}_{31}*_{M}\mathcal{R}_{\mathcal{\widehat{B}}_{7}},
 \end{array}
\end{align}
\end{subequations}
\end{small}
with $\mathcal{T}_{l1}$ $(l=\overline{1,5})$ are arbitrary with appropriate orders.
Similarly, we substitute \eqref{system 2227} and \eqref{system 2228} into \eqref{3.1112} yields:
\begin{small}
\begin{equation}
\label{3.111a}
\begin{gathered}
\mathcal{\widehat{A}}_{8}*_{N}\mathcal{V}_{33}*_{M}\mathcal{B}_{8}+\mathcal{A}_{9}*_{N}\mathcal{V}_{44}*_{M}\mathcal{\widehat{B}}_{9}\\
+\mathcal{A}_{9}*_{N}(\mathcal{H}_{3}*_{N}\mathcal{Y}_{3}*_{M}\mathcal{J}_{3}
+\mathcal{H}_{4}*_{N}\mathcal{W}*_{M}\mathcal{J}_{4})*_{M}\mathcal{B}_{8}=\mathcal{\widehat{E}}_{10},
\end{gathered}
\end{equation}
\end{small}
where  $\mathcal{\widehat{A}}_{8}$, $\mathcal{\widehat{B}}_{9},$ and $\mathcal{\widehat{E}}_{10}$ are defined in \eqref{system 3.33G} and \eqref{system 3.33H}.
Utilizing $Proposition$ $\ref{lma 2.3b}$, we have that \eqref{3.111a} is solvable if and only if the conditions \eqref{system 222c} are satisfied and there exist quaternion tensors $\mathcal{X}_{3}$ and $\mathcal{W}$ satisfy the tensor equation:
\begin{small}
\begin{equation}
\label{3.111r}
\begin{gathered}
\mathcal{\widehat{A}}_{5}*_{N}\mathcal{Y}_{3}*_{M}\mathcal{\widehat{B}}_{5}+\mathcal{\widehat{C}}_{5}*_{N}\mathcal{W}*_{M}\mathcal{\widehat{D}}_{5}
=\mathcal{\widehat{Q}},
\end{gathered}
\end{equation}
\end{small}
where the quaternion tensors $\mathcal{\widehat{A}}_{5},$ $\mathcal{\widehat{B}}_{5},$ $\mathcal{\widehat{C}}_{5}$, $\mathcal{\widehat{D}}_{5}$ and $\mathcal{\widehat{Q}}$ are defined in \eqref{system 3.33KKK}-\eqref{system 3.33L}. In that case,  $\mathcal{V}_{33}$ and $\mathcal{V}_{44}$ can be expressed as
\begin{small}
\begin{subequations}
\begin{align}
&\begin{array}{l}
 \mathcal{V}_{33}=\mathcal{\widehat{A}}_{8}^{\dagger}*_{N}\mathcal{\grave{E}}_{2}*_{M}\mathcal{B}_{8}^{\dagger}
-\mathcal{\widehat{A}}_{8}^{\dagger}*_{N}\mathcal{A}_{9}*_{N}\mathcal{M}_{22}^{\dagger}*_{N}\mathcal{\grave{E}}_{2}*_{M}\mathcal{B}_{8}^{\dagger}
-\mathcal{\widehat{A}}_{8}^{\dagger}*_{N}\mathcal{S}_{22}*_{N}\\
\ \ \ \ \ \mathcal{A}_{9}^{\dagger}*_{N}\mathcal{\grave{E}}_{2}*_{M}\mathcal{N}_{22}^{\dagger}*_{M}\mathcal{\widehat{B}}_{9}*_{M}\mathcal{B}_{8}^{\dagger}-\mathcal{\widehat{A}}_{8}^{\dagger}*_{N}\mathcal{S}_{22}
*_{N}\mathcal{J}_{21}*_{M}\mathcal{R}_{\mathcal{N}_{22}}*_{M}\mathcal{\widehat{B}}_{9}
\\
\ \ \ \ \ *_{M}\mathcal{B}_{8}^{\dagger}+\mathcal{L}_{\mathcal{\widehat{A}}_{8}}*_{N}\mathcal{J}_{41}+\mathcal{J}_{51}*_{M}\mathcal{R}_{\mathcal{B}_{8}},
 \end{array}\\
&\begin{array}{l}
 \mathcal{V}_{44}=\mathcal{M}_{22}^{\dagger}*_{N}\mathcal{\grave{E}}_{2}*_{M}\mathcal{\widehat{B}}_{9}^{\dagger}
 +\mathcal{S}_{22}^{\dagger}*_{N}\mathcal{S}_{22}*_{N}\mathcal{A}_{9}^{\dagger}
*_{N}\mathcal{\grave{E}}_{2}*_{M}\mathcal{N}_{22}^{\dagger}+\mathcal{L}_{\mathcal{M}_{22}}*_{N}\\
\ \ \ \ \ \ \mathcal{L}_{\mathcal{S}_{22}}*_{N}\mathcal{J}_{11}+\mathcal{L}_{\mathcal{M}_{22}}*_{N}\mathcal{J}_{21}*_{M}\mathcal{R}_{\mathcal{N}_{22}}+\mathcal{J}_{31}*_{M}\mathcal{R}_{\mathcal{\widehat{B}}_{9}},
 \end{array}
\end{align}
\end{subequations}
\end{small}
where $\mathcal{J}_{l1}$ $(l=\overline{1,5})$ are arbitrary tensors with appropriate sizes.
Now, we summarize all previous processes in our proof. The system of Sylvester-type quaternion  tensor equations \eqref{1.5} is solvable if and only if the conditions \eqref{system 222a},  \eqref{system 222b} and \eqref{system 222c} are satisfying and there exist quaternion tensors  $\mathcal{X}_{3}$, $\mathcal{Y}_{3}$ and $\mathcal{W}$ verify the following system:
\begin{small}
\begin{align}
\label{1.4cc}
\left\{
\begin{array}{rll}
\begin{gathered}
\mathcal{A}_{1}*_{N}\mathcal{X}_{3}*_{M}\mathcal{B}_{1}=\mathcal{E}_{1},\\
\mathcal{A}_{2}*_{N}\mathcal{Y}_{3}*_{M}\mathcal{B}_{2}=\mathcal{E}_{2},\\
\mathcal{A}_{3}*_{N}\mathcal{W}=\mathcal{E}_{3},\ \mathcal{W}*_{M}\mathcal{B}_{3}=\mathcal{E}_{4}\\
\mathcal{\widehat{A}}_{4}*_{N}\mathcal{X}_{3}*_{M}\mathcal{\widehat{B}}_{4}+\mathcal{\widehat{C}}_{4}*_{N}\mathcal{W}*_{M}\mathcal{\widehat{D}}_{4}=\mathcal{\widehat{P}},\\
\mathcal{\widehat{A}}_{5}*_{N}\mathcal{Y}_{3}*_{M}\mathcal{\widehat{B}}_{5}+\mathcal{\widehat{C}}_{5}*_{N}\mathcal{W}*_{M}\mathcal{\widehat{D}}_{5}=\mathcal{\widehat{Q}}.
\end{gathered}
\end{array}
  \right.
\end{align}
\end{small}
Finally, utilizing $Lemma$ $\ref{lma 2.3a},$ we have that  \eqref{1.4cc} is solvable  if and only if the conditions defined by \eqref{system 222ddd}-\eqref{system 222eee} are satisfying. In that case, $\mathcal{X}_{3}$, $\mathcal{Y}_{3}$ and $\mathcal{W}$  can expressed as
\begin{small}
\begin{align}
\label{system 2229}
&\mathcal{X}_{3}=\mathcal{A}_{1}^{\dagger}*_{N}\mathcal{E}_{1}*_{M}\mathcal{B}_{1}^{\dagger}
+\mathcal{L}_{\mathcal{A}_{1}}*_{N}\mathcal{U}_{1}+\mathcal{U}_{2}*_{M}\mathcal{R}_{\mathcal{B}_{1}},\\
\label{system 2230}
&\mathcal{Y}_{3}=\mathcal{A}_{2}^{\dagger}*_{N}\mathcal{E}_{2}*_{M}\mathcal{B}_{2}^{\dagger}
+\mathcal{L}_{\mathcal{A}_{2}}*_{N}\mathcal{U}_{3}+\mathcal{U}_{4}*_{M}\mathcal{R}_{\mathcal{B}_{2}},\\
\label{system 2231}
&\mathcal{W}=\mathcal{A}_{3}^{\dagger}*_{N}\mathcal{E}_{3}+\mathcal{L}_{\mathcal{A}_{3}}*_{N}\mathcal{E}_{4}*_{M}\mathcal{B}_{3}^{\dagger}
+\mathcal{L}_{\mathcal{A}_{3}}*_{N}\mathcal{U}_{5}*_{M}\mathcal{R}_{\mathcal{B}_{3}},
\end{align}
\end{small}
where $\mathcal{U}_{i}$   $(i=\overline{1,5})$  are arbitrary  tensors defined by \eqref{system 2222eeee}-\eqref{system 2222iiii}.
\end{proof}

Here, we give an  algorithm with a numerical example to illustrate $Theorem$ $\ref{system 3.33AAt}$.  All computations can be run on MATLAB 2020b.
\begin{algorithm}Calculate  the general solution to \eqref{1.5}
\begin{enumerate}
\item \textbf{Input} the system of two-sided four coupled Sylvester-like quaternion tensor equations \eqref{1.5} with viable orders over $\mathbb{H}$.
\item Compute all quaternion tensors, which appeared in \eqref{system 3.33A}-\eqref{system 3.33KK}.
 \item Check whether the Moore-Penrose inverses conditions in $Theorem$ $\ref{system 3.33AAt}$ are satisfiing or not. If not, return``The system \eqref{1.5} is inconsistent''.
 \item Else compute the quaternion unknowns $\mathcal{X}_{i}, \mathcal{Y}_{i}, \mathcal{W}$, where $(i=\overline{1,3})$  by \eqref{system 2225}-\eqref{system 22t}.
 \item \textbf{Output} the general solution of the system \eqref{1.5} is $\mathcal{X}_{i}, \mathcal{Y}_{i}, \mathcal{W}$.
\end{enumerate}
\end{algorithm}
\begin{example}Assume that the fourth ordered tensors in  $\eqref{1.5}$ are given as
\begin{small}
\begin{align*}
\begin{gathered}
\mathcal{A}_{1}(:,:,1,1)=\left(\begin{smallmatrix}
             \mathbf{k} & \mathbf{i+k} \\
            \mathbf{0} & \mathbf{0}
            \end{smallmatrix}\right),
\mathcal{A}_{1}(:,:,2,1)=\left(\begin{smallmatrix}
             \mathbf{5i} & \mathbf{1} \\
            \mathbf{0} & \mathbf{0}
            \end{smallmatrix}\right),\
\mathcal{A}_{1}(:,:,1,2)=\left(\begin{smallmatrix}
             \mathbf{2-i} & \mathbf{0} \\
            \mathbf{0} & \mathbf{2k}
            \end{smallmatrix}\right),\\
\mathcal{A}_{1}(:,:,2,2)=\left(\begin{smallmatrix}
             \mathbf{0} & \mathbf{i} \\
            \mathbf{k} & \mathbf{0}
            \end{smallmatrix}\right),\
\mathcal{B}_{1}(:,:,1,1)=\left(\begin{smallmatrix}
             \mathbf{k} & \mathbf{2-k} \\
            \mathbf{j} & \mathbf{0}
            \end{smallmatrix}\right),
\mathcal{B}_{1}(:,:,1,2)=\left(\begin{smallmatrix}
             \mathbf{0} & \mathbf{0} \\
            \mathbf{0} & \mathbf{3i}
            \end{smallmatrix}\right),\\
\mathcal{B}_{1}(:,:,2,1)=\left(\begin{smallmatrix}
             \mathbf{i} & \mathbf{j} \\
            \mathbf{k} & \mathbf{2}
            \end{smallmatrix}\right),\
\mathcal{B}_{1}(:,:,2,2)=\left(\begin{smallmatrix}
             \mathbf{2} & \mathbf{i} \\
            \mathbf{0} & \mathbf{j}
            \end{smallmatrix}\right),\
\mathcal{A}_{2}(:,:,1,1)=\left(\begin{smallmatrix}
             \mathbf{6} & \mathbf{k} \\
            \mathbf{0} & \mathbf{-k}
            \end{smallmatrix}\right),\\
\mathcal{A}_{2}(:,:,1,2)=\left(\begin{smallmatrix}
             \mathbf{k} & \mathbf{0} \\
            \mathbf{9} & \mathbf{0}
            \end{smallmatrix}\right),\
\mathcal{A}_{2}(:,:,2,1)=\left(\begin{smallmatrix}
             \mathbf{7} & \mathbf{0} \\
            \mathbf{0} & \mathbf{-3j}
            \end{smallmatrix}\right),\
\mathcal{A}_{2}(:,:,2,2)=\left(\begin{smallmatrix}
             \mathbf{2k} & \mathbf{0} \\
            \mathbf{8} & \mathbf{0}
            \end{smallmatrix}\right),\\
\mathcal{B}_{2}(:,:,1,1)=\left(\begin{smallmatrix}
             \mathbf{8} & \mathbf{0} \\
            \mathbf{0} & \mathbf{3j}
            \end{smallmatrix}\right),\
 \mathcal{B}_{2}(:,:,1,2)=\left(\begin{smallmatrix}
             \mathbf{0} & \mathbf{0} \\
            \mathbf{7} & \mathbf{0}
            \end{smallmatrix}\right),\
\mathcal{B}_{2}(:,:,2,1)=\left(\begin{smallmatrix}
             \mathbf{9} & \mathbf{i-k} \\
            \mathbf{0} & \mathbf{0}
            \end{smallmatrix}\right),\\
\mathcal{B}_{2}(:,:,2,2)=\left(\begin{smallmatrix}
             \mathbf{i} & \mathbf{0} \\
            \mathbf{6} & \mathbf{-k}
            \end{smallmatrix}\right),\
\mathcal{A}_{4}(:,:,1,1)=\left(\begin{smallmatrix}
             \mathbf{0} & \mathbf{j} \\
            \mathbf{j} & \mathbf{0}
            \end{smallmatrix}\right),\
 \mathcal{A}_{4}(:,:,1,2)=\left(\begin{smallmatrix}
             \mathbf{0} & \mathbf{k} \\
            \mathbf{0} & \mathbf{k}
            \end{smallmatrix}\right),\\
\mathcal{A}_{4}(:,:,2,1)=\left(\begin{smallmatrix}
             \mathbf{i+j} & \mathbf{0} \
            \mathbf{0} & \mathbf{k}
            \end{smallmatrix}\right),\
\mathcal{A}_{4}(:,:,2,2)=\left(\begin{smallmatrix}
             \mathbf{0} & \mathbf{0} \\
            \mathbf{j+k} & \mathbf{i}
            \end{smallmatrix}\right),\
\mathcal{B}_{4}(:,:,1,1)=\left(\begin{smallmatrix}
             \mathbf{0} & \mathbf{i+k} \\
            \mathbf{0} & \mathbf{j}
            \end{smallmatrix}\right),\\
\mathcal{B}_{4}(:,:,1,2)=\left(\begin{smallmatrix}
             \mathbf{i-j} & \mathbf{0} \\
            \mathbf{k} & \mathbf{0}
            \end{smallmatrix}\right),\
\mathcal{B}_{4}(:,:,2,1)=\left(\begin{smallmatrix}
             \mathbf{0} & \mathbf{0} \\
            \mathbf{j-k} & \mathbf{i}
            \end{smallmatrix}\right),\
\mathcal{B}_{4}(:,:,2,2)=\left(\begin{smallmatrix}
             \mathbf{k-i} & \mathbf{j} \\
            \mathbf{0} & \mathbf{0}
            \end{smallmatrix}\right),\\
\mathcal{A}_{5}(:,:,1,1)=\left(\begin{smallmatrix}
             \mathbf{0} & \mathbf{i} \\
            \mathbf{0} & \mathbf{0}
            \end{smallmatrix}\right),\
\mathcal{A}_{5}(:,:,1,2)=\left(\begin{smallmatrix}
             \mathbf{1} & \mathbf{0} \\
            \mathbf{0} & \mathbf{j}
            \end{smallmatrix}\right),\
\mathcal{A}_{5}(:,:,2,1)=\left(\begin{smallmatrix}
             \mathbf{0} & \mathbf{0} \\
            \mathbf{0} & \mathbf{-k}
            \end{smallmatrix}\right),\\
\mathcal{A}_{5}(:,:,2,2)=\left(\begin{smallmatrix}
             \mathbf{1} & \mathbf{i} \\
            \mathbf{0} & \mathbf{0}
            \end{smallmatrix}\right),\
\mathcal{B}_{5}(:,:,1,1)=\left(\begin{smallmatrix}
             \mathbf{2i} & \mathbf{0} \\
            \mathbf{2j} & \mathbf{0}
            \end{smallmatrix}\right),\
 \mathcal{B}_{5}(:,:,1,2)=\left(\begin{smallmatrix}
             \mathbf{0} & \mathbf{i-k} \\
            \mathbf{0} & \mathbf{i+k}
            \end{smallmatrix}\right),\\
\mathcal{B}_{5}(:,:,2,1)=\left(\begin{smallmatrix}
             \mathbf{-i+j} & \mathbf{0} \\
            \mathbf{0} & \mathbf{i+k}
            \end{smallmatrix}\right),\
 \mathcal{B}_{5}(:,:,2,2)=\left(\begin{smallmatrix}
             \mathbf{0} & \mathbf{3} \\
            \mathbf{3-i} & \mathbf{j}
            \end{smallmatrix}\right),\
\mathcal{A}_{3}(:,:,1,1)=\left(\begin{smallmatrix}
             \mathbf{0} & \mathbf{0} \\
            \mathbf{2} & \mathbf{i}
            \end{smallmatrix}\right),\\
\mathcal{A}_{3}(:,:,1,2)=\left(\begin{smallmatrix}
             \mathbf{2} & \mathbf{0} \\
            \mathbf{0} & \mathbf{2i}
            \end{smallmatrix}\right),\
\mathcal{A}_{3}(:,:,2,1)=\left(\begin{smallmatrix}
             \mathbf{0} & \mathbf{0} \\
            \mathbf{0} & \mathbf{k}
            \end{smallmatrix}\right),\
 \mathcal{A}_{3}(:,:,2,2)=\left(\begin{smallmatrix}
             \mathbf{0} & \mathbf{3-j} \\
            \mathbf{0} & \mathbf{0}
            \end{smallmatrix}\right),\\
\mathcal{B}_{3}(:,:,1,1)=\left(\begin{smallmatrix}
             \mathbf{i} & \mathbf{0} \\
            \mathbf{0} & \mathbf{2j}
            \end{smallmatrix}\right),\
 \mathcal{B}_{3}(:,:,1,2)=\left(\begin{smallmatrix}
             \mathbf{k} & \mathbf{-2k} \\
            \mathbf{0} & \mathbf{0}
            \end{smallmatrix}\right),\
\mathcal{B}_{3}(:,:,2,1)=\left(\begin{smallmatrix}
             \mathbf{0} & \mathbf{0} \\
            \mathbf{0} & \mathbf{2i}
            \end{smallmatrix}\right),\\
\mathcal{B}_{3}(:,:,2,2)=\left(\begin{smallmatrix}
             \mathbf{3i} & \mathbf{0} \\
            \mathbf{-j} & \mathbf{0}
            \end{smallmatrix}\right),\
\mathcal{E}_{1}(:,:,1,1)=\left(\begin{smallmatrix}
             \mathbf{-7+i-6j-5k} & \mathbf{2-j+2k} \\
            \mathbf{-k} & \mathbf{2-6i-4j}
            \end{smallmatrix}\right),\\
\mathcal{E}_{1}(:,:,1,2)=\left(\begin{smallmatrix}
             \mathbf{3+3i} & \mathbf{3+3i-3k} \\
            \mathbf{-3k} & \mathbf{0}
            \end{smallmatrix}\right),\
\mathcal{E}_{1}(:,:,2,1)=\left(\begin{smallmatrix}
             \mathbf{1+i+7j-2k} & \mathbf{-1-4i+2j} \\
            \mathbf{3j} & \mathbf{-2i-6k}
            \end{smallmatrix}\right),\\
\mathcal{E}_{1}(:,:,2,2)=\left(\begin{smallmatrix}
             \mathbf{2+4i-j+10k} & \mathbf{-1-i+3j-k} \\
            \mathbf{-1} & \mathbf{4+4j}
            \end{smallmatrix}\right),\
\mathcal{E}_{2}(:,:,1,1)=\left(\begin{smallmatrix}
             \mathbf{-56+8i+90j+18k} & \mathbf{-3-8k} \\
            \mathbf{24i-72j+72k} & \mathbf{39+8k}
            \end{smallmatrix}\right),\\
\mathcal{E}_{2}(:,:,1,2)=\left(\begin{smallmatrix}
             \mathbf{84i-14k} & \mathbf{14j} \\
            \mathbf{-126} & \mathbf{-14j}
            \end{smallmatrix}\right),\
\mathcal{E}_{2}(:,:,2,1)=\left(\begin{smallmatrix}
             \mathbf{-45+42i-18j-33k} & \mathbf{3i-6k} \\
            \mathbf{8-73j+81k} & \mathbf{12i+21k}
            \end{smallmatrix}\right),\\
\mathcal{E}_{2}(:,:,2,2)=\left(\begin{smallmatrix}
             \mathbf{-1+65i+6j-42k} & \mathbf{-i+11j} \\
            \mathbf{-116+9j+9k} & \mathbf{13i-11j}
            \end{smallmatrix}\right),\
\mathcal{E}_{3}(:,:,1,1)=\left(\begin{smallmatrix}
             \mathbf{0} & \mathbf{0} \\
            \mathbf{2i-2k} & \mathbf{-1+j}
            \end{smallmatrix}\right),\\
\mathcal{E}_{3}(:,:,1,2)=\left(\begin{smallmatrix}
             \mathbf{0} & \mathbf{0} \\
            \mathbf{2j-2k} & \mathbf{j+k}
            \end{smallmatrix}\right),\
\mathcal{E}_{3}(:,:,2,1)=\left(\begin{smallmatrix}
             \mathbf{0} & \mathbf{3-3i-j-k} \\
            \mathbf{2+2i} & \mathbf{-1+i}
            \end{smallmatrix}\right) ,\\
\mathcal{E}_{3}(:,:,2,2)=\left(\begin{smallmatrix}
             \mathbf{0} & \mathbf{6i+2k} \\
            \mathbf{0} & \mathbf{0}
            \end{smallmatrix}\right),\
\mathcal{E}_{4}(:,:,1,1)=\left(\begin{smallmatrix}
             \mathbf{-1-j} & \mathbf{0} \\
            \mathbf{0} & \mathbf{4k}
            \end{smallmatrix}\right),\
\mathcal{E}_{4}(:,:,1,2)=\left(\begin{smallmatrix}
             \mathbf{-1-2i-j} & \mathbf{0} \\
            \mathbf{0} & \mathbf{0}
            \end{smallmatrix}\right),\\
\mathcal{E}_{4}(:,:,2,1)=\left(\begin{smallmatrix}
             \mathbf{0} & \mathbf{0} \\
            \mathbf{0} & \mathbf{-4}
            \end{smallmatrix}\right),\
\mathcal{E}_{4}(:,:,2,2)=\left(\begin{smallmatrix}
             \mathbf{-3-4j-k} & \mathbf{0} \\
            \mathbf{0} & \mathbf{-j+k}
            \end{smallmatrix}\right),\\
\mathcal{E}_{5}(:,:,1,1)=\left(\begin{smallmatrix}
             \mathbf{-1-k} & \mathbf{-2j} \\
            \mathbf{-2j} & \mathbf{j}
            \end{smallmatrix}\right),\
 \mathcal{E}_{5}(:,:,1,2)=\left(\begin{smallmatrix}
             \mathbf{0} & \mathbf{1+2k} \\
            \mathbf{-1+i} & \mathbf{1-j+2k}
            \end{smallmatrix}\right),\\
\mathcal{E}_{5}(:,:,2,1)=\left(\begin{smallmatrix}
             \mathbf{0} & \mathbf{-1+j} \\
            \mathbf{i+j} & \mathbf{-k}
            \end{smallmatrix}\right),\
\mathcal{E}_{5}(:,:,2,2)=\left(\begin{smallmatrix}
             \mathbf{3i+j} & \mathbf{0} \\
            \mathbf{0} & \mathbf{-1+2k}
            \end{smallmatrix}\right),            
\end{gathered}            
 \end{align*}                  
\begin{align*}
\begin{gathered} 
\mathcal{E}_{6}(:,:,1,1)=\left(\begin{smallmatrix}
             \mathbf{j-2k} & \mathbf{i+2j-k} \\
            \mathbf{0} & \mathbf{i+2k}
            \end{smallmatrix}\right),\
\mathcal{E}_{6}(:,:,2,1)=\left(\begin{smallmatrix}
             \mathbf{-1-j-k} & \mathbf{2i} \\
            \mathbf{0} & \mathbf{i}
            \end{smallmatrix}\right),\\
\mathcal{E}_{6}(:,:,1,2)=\left(\begin{smallmatrix}
             \mathbf{2i+j+k} & \mathbf{0} \\
            \mathbf{0} & \mathbf{-1-i-k}
            \end{smallmatrix}\right),\
\mathcal{E}_{6}(:,:,2,2)=\left(\begin{smallmatrix}
             \mathbf{-2-2j} & \mathbf{-1} \\
            \mathbf{0} & \mathbf{1}
            \end{smallmatrix}\right),\\
\mathcal{E}_{7}(:,:,1,1)=\left(\begin{smallmatrix}
             \mathbf{0} & \mathbf{0} \\
            \mathbf{0} & \mathbf{-2j}
            \end{smallmatrix}\right),\
\mathcal{E}_{7}(:,:,1,2)=\left(\begin{smallmatrix}
             \mathbf{2i} & \mathbf{-3+j} \\
            \mathbf{0} & \mathbf{0}
            \end{smallmatrix}\right),\
\mathcal{E}_{7}(:,:,2,1)=\left(\begin{smallmatrix}
             \mathbf{3i+3j} & \mathbf{3k} \\
            \mathbf{0} & \mathbf{-3k},
             \end{smallmatrix}\right),\\
\mathcal{E}_{7}(:,:,2,2)=\left(\begin{smallmatrix}
             \mathbf{4j} & \mathbf{-5} \\
            \mathbf{0} & \mathbf{-4},
             \end{smallmatrix}\right),\
\mathcal{E}_{8}(:,:,1,1)=\left(\begin{smallmatrix}
             \mathbf{4i} & \mathbf{-2} \\
            \mathbf{2k} & \mathbf{-2},
             \end{smallmatrix}\right),\\
\mathcal{E}_{8}(:,:,1,2)=\left(\begin{smallmatrix}
             \mathbf{-1+3i-j-3k} & \mathbf{0} \\
            \mathbf{i-k} & \mathbf{-2+2j},
             \end{smallmatrix}\right),\
\mathcal{E}_{8}(:,:,2,1)=\left(\begin{smallmatrix}
             \mathbf{-2i+2j} & \mathbf{1+k} \\
            \mathbf{i-k} & \mathbf{-1+j},
             \end{smallmatrix}\right),\\
\mathcal{E}_{8}(:,:,2,2)=\left(\begin{smallmatrix}
             \mathbf{9-3k} & \mathbf{0} \\
            \mathbf{3i} & \mathbf{2i+3j+k},
             \end{smallmatrix}\right),\
\mathcal{A}_{6}(:,:,1,1)=\left(\begin{smallmatrix}
             \mathbf{0} & \mathbf{k} \\
            \mathbf{0} & \mathbf{2i}
            \end{smallmatrix}\right),
\mathcal{A}_{6}(:,:,1,2)=\left(\begin{smallmatrix}
             \mathbf{-1} & \mathbf{-1+i} \\
            \mathbf{0} & \mathbf{0}
            \end{smallmatrix}\right),\\
\mathcal{A}_{6}(:,:,2,1)=\left(\begin{smallmatrix}
             \mathbf{0} & \mathbf{1} \\
            \mathbf{0} & \mathbf{1-i}
            \end{smallmatrix}\right),\
 \mathcal{A}_{6}(:,:,2,2)=\left(\begin{smallmatrix}
             \mathbf{0} & \mathbf{0} \\
            \mathbf{2} & \mathbf{2-i}
            \end{smallmatrix}\right),\
\mathcal{B}_{6}(:,:,1,2)=\left(\begin{smallmatrix}
             \mathbf{i} & \mathbf{2-i} \\
            \mathbf{j} & \mathbf{0}
            \end{smallmatrix}\right),\\
\mathcal{B}_{6}(:,:,1,1)=\left(\begin{smallmatrix}
             \mathbf{0} & \mathbf{0} \\
            \mathbf{0} & \mathbf{i+j+k}
            \end{smallmatrix}\right),\
\mathcal{B}_{6}(:,:,2,1)=\left(\begin{smallmatrix}
             \mathbf{j} & \mathbf{2-j} \\
            \mathbf{k} & \mathbf{0}
            \end{smallmatrix}\right),\
\mathcal{B}_{6}(:,:,2,2)=\left(\begin{smallmatrix}
             \mathbf{0} & \mathbf{0} \\
            \mathbf{0} & \mathbf{i+k}
            \end{smallmatrix}\right),\\           
\mathcal{A}_{7}(:,:,1,1)=\left(\begin{smallmatrix}
             \mathbf{i-k} & \mathbf{0} \\
            \mathbf{0} & \mathbf{-2k}
            \end{smallmatrix}\right),\
\mathcal{A}_{7}(:,:,1,2)=\left(\begin{smallmatrix}
             \mathbf{0} & \mathbf{i} \\
            \mathbf{j} & \mathbf{0}
            \end{smallmatrix}\right),\
\mathcal{A}_{7}(:,:,2,1)=\left(\begin{smallmatrix}
             \mathbf{i} & \mathbf{j-k} \\
            \mathbf{0} & \mathbf{0}
            \end{smallmatrix}\right),\\
\mathcal{A}_{7}(:,:,2,2)=\left(\begin{smallmatrix}
             \mathbf{0} & \mathbf{k} \\
            \mathbf{i} & \mathbf{0}
            \end{smallmatrix}\right),\
\mathcal{B}_{7}(:,:,1,1)=\left(\begin{smallmatrix}
             \mathbf{0} & \mathbf{j} \\
            \mathbf{k} & \mathbf{0}
            \end{smallmatrix}\right),\
\mathcal{B}_{7}(:,:,1,2)=\left(\begin{smallmatrix}
             \mathbf{j} & \mathbf{0} \\
            \mathbf{0} & \mathbf{k}
            \end{smallmatrix}\right),\\
\mathcal{B}_{7}(:,:,2,1)=\left(\begin{smallmatrix}
             \mathbf{i} & \mathbf{0} \\
            \mathbf{0} & \mathbf{j}
            \end{smallmatrix}\right),\
\mathcal{B}_{7}(:,:,2,2)=\left(\begin{smallmatrix}
             \mathbf{1-i} & \mathbf{0} \\
            \mathbf{0} & \mathbf{j}
            \end{smallmatrix}\right),\
\mathcal{C}_{3}(:,:,1,1)=\left(\begin{smallmatrix}
             \mathbf{0} & \mathbf{i-k} \\
            \mathbf{0} & \mathbf{j}
            \end{smallmatrix}\right),\\
\mathcal{C}_{3}(:,:,1,2)=\left(\begin{smallmatrix}
             \mathbf{i+j} & \mathbf{0} \\
            \mathbf{i-2j} & \mathbf{0}
            \end{smallmatrix}\right),\
\mathcal{C}_{3}(:,:,2,1)=\left(\begin{smallmatrix}
             \mathbf{0} & \mathbf{0} \\
            \mathbf{j+k} & \mathbf{i}
            \end{smallmatrix}\right),\
\mathcal{C}_{3}(:,:,2,2)=\left(\begin{smallmatrix}
             \mathbf{j} & \mathbf{-k} \\
            \mathbf{0} & \mathbf{0}
            \end{smallmatrix}\right),\\
\mathcal{D}_{3}(:,:,1,1)=\left(\begin{smallmatrix}
             \mathbf{i} & \mathbf{0} \\
            \mathbf{j} & \mathbf{0}
            \end{smallmatrix}\right),\
 \mathcal{D}_{3}(:,:,1,2)=\left(\begin{smallmatrix}
             \mathbf{0} & \mathbf{j} \\
            \mathbf{k} & \mathbf{0}
            \end{smallmatrix}\right),\
\mathcal{D}_{3}(:,:,2,1)=\left(\begin{smallmatrix}
             \mathbf{k} & \mathbf{i} \\
            \mathbf{0} & \mathbf{0}
            \end{smallmatrix}\right),\\
\mathcal{D}_{3}(:,:,2,2)=\left(\begin{smallmatrix}
             \mathbf{i} & \mathbf{0} \\
            \mathbf{0} & \mathbf{i}
            \end{smallmatrix}\right),\
\mathcal{C}_{4}(:,:,1,1)=\left(\begin{smallmatrix}
             \mathbf{i} & \mathbf{-i} \\
            \mathbf{0} & \mathbf{0}
            \end{smallmatrix}\right),\
\mathcal{C}_{4}(:,:,1,2)=\left(\begin{smallmatrix}
             \mathbf{j} & \mathbf{-j} \\
            \mathbf{0} & \mathbf{0}
            \end{smallmatrix}\right),\\
\mathcal{C}_{4}(:,:,2,1)=\left(\begin{smallmatrix}
             \mathbf{k} & \mathbf{-k} \\
            \mathbf{0} & \mathbf{0}
            \end{smallmatrix}\right),\
 \mathcal{C}_{4}(:,:,2,2)=\left(\begin{smallmatrix}
             \mathbf{0} & \mathbf{0} \\
            \mathbf{i} & \mathbf{-i}
            \end{smallmatrix}\right),\
\mathcal{D}_{4}(:,:,1,1)=\left(\begin{smallmatrix}
             \mathbf{0} & \mathbf{0} \\
            \mathbf{j} & \mathbf{-j}
            \end{smallmatrix}\right),\\
\mathcal{D}_{4}(:,:,1,2)=\left(\begin{smallmatrix}
             \mathbf{0} & \mathbf{0} \\
            \mathbf{k} & \mathbf{-k}
            \end{smallmatrix}\right),\
\mathcal{D}_{4}(:,:,2,1)=\left(\begin{smallmatrix}
             \mathbf{i} & \mathbf{i+j} \\
            \mathbf{0} & \mathbf{0}
            \end{smallmatrix}\right),\
 \mathcal{D}_{4}(:,:,2,2)=\left(\begin{smallmatrix}
             \mathbf{j} & \mathbf{j+k} \\
            \mathbf{0} & \mathbf{0}
            \end{smallmatrix}\right),\\
\mathcal{A}_{8}(:,:,1,1)=\left(\begin{smallmatrix}
             \mathbf{2j} & \mathbf{0} \\
            \mathbf{3k} & \mathbf{0}
            \end{smallmatrix}\right),\
 \mathcal{A}_{8}(:,:,1,2)=\left(\begin{smallmatrix}
             \mathbf{i-k} & \mathbf{-k} \\
            \mathbf{0} & \mathbf{0}
            \end{smallmatrix}\right),\
\mathcal{A}_{8}(:,:,2,1)=\left(\begin{smallmatrix}
             \mathbf{0} & \mathbf{0} \\
            \mathbf{i+k} & \mathbf{-k}
            \end{smallmatrix}\right),\\
\mathcal{A}_{8}(:,:,2,2)=\left(\begin{smallmatrix}
             \mathbf{k} & \mathbf{0} \\
            \mathbf{2k} & \mathbf{0}
            \end{smallmatrix}\right),\
\mathcal{B}_{8}(:,:,1,1)=\left(\begin{smallmatrix}
             \mathbf{0} & \mathbf{j} \\
            \mathbf{2j} & \mathbf{0}
            \end{smallmatrix}\right),\
 \mathcal{B}_{8}(:,:,1,2)=\left(\begin{smallmatrix}
             \mathbf{0} & \mathbf{i} \\
            \mathbf{3i} & \mathbf{0}
            \end{smallmatrix}\right),\\
\mathcal{B}_{8}(:,:,2,1)=\left(\begin{smallmatrix}
             \mathbf{i-j} & \mathbf{0} \\
            \mathbf{0} & \mathbf{-j}
            \end{smallmatrix}\right),\
\mathcal{B}_{8}(:,:,2,2)=\left(\begin{smallmatrix}
             \mathbf{j+k} & \mathbf{0} \\
            \mathbf{0} & \mathbf{-k}
            \end{smallmatrix}\right),\
\mathcal{A}_{9}(:,:,1,1)=\left(\begin{smallmatrix}
             \mathbf{0} & \mathbf{i+j} \\
            \mathbf{0} & \mathbf{k}
            \end{smallmatrix}\right),\\
\mathcal{A}_{9}(:,:,1,2)=\left(\begin{smallmatrix}
             \mathbf{0} & \mathbf{j+k} \\
            \mathbf{j} & \mathbf{k}
            \end{smallmatrix}\right),\
\mathcal{A}_{9}(:,:,2,1)=\left(\begin{smallmatrix}
             \mathbf{i+j} & \mathbf{i} \\
            \mathbf{0} & \mathbf{0}
            \end{smallmatrix}\right),\
\mathcal{A}_{9}(:,:,2,2)=\left(\begin{smallmatrix}
             \mathbf{j+k} & \mathbf{0} \\
            \mathbf{j} & \mathbf{k}
            \end{smallmatrix}\right),\\
\mathcal{B}_{9}(:,:,1,1)=\left(\begin{smallmatrix}
             \mathbf{i+j} & \mathbf{0} \\
            \mathbf{k} & \mathbf{0}
            \end{smallmatrix}\right),\
 \mathcal{B}_{9}(:,:,1,2)=\left(\begin{smallmatrix}
             \mathbf{0} & \mathbf{i+k} \\
            \mathbf{i} & \mathbf{k}
            \end{smallmatrix}\right),\
\mathcal{B}_{9}(:,:,2,1)=\left(\begin{smallmatrix}
             \mathbf{0} & \mathbf{0} \\
            \mathbf{j-k} & \mathbf{-j}
            \end{smallmatrix}\right),\\
\mathcal{B}_{9}(:,:,2,2)=\left(\begin{smallmatrix}
             \mathbf{i+j} & \mathbf{i} \\
            \mathbf{0} & \mathbf{0}
            \end{smallmatrix}\right),\
\mathcal{H}_{3}(:,:,1,1)=\left(\begin{smallmatrix}
             \mathbf{2} & \mathbf{i} \\
            \mathbf{0} & \mathbf{0}
            \end{smallmatrix}\right),\
\mathcal{H}_{3}(:,:,1,2)=\left(\begin{smallmatrix}
             \mathbf{0} & \mathbf{-5k} \\
            \mathbf{0} & \mathbf{k}
            \end{smallmatrix}\right),\\
\mathcal{H}_{3}(:,:,2,1)=\left(\begin{smallmatrix}
             \mathbf{3} & \mathbf{0} \\
            \mathbf{-i} & \mathbf{0}
            \end{smallmatrix}\right),\
\mathcal{H}_{3}(:,:,2,2)=\left(\begin{smallmatrix}
             \mathbf{0} & \mathbf{-6} \\
            \mathbf{0} & \mathbf{-i}
            \end{smallmatrix}\right),\
\mathcal{J}_{3}(:,:,1,1)=\left(\begin{smallmatrix}
             \mathbf{4} & \mathbf{j} \\
            \mathbf{0} & \mathbf{0}
            \end{smallmatrix}\right),\\
 \mathcal{J}_{3}(:,:,1,2)=\left(\begin{smallmatrix}
             \mathbf{-i} & \mathbf{-7} \\
            \mathbf{0} & \mathbf{0}
            \end{smallmatrix}\right),\
\mathcal{J}_{3}(:,:,2,1)=\left(\begin{smallmatrix}
             \mathbf{5} & \mathbf{0} \\
            \mathbf{-j} & \mathbf{0}
            \end{smallmatrix}\right),\
\mathcal{J}_{3}(:,:,2,2)=\left(\begin{smallmatrix}
             \mathbf{j} & \mathbf{-8} \\
            \mathbf{0} & \mathbf{0}
            \end{smallmatrix}\right),\\
\mathcal{H}_{4}(:,:,1,1)=\left(\begin{smallmatrix}
             \mathbf{i} & \mathbf{j} \\
            \mathbf{0} & \mathbf{0}
            \end{smallmatrix}\right),\
\mathcal{H}_{4}(:,:,1,2)=\left(\begin{smallmatrix}
             \mathbf{j} & \mathbf{k} \\
            \mathbf{0} & \mathbf{0}
            \end{smallmatrix}\right),\
\mathcal{H}_{4}(:,:,2,1)=\left(\begin{smallmatrix}
             \mathbf{i} & \mathbf{k} \\
            \mathbf{0} & \mathbf{0}
            \end{smallmatrix}\right),\\
\mathcal{H}_{4}(:,:,2,2)=\left(\begin{smallmatrix}
             \mathbf{1} & \mathbf{0} \\
            \mathbf{2-i} & \mathbf{0}
            \end{smallmatrix}\right),\
\mathcal{J}_{4}(:,:,1,1)=\left(\begin{smallmatrix}
             \mathbf{0} & \mathbf{0} \\
            \mathbf{k} & \mathbf{i}
            \end{smallmatrix}\right),\
\mathcal{J}_{4}(:,:,1,2)=\left(\begin{smallmatrix}
             \mathbf{0} & \mathbf{2} \\
            \mathbf{0} & \mathbf{-j}
            \end{smallmatrix}\right),\\
\mathcal{J}_{4}(:,:,2,1)=\left(\begin{smallmatrix}
             \mathbf{0} & \mathbf{0} \\
            \mathbf{2-i} & \mathbf{2+j}
            \end{smallmatrix}\right),\
 \mathcal{J}_{4}(:,:,2,2)=\left(\begin{smallmatrix}
             \mathbf{0} & \mathbf{3-k} \\
            \mathbf{0} & \mathbf{3+k}
            \end{smallmatrix}\right),\\        
\mathcal{E}_{9}(:,:,1,1)=\left(\begin{smallmatrix}
             \mathbf{2+8j-3k} & \mathbf{-10-4i+2j+2k} \\
            \mathbf{-4-i-j-k} & \mathbf{-4i+15j-7k}
            \end{smallmatrix}\right),\\
\mathcal{E}_{9}(:,:,2,1)=\left(\begin{smallmatrix}
             \mathbf{-6+4i-2j+3k} & \mathbf{12+13i-5j-k} \\
            \mathbf{2+8i+2j+10k} & \mathbf{-6+2i+7j+k}
            \end{smallmatrix}\right),\\
\mathcal{E}_{10}(:,:,1,1)=\left(\begin{smallmatrix}
             \mathbf{5+61i-37j-9k} & \mathbf{154+65i-136j-191k} \\
            \mathbf{26-4i-36j-18k} & \mathbf{-48+160i+16j-41k}
            \end{smallmatrix}\right),\\
\mathcal{E}_{10}(:,:,1,2)=\left(\begin{smallmatrix}
             \mathbf{12-45i-75j+7k} & \mathbf{249-133i-65j+171k} \\
            \mathbf{50-56i+5j+31k} & \mathbf{58+45i-176j-50k}
            \end{smallmatrix}\right),\
,\\
\mathcal{E}_{10}(:,:,2,1)=\left(\begin{smallmatrix}
             \mathbf{10-73i+36j-13k} & \mathbf{1-34i+7j-27k} \\
            \mathbf{-87-95i+186j+117k} & \mathbf{98-141i+14j+21k}
            \end{smallmatrix}\right),\\
\mathcal{E}_{10}(:,:,2,2)=\left(\begin{smallmatrix}
             \mathbf{-53+3i+24j+26k} & \mathbf{7+23i+42j-55k} \\
            \mathbf{-119+143i-15j+63k} & \mathbf{-143-37i+78j+36k}
            \end{smallmatrix}\right),
\end{gathered}
\end{align*}
\end{small}
\begin{small}
\begin{align*}
\begin{gathered}                
\mathcal{E}_{9}(:,:,1,2)=\left(\begin{smallmatrix}
             \mathbf{-6-8i-3j} & \mathbf{15+10i-7j+7k} \\
            \mathbf{8+8i-2j+14k} & \mathbf{-7-10i+2j+17k}
            \end{smallmatrix}\right),\
\mathcal{E}_{9}(:,:,2,2)=\left(\begin{smallmatrix}
             \mathbf{4+4i+6j+2k} & \mathbf{-9+3i+3j+2k} \\
            \mathbf{-6+2i-3k} & \mathbf{3+5i+15j-k}
            \end{smallmatrix}\right).
\end{gathered}            
\end{align*}
\end{small}
Direct computations yields
\begin{small}
\begin{align*}
\begin{gathered}     
\mathcal{R}_{\mathcal{A}_{4}}*_{2}\mathcal{E}_{5}=0,\ \mathcal{E}_{6}*_{2}\mathcal{L}_{\mathcal{B}_{4}}=0,\
\mathcal{R}_{\mathcal{A}_{5}}*_{2}\mathcal{E}_{7}=0,\ \mathcal{E}_{8}*_{2}\mathcal{L}_{\mathcal{B}_{5}}=0,\\
\mathcal{R}_{\mathcal{M}_{11}}*_{2}\mathcal{R}_{\mathcal{\widehat{A}}_{6}}*_{2}\mathcal{\widehat{E}}_{9}=0,\
\mathcal{\widehat{E}}_{9}*_{2}\mathcal{L}_{\mathcal{B}_{6}}*_{2}\mathcal{L}_{\mathcal{N}_{11}}=0,\
 \mathcal{R}_{\mathcal{A}_{7}}*_{2}\mathcal{\widehat{E}}_{9}*_{2}\mathcal{L}_{\mathcal{B}_{6}}=0,\\
\mathcal{R}_{\mathcal{M}_{22}}*_{2}\mathcal{R}_{\mathcal{\widehat{A}}_{8}}*_{2}\mathcal{\widehat{E}}_{10}=0,\
\mathcal{\widehat{E}}_{10}*_{2}\mathcal{L}_{\mathcal{B}_{8}}*_{2}\mathcal{L}_{\mathcal{N}_{22}}=0,\
 \mathcal{R}_{\mathcal{A}_{9}}*_{2}\mathcal{\widehat{E}}_{10}*_{2}\mathcal{L}_{\mathcal{B}_{8}}=0,\\
\mathcal{R}_{\mathcal{A}_{3}}*_{2}\mathcal{E}_{3}=0,\ \mathcal{E}_{4}*_{2}\mathcal{L}_{\mathcal{B}_{3}}=0,\
\mathcal{A}_{3}*_{2}\mathcal{E}_{3}=\mathcal{E}_{4}*_{2}\mathcal{B}_{3},\\
\mathcal{R}_{\mathcal{\widehat{A}}_{4}}*_{2}\mathcal{G}*_{2}\mathcal{L}_{\mathcal{D}_{6}}=0,\
\mathcal{R}_{\mathcal{C}_{6}}*_{2}\mathcal{G}*_{2}\mathcal{L}_{\mathcal{\widehat{B}}_{4}}=0,\
\mathcal{R}_{\mathcal{S}_{1}}*_{2}\mathcal{R}_{\mathcal{\widehat{A}}_{4}}*_{N}\mathcal{G}=0,\\
\mathcal{G}*_{2}\mathcal{L}_{\mathcal{\widehat{B}}_{4}}*_{2}\mathcal{L}_{\mathcal{N}_{1}}=0,\
\mathcal{R}_{\mathcal{\widehat{A}}_{5}}*_{N}\mathcal{F}*_{2}\mathcal{L}_{\mathcal{D}_{7}}=0,\
\mathcal{R}_{\mathcal{C}_{7}}*_{2}\mathcal{F}*_{M}\mathcal{L}_{\mathcal{\widehat{B}}_{5}}=0,\\
\mathcal{R}_{\mathcal{S}_{2}}*_{2}\mathcal{R}_{\mathcal{\widehat{A}}_{5}}*_{2}\mathcal{F}=0,\
\mathcal{F}*_{2}\mathcal{L}_{\mathcal{\widehat{B}}_{5}}*_{2}\mathcal{L}_{\mathcal{N}_{2}}=0,\
\mathcal{R}_{\mathcal{M}}*_{2}\mathcal{R}_{\mathcal{A}}*_{2}\mathcal{E}=0,\\
\mathcal{R}_{\mathcal{A}}*_{2}\mathcal{E}*_{2}\mathcal{L}_{\mathcal{D}}=0,\
\mathcal{E}*_{2}\mathcal{L}_{\mathcal{B}}*_{2}\mathcal{L}_{\mathcal{N}}=0,\
\mathcal{R}_{\mathcal{C}}*_{2}\mathcal{E}*_{2}\mathcal{L}_{\mathcal{B}}=0,\\
\mathcal{R}_{\mathcal{A}_{1}}*_{2}\mathcal{E}_{1}=0,\ \mathcal{E}_{1}*_{2}\mathcal{L}_{\mathcal{B}_{1}}=0,\
\mathcal{R}{\mathcal{A}_{2}}*_{2}\mathcal{E}_{2}=0,\ \mathcal{E}_{2}*_{2}\mathcal{B}_{2}=0,\\
\mathcal{R}_{\mathcal{A}_{33}}*_{2}\mathcal{E}_{33}=0,\ \mathcal{E}_{33}*_{2}\mathcal{L}_{\mathcal{B}_{33}}=0,\
\mathcal{R}{\mathcal{A}_{55}}*_{2}\mathcal{E}_{55}=0,\ \mathcal{E}_{55}*_{2}\mathcal{L}_{\mathcal{B}_{55}}=0,\\
\mathcal{R}_{\mathcal{A}_{77}}*_{2}\mathcal{E}_{77}=0,\ \mathcal{E}_{77}*_{2}\mathcal{L}_{\mathcal{B}_{77}}=0,\
\mathcal{R}_{\mathcal{A}_{99}}*_{2}\mathcal{E}_{99}=0,\ \mathcal{E}_{99}*_{2}\mathcal{L}_{\mathcal{B}_{99}}=0,\\
\mathcal{R}_{\mathcal{\widetilde{A}}}*_{2}\mathcal{\widetilde{E}}*_{2}\mathcal{L}_{\widetilde{B}}=0.
\end{gathered}      
\end{align*}
\end{small}
Consequently, \eqref{1.5} is solvable. In that case, the general solution to \eqref{1.5} can  be expressed as
\begin{small}
\begin{align*}
\begin{gathered}
\mathcal{X}_{3}(:,:,1,1)=\left(\begin{smallmatrix}
             \mathbf{0} & \mathbf{i} \\
            \mathbf{j} & \mathbf{0}
           \end{smallmatrix}\right),\
\mathcal{X}_{3}(:,:,1,2)=\left(\begin{smallmatrix}
             \mathbf{0} & \mathbf{2j} \\
            \mathbf{k} & \mathbf{0}
           \end{smallmatrix}\right),\
\mathcal{X}_{3}(:,:,2,1)=\left(\begin{smallmatrix}
             \mathbf{-i+k} & \mathbf{0} \\
            \mathbf{0} & \mathbf{j}
           \end{smallmatrix}\right),\\
\mathcal{X}_{3}(:,:,2,2)=\left(\begin{smallmatrix}
             \mathbf{j-k} & \mathbf{0} \\
            \mathbf{0} & \mathbf{i}
           \end{smallmatrix}\right),\
\mathcal{Y}_{3}(:,:,1,1)=\left(\begin{smallmatrix}
             \mathbf{-1} & \mathbf{-j+k} \\
            \mathbf{0} & \mathbf{0}
           \end{smallmatrix}\right),\
\mathcal{Y}_{3}(:,:,1,2)=\left(\begin{smallmatrix}
             \mathbf{-3i} & \mathbf{0} \\
            \mathbf{5} & \mathbf{0}
           \end{smallmatrix}\right),\\
\mathcal{Y}_{3}(:,:,2,1)=\left(\begin{smallmatrix}
             \mathbf{2i} & \mathbf{-2} \\
            \mathbf{0} & \mathbf{0}
           \end{smallmatrix}\right),\
 \mathcal{Y}_{3}(:,:,2,2)=\left(\begin{smallmatrix}
             \mathbf{i} & \mathbf{0} \\
            \mathbf{4} & \mathbf{-k}
           \end{smallmatrix}\right),\
\mathcal{X}_{1}(:,:,1,1)=\left(\begin{smallmatrix}
             \mathbf{-2} & \mathbf{0} \\
            \mathbf{i} & \mathbf{0}
           \end{smallmatrix}\right),\\
\mathcal{X}_{1}(:,:,1,2)=\left(\begin{smallmatrix}
             \mathbf{0} & \mathbf{2-k} \\
            \mathbf{0} & \mathbf{k}
           \end{smallmatrix}\right),\
\mathcal{X}_{1}(:,:,2,1)=\left(\begin{smallmatrix}
             \mathbf{1+j} & \mathbf{0} \\
            \mathbf{0} & \mathbf{-j}
           \end{smallmatrix}\right),\
 \mathcal{X}_{1}(:,:,2,2)=\left(\begin{smallmatrix}
             \mathbf{0} & \mathbf{0} \\
            \mathbf{2+k} & \mathbf{0}
           \end{smallmatrix}\right),\\
\mathcal{X}_{2}(:,:,1,1)=\left(\begin{smallmatrix}
             \mathbf{k} & \mathbf{0} \\
            \mathbf{0} & \mathbf{i}
           \end{smallmatrix}\right),\
\mathcal{X}_{2}(:,:,1,2)=\left(\begin{smallmatrix}
             \mathbf{j-1} & \mathbf{j} \\
            \mathbf{0} & \mathbf{1}
           \end{smallmatrix}\right),\
\mathcal{X}_{2}(:,:,2,1)=\left(\begin{smallmatrix}
             \mathbf{1} & \mathbf{2} \\
            \mathbf{0} & \mathbf{i}
           \end{smallmatrix}\right),\\
\mathcal{X}_{2}(:,:,2,2)=\left(\begin{smallmatrix}
             \mathbf{1} & \mathbf{3} \\
            \mathbf{0} & \mathbf{j}
           \end{smallmatrix}\right),\
\mathcal{Y}_{1}(:,:,1,1)=\left(\begin{smallmatrix}
             \mathbf{0} & \mathbf{0} \\
            \mathbf{2i} & \mathbf{0}
           \end{smallmatrix}\right),\
\mathcal{Y}_{1}(:,:,1,2)=\left(\begin{smallmatrix}
             \mathbf{i-k} & \mathbf{0} \\
            \mathbf{0} & \mathbf{2i}
           \end{smallmatrix}\right),\\
\mathcal{Y}_{1}(:,:,2,1)=\left(\begin{smallmatrix}
             \mathbf{0} & \mathbf{3i} \\
            \mathbf{0} & \mathbf{3j}
           \end{smallmatrix}\right),\
 \mathcal{Y}_{1}(:,:,2,2)=\left(\begin{smallmatrix}
             \mathbf{5i} & \mathbf{4j} \\
            \mathbf{0} & \mathbf{0}
           \end{smallmatrix}\right),\
\mathcal{Y}_{2}(:,:,1,1)=\left(\begin{smallmatrix}
             \mathbf{2} & \mathbf{i} \\
            \mathbf{0} & \mathbf{0}
           \end{smallmatrix}\right),\\
\mathcal{Y}_{2}(:,:,1,2)=\left(\begin{smallmatrix}
             \mathbf{3-k} & \mathbf{0} \\
            \mathbf{0} & \mathbf{i}
           \end{smallmatrix}\right),\
\mathcal{Y}_{2}(:,:,2,1)=\left(\begin{smallmatrix}
             \mathbf{0} & \mathbf{0} \\
            \mathbf{i} & \mathbf{j}
           \end{smallmatrix}\right),\
 \mathcal{Y}_{2}(:,:,2,2)=\left(\begin{smallmatrix}
             \mathbf{0} & \mathbf{0} \\
            \mathbf{j} & \mathbf{k}
           \end{smallmatrix}\right),\\
\mathcal{W}(:,:,1,1)=\left(\begin{smallmatrix}
             \mathbf{i-k} & \mathbf{0} \\
            \mathbf{0} & \mathbf{0}
           \end{smallmatrix}\right),\
\mathcal{W}(:,:,1,2)=\left(\begin{smallmatrix}
             \mathbf{j-k} & \mathbf{0} \\
            \mathbf{0} & \mathbf{0}
           \end{smallmatrix}\right),\
\mathcal{W}(:,:,2,1)=\left(\begin{smallmatrix}
             \mathbf{1+i} & \mathbf{0} \\
            \mathbf{0} & \mathbf{1-i}
           \end{smallmatrix}\right),\\
\mathcal{W}(:,:,2,2)=\left(\begin{smallmatrix}
             \mathbf{0} & \mathbf{0} \\
            \mathbf{0} & \mathbf{2j}
           \end{smallmatrix}\right).
\end{gathered}
\end{align*}
\end{small}
\end{example}

\section{\textbf{Some implementations of the main system}}
\begin{remark}
\label{system 22t1}
Set $\mathcal{A}_{4}=\mathcal{B}_{4}=\mathcal{A}_{5}=\mathcal{B}_{5}=0$ in the system \eqref{1.5}, wel obtain the solvability conditions and the general solution for the system of tensor equations:
\begin{small}
\begin{align}
\label{1.6a}
\left\{
\begin{array}{rll}
\begin{gathered}
\mathcal{A}_{1}*_{N}\mathcal{X}_{3}*_{M}\mathcal{B}_{1}=\mathcal{E}_{1},\\
\mathcal{A}_{2}*_{N}\mathcal{Y}_{3}*_{M}\mathcal{B}_{2}=\mathcal{E}_{2},\\
\mathcal{A}_{3}*_{N}\mathcal{W}=\mathcal{E}_{3},\ \mathcal{W}*_{M}\mathcal{B}_{3}=\mathcal{E}_{4},\\
\mathcal{A}_{6}*_{N}\mathcal{X}_{1}*_{M}\mathcal{B}_{6}+\mathcal{A}_{7}*_{N}\mathcal{X}_{2}*_{M}\mathcal{B}_{7}\\
+\mathcal{A}_{7}*_{N}(\mathcal{C}_{3}*_{N}\mathcal{X}_{3}*_{M}\mathcal{D}_{3}
+\mathcal{C}_{4}*_{N}\mathcal{W}*_{M}\mathcal{D}_{4})*_{M}\mathcal{B}_{6}=\mathcal{E}_{9},\\
\mathcal{A}_{8}*_{N}\mathcal{Y}_{1}*_{M}\mathcal{B}_{8}+\mathcal{A}_{9}*_{N}\mathcal{Y}_{2}*_{M}\mathcal{B}_{9}\\
+\mathcal{A}_{9}*_{N}(\mathcal{H}_{3}*_{N}\mathcal{Y}_{3}*_{M}\mathcal{J}_{3}
+\mathcal{H}_{4}*_{N}\mathcal{W}*_{M}\mathcal{J}_{4})*_{M}\mathcal{B}_{8}=\mathcal{E}_{10},
\end{gathered}
\end{array}
  \right.
\end{align}
\end{small}
\end{remark}
Now, we consider the solvability conditions and the general solution to the system \eqref{1.6}, as a particular case of \eqref{1.5}.
\begin{theorem}\label{system 3.33AAAAA}
Consider the  system of tensor equations \eqref{1.5}, where
\begin{small}
\begin{align*}
&\mathcal{A}_{1} \in \mathbb{H}^{I(N)\times J(N)},\ \mathcal{A}_{2} \in \mathbb{H}^{I(N)\times Q(N)},\ \mathcal{A}_{3} \in \mathbb{H}^{I(N)\times P(N)},
\ \mathcal{A}_{6} \in \mathbb{H}^{A(N)\times E(N)},\\
&\mathcal{A}_{8} \in \mathbb{H}^{C(N)\times R(N)},\
 \mathcal{B}_{1} \in \mathbb{H}^{L(M)\times K(M)},\ \mathcal{B}_{2} \in \mathbb{H}^{S(M)\times K(M)},\ \mathcal{B}_{3} \in \mathbb{H}^{T(M)\times K(M)},\\
& \mathcal{B}_{7} \in \mathbb{H}^{H(N)\times B(N)},\ \mathcal{B}_{9} \in \mathbb{H}^{O(M)\times D(M)},\ \mathcal{C}_{3} \in \mathbb{H}^{I(N)\times J(N)},\ \mathcal{C}_{4} \in \mathbb{H}^{I(N)\times P(N)},\\
&\mathcal{D}_{3} \in \mathbb{H}^{L(M)\times K(M)},\ \mathcal{D}_{4} \in \mathbb{H}^{T(M)\times K(M)},\
\mathcal{H}_{3} \in \mathbb{H}^{I(N)\times Q(N)},\ \mathcal{H}_{4} \in \mathbb{H}^{I(N)\times P(N)},\\
&\mathcal{J}_{3} \in \mathbb{H}^{S(M)\times K(M)},\ \mathcal{J}_{4} \in \mathbb{H}^{T(M)\times K(M)},\
\mathcal{E}_{1} \in \mathbb{H}^{I(N)\times K(M)},\ \mathcal{E}_{2} \in \mathbb{H}^{I(N)\times K(M)},\\
&\mathcal{E}_{3} \in \mathbb{H}^{I(N)\times T(M)},\ \mathcal{E}_{4} \in \mathbb{H}^{P(N)\times K(M)},\
\mathcal{E}_{9} \in \mathbb{H}^{A(N)\times B(M)},\ \mathcal{E}_{10} \in \mathbb{H}^{C(N)\times D(M)}
\end{align*}
\end{small}
are given tensors over $\mathbb{H}$. Set
\begin{small}
\begin{subequations}
\begin{align*}
\begin{gathered}
\mathcal{\widehat{A}}_{4}=\mathcal{R}_{\mathcal{A}_{6}}*_{N}\mathcal{C}_{3},\ \mathcal{\widehat{C}}_{4}=\mathcal{R}_{\mathcal{A}_{6}}*_{N}\mathcal{C}_{4},\ \mathcal{\widehat{B}}_{4}=\mathcal{D}_{3}*_{M}\mathcal{L}_{\mathcal{B}_{7}},\\
\mathcal{\widehat{D}}_{4}=\mathcal{D}_{4}*_{M}\mathcal{L}_{\mathcal{B}_{7}},\ \mathcal{\widehat{P}}=\mathcal{R}_{\mathcal{A}_{6}}*_{N}\mathcal{E}_{9}*_{M}\mathcal{L}_{\mathcal{B}_{7}},\ \mathcal{\widehat{A}}_{5}=\mathcal{R}_{\mathcal{A}_{8}}*_{N}\mathcal{H}_{3},\  \mathcal{\widehat{C}}_{5}=\mathcal{R}_{\mathcal{A}_{8}}\\
*_{N}\mathcal{H}_{4},\
\mathcal{\widehat{B}}_{5}=\mathcal{J}_{3}*_{M}\mathcal{L}_{\mathcal{B}_{9}},\
\mathcal{\widehat{D}}_{5}=\mathcal{J}_{4}*_{M}\mathcal{L}_{\mathcal{B}_{9}},\
\mathcal{\widehat{Q}}=\mathcal{R}_{\mathcal{A}_{8}}*_{N}\mathcal{E}_{10}*_{M}\mathcal{L}_{\mathcal{B}_{9}},\\
\mathcal{C}_{6}=\mathcal{\widehat{C}}_{4}*_{N}\mathcal{L}_{\mathcal{A}_{3}},\ \mathcal{D}_{6}=\mathcal{R}_{\mathcal{B}_{3}}*_{M}\mathcal{\widehat{D}}_{4},\
\mathcal{C}_{7}=\mathcal{\widehat{C}}_{5}*_{N}\mathcal{L}_{\mathcal{A}_{3}},\ \mathcal{D}_{7}=\mathcal{R}_{\mathcal{B}_{3}}*_{M}\mathcal{\widehat{D}}_{5},\\
\mathcal{G}=\mathcal{\widehat{P}}-\mathcal{\widehat{C}}_{4}*_{N}\mathcal{A}_{3}^{\dagger}*_{N}\mathcal{E}_{3}*_{M}\mathcal{\widehat{D}}_{4}
-\mathcal{\widehat{C}}_{4}*_{N}\mathcal{L}_{\mathcal{A}_{3}}*_{N}\mathcal{E}_{4}*_{M}\mathcal{B}_{3}^{\dagger}*_{M}\mathcal{\widehat{D}}_{4},\\
\mathcal{F}=\mathcal{\widehat{Q}}-\mathcal{\widehat{C}}_{5}*_{N}\mathcal{A}_{3}^{\dagger}*_{N}\mathcal{E}_{3}*_{M}\mathcal{\widehat{D}}_{5}
-\mathcal{\widehat{C}}_{5}*_{N}\mathcal{L}_{\mathcal{A}_{3}}*_{N}\mathcal{E}_{4}*_{M}\mathcal{B}_{3}^{\dagger}*_{M}\mathcal{\widehat{D}}_{5},\\
\mathcal{M}_{1}=\mathcal{R}_{\mathcal{\widehat{A}}_{4}}*_{M}\mathcal{C}_{6},\ \mathcal{N}_{1}=\mathcal{D}_{6}*_{M}\mathcal{L}_{\mathcal{\widehat{B}}_{4}},\
\mathcal{S}_{1}=\mathcal{C}_{6}*_{N}\mathcal{L}_{\mathcal{M}_{1}},\\
\mathcal{M}_{2}=\mathcal{R}_{\mathcal{\widehat{A}}_{5}}*_{M}\mathcal{C}_{7},\ \mathcal{N}_{2}=\mathcal{D}_{7}*_{M}\mathcal{L}_{\mathcal{\widehat{B}}_{5}},\
\mathcal{S}_{2}=\mathcal{C}_{7}*_{N}\mathcal{L}_{\mathcal{M}_{2}},\\
\mathcal{A}_{11}=\begin{pmatrix}\mathcal{L}_{\mathcal{M}_{1}}*_{N}\mathcal{L}_{\mathcal{S}_{1}}& \mathcal{L}_{\mathcal{M}_{2}}*_{N}\mathcal{L}_{\mathcal{S}_{2}}\end{pmatrix},\
\mathcal{B}_{11}=\begin{pmatrix}\mathcal{R}_{\mathcal{D}_{6}}\\ \mathcal{R}_{\mathcal{D}_{7}}\end{pmatrix},\\
\begin{array}{l}
\mathcal{E}_{11}=\mathcal{M}_{2}^{\dagger}*_{N}\mathcal{F}*_{M}\mathcal{D}_{7}^{\dagger}
+\mathcal{S}_{2}^{\dagger}*_{N}\mathcal{S}_{2}*_{N}\mathcal{C}_{7}^{\dagger}
*_{N}\mathcal{F}*_{M}\mathcal{N}_{2}^{\dagger}
-\mathcal{M}_{1}^{\dagger}*_{N}\mathcal{G}*_{M}\mathcal{D}_{6}^{\dagger}\\
\ \ \ \ \ \ \ -\mathcal{S}_{1}^{\dagger}*_{N}\mathcal{S}_{1}*_{N}\mathcal{C}_{6}^{\dagger}*_{N}\mathcal{G}*_{M}\mathcal{N}_{1}^{\dagger},\ \mathcal{A}=\mathcal{R}_{\mathcal{A}_{11}}*_{N}\mathcal{L}_{\mathcal{M}_{1}},
\end{array}\\
\
\mathcal{B}=\mathcal{R}_{\mathcal{N}_{1}}*_{M}\mathcal{L}_{\mathcal{B}_{11}},\
\mathcal{C}=\mathcal{R}_{\mathcal{A}_{11}}*_{N}\mathcal{L}_{\mathcal{M}_{2}},\
\mathcal{D}=\mathcal{R}_{\mathcal{N}_{2}}*_{M}\mathcal{L}_{\mathcal{B}_{11}},\ \mathcal{E}=\mathcal{R}_{\mathcal{A}_{11}}\\
*_{N}\mathcal{E}_{11}*_{M}\mathcal{L}_{\mathcal{B}_{11}},\
\mathcal{M}=\mathcal{R}_{\mathcal{A}}*_{N}\mathcal{C},\
\mathcal{N}=\mathcal{D}*_{M}\mathcal{L}_{\mathcal{B}},\
\mathcal{S}=\mathcal{C}*_{N}\mathcal{L}_{\mathcal{M}},\\
\mathcal{A}_{22}=\begin{pmatrix}\mathcal{L}_{\mathcal{A}_{1}}& \mathcal{L}_{\mathcal{\widehat{A}}_{4}}\end{pmatrix},\
\mathcal{B}_{22}=\begin{pmatrix}\mathcal{R}_{\mathcal{B}_{1}}\\ \mathcal{R}_{\mathcal{\widehat{B}}_{4}}\end{pmatrix},\
\mathcal{C}_{22}=\mathcal{\widehat{A}}_{4}^{\dagger}*_{N}\mathcal{S}_{1},\
\mathcal{D}_{22}=\mathcal{R}_{\mathcal{N}_{1}}*_{M}\\
\begin{array}{l}
\mathcal{D}_{6}*_{N}\mathcal{\widehat{B}}_{4}^{\dagger},\ \mathcal{E}_{22}=\mathcal{\widehat{A}}_{4}^{\dagger}*_{N}\mathcal{G}*_{M}\mathcal{\widehat{B}}_{4}^{\dagger}
-\mathcal{A}_{1}^{\dagger}*_{N}\mathcal{E}_{1}*_{M}\mathcal{B}_{1}^{\dagger}
-\mathcal{\widehat{A}}_{4}^{\dagger}*_{N}\mathcal{S}_{1}*_{N}\mathcal{C}_{6}^{\dagger}\\
\ \ \ \ \ \ *_{N}\mathcal{G}*_{M}\mathcal{N}_{1}*_{M}\mathcal{D}_{6}*_{M}\mathcal{\widehat{B}}_{4}^{\dagger}-\mathcal{\widehat{A}}_{4}^{\dagger}*_{N}\mathcal{C}_{6}*_{N}\mathcal{M}_{1}^{\dagger}
*_{N}\mathcal{G}*_{M}\mathcal{\widehat{B}}_{4}^{\dagger},
\end{array}\\
\mathcal{A}_{33}=\mathcal{R}_{\mathcal{A}_{22}}*_{N}\mathcal{C}_{22},\
\mathcal{B}_{33}=\mathcal{D}_{22}*_{M}\mathcal{L}_{\mathcal{B}_{22}},\
\mathcal{E}_{33}=\mathcal{R}_{\mathcal{A}_{22}}*_{N}\mathcal{E}_{22}*_{M}\mathcal{L}_{\mathcal{B}_{22}},\\
\mathcal{A}_{44}=\begin{pmatrix}\mathcal{L}_{\mathcal{A}_{2}}& \mathcal{L}_{\mathcal{\widehat{A}}_{5}}\end{pmatrix},\
\mathcal{B}_{44}=\begin{pmatrix}\mathcal{R}_{\mathcal{B}_{2}}\\ \mathcal{R}_{\mathcal{\widehat{B}}_{5}}\end{pmatrix},\
\mathcal{C}_{44}=\mathcal{\widehat{A}}_{5}^{\dagger}*_{N}\mathcal{S}_{2},\ \mathcal{D}_{44}=\\
\begin{array}{l}
\mathcal{R}_{\mathcal{N}_{2}}*_{M}\mathcal{D}_{7}*_{N}\mathcal{\widehat{B}}_{5}^{\dagger},\ \mathcal{E}_{44}=\mathcal{\widehat{A}}_{5}^{\dagger}*_{N}\mathcal{F}*_{M}\mathcal{\widehat{B}}_{5}^{\dagger}
-\mathcal{A}_{2}^{\dagger}*_{N}\mathcal{E}_{2}*_{M}\mathcal{B}_{2}^{\dagger}
-\mathcal{\widehat{A}}_{5}^{\dagger}*_{N}\mathcal{S}_{2}\\
\ \ \ \ \ \ *_{N}\mathcal{C}_{7}^{\dagger}*_{N}\mathcal{F}*_{M}\mathcal{N}_{2}*_{M}\mathcal{D}_{7}*_{M}\mathcal{\widehat{B}}_{5}^{\dagger}-\mathcal{\widehat{A}}_{5}^{\dagger}*_{N}\mathcal{C}_{7}*_{N}\mathcal{M}_{2}^{\dagger}*_{N}\mathcal{F}
*_{M}\mathcal{\widehat{B}}_{5}^{\dagger},
\end{array}\\
\mathcal{A}_{55}=\mathcal{R}_{\mathcal{A}_{44}}*_{N}\mathcal{C}_{44},\
\mathcal{B}_{55}=\mathcal{D}_{44}*_{M}\mathcal{L}_{\mathcal{B}_{44}},\
\mathcal{E}_{55}=\mathcal{R}_{\mathcal{A}_{44}}*_{N}\mathcal{E}_{44}*_{M}\mathcal{L}_{\mathcal{B}_{44}},
\end{gathered}
\end{align*}
\begin{align*}
\begin{gathered}
\mathcal{A}_{66}=\begin{pmatrix}\mathcal{L}_{\mathcal{A}}& \mathcal{L}_{\mathcal{A}_{33}}\end{pmatrix},\
\mathcal{B}_{66}=\begin{pmatrix}\mathcal{R}_{\mathcal{B}}\\ \mathcal{R}_{\mathcal{B}_{33}}\end{pmatrix},\
\mathcal{C}_{66}=\mathcal{A}^{\dagger}*_{N}\mathcal{S},\
\mathcal{D}_{66}=\\
\begin{array}{l}
\mathcal{R}_{\mathcal{N}}*_{M}\mathcal{D}*_{M}\mathcal{B}^{\dagger},\ \mathcal{E}_{66}=\mathcal{A}_{33}^{\dagger}*_{N}\mathcal{E}_{33}*_{M}\mathcal{B}_{33}^{\dagger}
-\mathcal{A}^{\dagger}*_{N}\mathcal{E}*_{M}\mathcal{B}^{\dagger}
+\mathcal{A}^{\dagger}*_{N}\mathcal{S}\\
\ \ \ \ \ \ *_{N}\mathcal{C}^{\dagger}*_{N}\mathcal{E}*_{M}\mathcal{N}^{\dagger}*_{M}\mathcal{D}*_{N}\mathcal{B}^{\dagger}+\mathcal{A}^{\dagger}*_{N}\mathcal{C}*_{N}\mathcal{M}^{\dagger}*_{N}\mathcal{E}
*_{M}\mathcal{B}^{\dagger},
\end{array}\\
\mathcal{A}_{77}=\mathcal{R}_{\mathcal{A}_{66}}*_{N}\mathcal{C}_{66},\
\mathcal{B}_{77}=\mathcal{D}_{66}*_{M}\mathcal{L}_{\mathcal{B}_{66}},\
\mathcal{E}_{77}=\mathcal{R}_{\mathcal{A}_{66}}*_{N}\mathcal{E}_{66}*_{M}\mathcal{L}_{\mathcal{B}_{66}},\\
\mathcal{A}_{88}=\begin{pmatrix}\mathcal{L}_{\mathcal{M}}*_{N}\mathcal{L}_{\mathcal{S}}& \mathcal{L}_{\mathcal{A}_{55}}\end{pmatrix},\
\mathcal{B}_{88}=\begin{pmatrix}\mathcal{R}_{\mathcal{D}}\\ \mathcal{R}_{\mathcal{B}_{55}}\end{pmatrix},\
\mathcal{C}_{88}=\mathcal{L}_{\mathcal{M}},\
\mathcal{D}_{88}=\mathcal{R}_{\mathcal{N}},\ \mathcal{E}_{88}\\
=\mathcal{A}_{55}^{\dagger}*_{N}\mathcal{E}_{55}*_{M}\mathcal{B}_{55}^{\dagger}
-\mathcal{M}^{\dagger}*_{N}\mathcal{E}*_{M}\mathcal{B}^{\dagger}
-\mathcal{S}^{\dagger}*_{N}\mathcal{S}*_{N}\mathcal{C}^{\dagger}*_{N}\mathcal{E}*_{M}\mathcal{N}^{\dagger},\\
\mathcal{A}_{99}=\mathcal{R}_{\mathcal{A}_{88}}*_{N}\mathcal{C}_{88},\
\mathcal{B}_{99}=\mathcal{D}_{88}*_{M}\mathcal{L}_{\mathcal{B}_{88}},\
\mathcal{E}_{99}=\mathcal{R}_{\mathcal{A}_{88}}*_{N}\mathcal{E}_{88}*_{M}\mathcal{L}_{\mathcal{B}_{88}},\\
\mathcal{\widetilde{A}}=\begin{pmatrix}\mathcal{L}_{\mathcal{A}_{77}}& -\mathcal{L}_{\mathcal{A}_{99}}\end{pmatrix},\
\mathcal{\widetilde{B}}=\begin{pmatrix}\mathcal{R}_{\mathcal{B}_{77}}\\ -\mathcal{R}_{\mathcal{B}_{99}}\end{pmatrix},\\
\mathcal{\widetilde{E}}=\mathcal{A}_{77}^{\dagger}*_{N}\mathcal{E}_{77}*_{M}\mathcal{B}_{77}^{\dagger}
-\mathcal{A}_{99}^{\dagger}*_{N}\mathcal{E}_{99}*_{M}\mathcal{B}_{99}^{\dagger},\
\end{gathered}
\end{align*}
\end{subequations}
\end{small}
then the system \eqref{1.6} is solvable if and only if
\begin{small}
\begin{align*}
&\mathcal{R}_{\mathcal{A}_{3}}*_{N}\mathcal{E}_{3}=0,\ \mathcal{E}_{4}*_{M}\mathcal{L}_{\mathcal{B}_{3}}=0,\
\mathcal{A}_{3}*_{N}\mathcal{E}_{3}=\mathcal{E}_{4}*_{M}\mathcal{B}_{3},\\
&\mathcal{R}_{\mathcal{\widehat{A}}_{4}}*_{N}\mathcal{G}*_{M}\mathcal{L}_{\mathcal{D}_{6}}=0,\
\mathcal{R}_{\mathcal{C}_{6}}*_{N}\mathcal{G}*_{M}\mathcal{L}_{\mathcal{\widehat{B}}_{4}}=0,\
\mathcal{R}_{\mathcal{S}_{1}}*_{N}\mathcal{R}_{\mathcal{\widehat{A}}_{4}}*_{N}\mathcal{G}=0,\\
&\mathcal{G}*_{M}\mathcal{L}_{\mathcal{\widehat{B}}_{4}}*_{M}\mathcal{L}_{\mathcal{N}_{1}}=0,\
\mathcal{R}_{\mathcal{\widehat{A}}_{5}}*_{N}\mathcal{F}*_{M}\mathcal{L}_{\mathcal{D}_{7}}=0,\
\mathcal{R}_{\mathcal{C}_{7}}*_{N}\mathcal{F}*_{M}\mathcal{L}_{\mathcal{\widehat{B}}_{5}}=0,\\
&\mathcal{R}_{\mathcal{S}_{2}}*_{N}\mathcal{R}_{\mathcal{\widehat{A}}_{5}}*_{N}\mathcal{F}=0,\
\mathcal{F}*_{M}\mathcal{L}_{\mathcal{\widehat{B}}_{5}}*_{M}\mathcal{L}_{\mathcal{N}_{2}}=0,\
\mathcal{R}_{\mathcal{M}}*_{N}\mathcal{R}_{\mathcal{A}}*_{N}\mathcal{E}=0,\\
&\mathcal{R}_{\mathcal{A}}*_{N}\mathcal{E}*_{M}\mathcal{L}_{\mathcal{D}}=0,\
\mathcal{E}*_{M}\mathcal{L}_{\mathcal{B}}*_{M}\mathcal{L}_{\mathcal{N}}=0,\
\mathcal{R}_{\mathcal{C}}*_{N}\mathcal{E}*_{M}\mathcal{L}_{\mathcal{B}}=0,\\
&\mathcal{R}_{\mathcal{A}_{1}}*_{N}\mathcal{E}_{1}=0,\ \mathcal{E}_{1}*_{M}\mathcal{L}_{\mathcal{B}_{1}}=0,\
\mathcal{R}_{\mathcal{A}_{2}}*_{N}\mathcal{E}_{2}=0,\ \mathcal{E}_{2}*_{M}\mathcal{L}_{\mathcal{B}_{2}}=0,\\
&\mathcal{R}_{\mathcal{A}_{33}}*_{N}\mathcal{E}_{33}=0,\ \mathcal{E}_{33}*_{M}\mathcal{L}_{\mathcal{B}_{33}}=0,\
\mathcal{R}_{\mathcal{A}_{55}}*_{N}\mathcal{E}_{55}=0,\ \mathcal{E}_{55}*_{M}\mathcal{L}_{\mathcal{B}_{55}}=0,\\
&\mathcal{R}_{\mathcal{A}_{77}}*_{N}\mathcal{E}_{77}=0,\ \mathcal{E}_{77}*_{M}\mathcal{L}_{\mathcal{B}_{77}}=0,\
\mathcal{R}_{\mathcal{A}_{99}}*_{N}\mathcal{E}_{99}=0,\ \mathcal{E}_{99}*_{M}\mathcal{L}_{\mathcal{B}_{99}}=0,\\
&\mathcal{R}_{\mathcal{\widetilde{A}}}*_{N}\mathcal{\widetilde{E}}*_{M}\mathcal{L}_{\widetilde{B}}=0.
\end{align*}
\end{small}
Under these constraints, the general solution to \eqref{1.5} can be expressed as:
\begin{small}
\begin{align*}
&\mathcal{X}_{1}=\mathcal{A}_{6}^{\dagger}*_{N}\mathcal{\grave{E}}_{1}-\mathcal{\widehat{T}}_{21}*_{M}\mathcal{B}_{7}+\mathcal{L}_{\mathcal{A}_{6}}*_{N}\mathcal{T}_{41},\\
&\mathcal{X}_{2}=\mathcal{R}_{\mathcal{A}_{6}}*_{N}\mathcal{\grave{E}}_{1}*_{M}B_{7}+\mathcal{A}_{6}*_{N}\mathcal{\widehat{T}}_{21}
+\mathcal{T}_{31}*_{M}\mathcal{R}_{\mathcal{B}_{7}},\\
&\mathcal{Y}_{1}=\mathcal{A}_{8}^{\dagger}*_{N}\mathcal{\grave{E}}_{2}-\mathcal{\widehat{J}}_{21}*_{M}\mathcal{B}_{9}+\mathcal{L}_{\mathcal{A}_{8}}*_{N}\mathcal{J}_{41},\\
&\mathcal{Y}_{2}=\mathcal{R}_{\mathcal{A}_{8}}*_{N}\mathcal{\grave{E}}_{2}*_{M}B_{9}+\mathcal{A}_{8}*_{N}\mathcal{\widehat{J}}_{21}
+\mathcal{J}_{31}*_{M}\mathcal{R}_{\mathcal{B}_{9}},\\
&\mathcal{X}_{3}=\mathcal{A}_{1}^{\dagger}*_{N}\mathcal{E}_{1}*_{M}\mathcal{B}_{1}^{\dagger}
+\mathcal{L}_{\mathcal{A}_{1}}*_{N}\mathcal{U}_{1}+\mathcal{U}_{2}*_{M}\mathcal{R}_{\mathcal{B}_{1}},\\
&\mathcal{Y}_{3}=\mathcal{A}_{2}^{\dagger}*_{N}\mathcal{E}_{2}*_{M}\mathcal{B}_{2}^{\dagger}
+\mathcal{L}_{\mathcal{A}_{2}}*_{N}\mathcal{U}_{3}+\mathcal{U}_{4}*_{M}\mathcal{R}_{\mathcal{B}_{2}},\\
&\mathcal{W}=\mathcal{A}_{3}^{\dagger}*_{N}\mathcal{E}_{3}+\mathcal{L}_{\mathcal{A}_{3}}*_{N}\mathcal{E}_{4}*_{M}\mathcal{B}_{3}^{\dagger}
+\mathcal{L}_{\mathcal{A}_{3}}*_{N}\mathcal{U}_{5}*_{M}\mathcal{R}_{\mathcal{B}_{3}},\\
&\mathcal{\grave{E}}_{1}=\mathcal{E}_{9}-\mathcal{C}_{3}*_{N}\mathcal{X}_{3}*_{M}\mathcal{D}_{3}
-\mathcal{C}_{4}*_{N}\mathcal{W}*_{M}\mathcal{D}_{4},\\
&\mathcal{\grave{E}}_{2}=\mathcal{E}_{10}-\mathcal{H}_{3}*_{N}\mathcal{Y}_{3}*_{M}\mathcal{J}_{3}
-\mathcal{H}_{4}*_{N}\mathcal{W}*_{M}\mathcal{J}_{4},
\end{align*}
\end{small}
and $\mathcal{U}_{i}$ $(i=\overline{1,5})$ are defined by \eqref{system 2222eeee}-\eqref{system 2222iiii}
with $\mathcal{W}_{i}$, $\mathcal{K}_{i}$\, $\mathcal{H}_{jk}$ $(i,k=\overline{1,3},\ j=\overline{1,4}),$ $\mathcal{T}_{l1},\ \mathcal{\widehat{T}}_{21}$, $\mathcal{J}_{l1}$ and $\mathcal{\widehat{J}}_{21}$  $(l=3,4)$ are arbitrary  tensors with suitable orders.
\end{theorem}
\begin{proof}
Apply $Theorem$ \ref{system 3.33AAt}, whenever $\mathcal{A}_{4}=\mathcal{A}_{5}=\mathcal{B}_{4}=\mathcal{B}_{5}=0$ and $\mathcal{B}_{6}=\mathcal{B}_{8}=\mathcal{A}_{7}=\mathcal{A}_{9}=\mathcal{I}.$
\end{proof}
On utilizing $Theorem$ \ref{system 3.33AAAAA}, we derive the solvability conditions and the general solution to  \eqref{1.7a}, where the quaternion tensors $\mathcal{X}_{3}$, $\mathcal{Y}_{3}$, and $\mathcal{W}$ are $\eta$-Hermitian.
\begin{theorem}\label{system 3.33AAAAa}
Consider the  system of tensor equations \eqref{1.7a}, where
\begin{small}
\begin{align*}
&\mathcal{A}_{1} \in \mathbb{H}^{I(N)\times J(N)},\ \mathcal{A}_{2} \in \mathbb{H}^{I(N)\times Q(N)},\ \mathcal{A}_{3} \in \mathbb{H}^{I(N)\times P(N)},
\ \mathcal{A}_{6} \in \mathbb{H}^{I(N)\times E(N)},\\
&\mathcal{A}_{8} \in \mathbb{H}^{I(N)\times R(N)},\
\mathcal{C}_{3} \in \mathbb{H}^{I(N)\times J(N)},\ \mathcal{C}_{4} \in \mathbb{H}^{I(N)\times P(N)},\
\mathcal{H}_{3} \in \mathbb{H}^{I(N)\times Q(N)}, \\
&\mathcal{H}_{4} \in \mathbb{H}^{I(N)\times P(N)},\
\mathcal{E}_{i}=\mathcal{E}_{i}^{\eta^{*}} \in \mathbb{H}^{I(N)\times I(N)},\
\mathcal{E}_{3} \in \mathbb{H}^{I(N)\times P(N)}\ \ \ (i \in \{1,2,9,10\})
\end{align*}
\end{small}
 are given tensors over $\mathbb{H}$. Set
\begin{small}
\begin{subequations}
\begin{align*}
\begin{gathered}
\mathcal{\widehat{A}}_{4}=\mathcal{R}_{\mathcal{A}_{6}}*_{N}\mathcal{C}_{3},\ \mathcal{\widehat{C}}_{4}=\mathcal{R}_{\mathcal{A}_{6}}*_{N}\mathcal{C}_{4},\
\mathcal{\widehat{P}}=\mathcal{R}_{\mathcal{A}_{6}}*_{N}\mathcal{E}_{9}*_{M}(\mathcal{R}_{\mathcal{A}_{6}})^{\eta^{*}},\\
\mathcal{\widehat{A}}_{5}=\mathcal{R}_{\mathcal{A}_{8}}*_{N}\mathcal{H}_{3},\  \mathcal{\widehat{C}}_{5}=\mathcal{R}_{\mathcal{A}_{8}}*_{N}\mathcal{H}_{4},\
\mathcal{\widehat{Q}}=\mathcal{R}_{\mathcal{A}_{8}}*_{N}\mathcal{E}_{10}*_{M}(\mathcal{R}_{\mathcal{A}_{8}})^{\eta^{*}},\\
\mathcal{C}_{6}=\mathcal{\widehat{C}}_{4}*_{N}\mathcal{L}_{\mathcal{A}_{3}},\
\mathcal{C}_{7}=\mathcal{\widehat{C}}_{5}*_{N}\mathcal{L}_{\mathcal{A}_{3}},\ \mathcal{G}=\mathcal{\widehat{P}}-
\mathcal{\widehat{C}}_{4}*_{N}\mathcal{A}_{3}^{\dagger}*_{N}\mathcal{E}_{3}*_{N}(\mathcal{\widehat{C}}_{4})^{\eta^{*}}\\
-\mathcal{\widehat{C}}_{4}*_{N}\mathcal{L}_{\mathcal{A}_{3}}*_{N}\mathcal{E}_{3}^{\eta^{*}}*_{N}(\mathcal{A}_{3}^{\dagger})^{\eta^{*}}
*_{N}(\mathcal{\widehat{C}}_{4})^{\eta^{*}},\ \mathcal{F}=\mathcal{\widehat{Q}}-\mathcal{\widehat{C}}_{5}*_{N}\mathcal{A}_{3}^{\dagger}*_{N}\mathcal{E}_{3}*_{N}\\
(\mathcal{\widehat{C}}_{5})^{\eta^{*}}
-\mathcal{\widehat{C}}_{5}*_{N}\mathcal{L}_{\mathcal{A}_{3}}*_{N}\mathcal{E}_{3}^{\eta^{*}}*_{N}(\mathcal{A}_{3}^{\dagger})^{\eta^{*}}
*_{N}(\mathcal{\widehat{C}}_{5})^{\eta^{*}}, \mathcal{M}_{1}=\mathcal{R}_{\mathcal{\widehat{A}}_{4}}*_{N}\mathcal{C}_{6},\\
\mathcal{S}_{1}=\mathcal{C}_{6}*_{N}\mathcal{L}_{\mathcal{M}_{1}},\ \mathcal{M}_{2}=\mathcal{R}_{\mathcal{\widehat{A}}_{5}}*_{N}\mathcal{C}_{7},\
\mathcal{S}_{2}=\mathcal{C}_{7}*_{N}\mathcal{L}_{\mathcal{M}_{2}},\\
\mathcal{A}_{11}=\begin{pmatrix}\mathcal{L}_{\mathcal{M}_{1}}*_{N}\mathcal{L}_{\mathcal{S}_{1}}& \mathcal{L}_{\mathcal{M}_{2}}*_{N}\mathcal{L}_{\mathcal{S}_{2}}\end{pmatrix},\
\mathcal{B}_{11}=\begin{pmatrix}\mathcal{R}_{\mathcal{C}_{6}^{\eta^{*}}}\\ \mathcal{R}_{\mathcal{C}_{7}^{\eta^{*}}}\end{pmatrix},\\
\begin{array}{l}
\mathcal{E}_{11}=\mathcal{N}_{2}^{\dagger}*_{N}\mathcal{F}*_{N}(\mathcal{C}_{7}^{\dagger})^{\eta^{*}}
+\mathcal{S}_{2}^{\dagger}*_{N}\mathcal{S}_{2}*_{N}\mathcal{C}_{7}^{\dagger}
*_{N}\mathcal{F}*_{N}(\mathcal{N}_{2}^{\dagger})^{\eta^{*}}
-\mathcal{N}_{1}^{\dagger}\\
\ \ \ \ \ \ \ *_{N}\mathcal{G}*_{N}(\mathcal{C}_{6}^{\dagger})^{\eta^{*}}-\mathcal{S}_{1}^{\dagger}*_{N}\mathcal{S}_{1}*_{N}\mathcal{C}_{6}^{\dagger}*_{N}\mathcal{G}*_{N}(\mathcal{N}_{1}^{\dagger})^{\eta^{*}},\ \mathcal{A}=\mathcal{R}_{\mathcal{A}_{11}}*_{N}
\end{array}\\
\mathcal{L}_{\mathcal{M}_{1}},\
\mathcal{B}=\mathcal{R}_{\mathcal{M}_{1}^{\eta^{*}}}*_{M}\mathcal{L}_{\mathcal{B}_{11}},\
\mathcal{C}=\mathcal{R}_{\mathcal{A}_{11}}*_{N}\mathcal{L}_{\mathcal{M}_{2}},\
\mathcal{D}=\mathcal{R}_{\mathcal{M}_{2}^{\eta^{*}}}*_{M}\mathcal{L}_{\mathcal{B}_{11}},\ \mathcal{E}\\
=\mathcal{R}_{\mathcal{A}_{11}}*_{N}\mathcal{E}_{11}*_{M}\mathcal{L}_{\mathcal{B}_{11}},\
\mathcal{M}=\mathcal{R}_{\mathcal{A}}*_{N}\mathcal{C},\
\mathcal{N}=\mathcal{D}*_{M}\mathcal{L}_{\mathcal{B}},\
\mathcal{S}=\mathcal{C}*_{N}\mathcal{L}_{\mathcal{M}},\\
\mathcal{A}_{22}=\begin{pmatrix}\mathcal{L}_{\mathcal{A}_{1}}& \mathcal{L}_{\mathcal{\widehat{A}}_{4}}\end{pmatrix},\
\mathcal{C}_{22}=\mathcal{\widehat{A}}_{4}^{\dagger}*_{N}\mathcal{C}_{6}*_{N}\mathcal{L}_{\mathcal{M}_{1}},\ \mathcal{A}_{33}=\mathcal{R}_{\mathcal{A}_{22}}*_{N}\mathcal{C}_{22},\\
\begin{array}{l}
\mathcal{E}_{22}=\mathcal{\widehat{A}}_{4}^{\dagger}*_{N}\mathcal{G}*_{M}(\mathcal{\widehat{A}}_{4}^{\dagger})^{\eta^{*}}
-\mathcal{A}_{1}^{\dagger}*_{N}\mathcal{E}_{1}*_{M}(\mathcal{A}_{1}^{\dagger})^{\eta^{*}}
-\mathcal{\widehat{A}}_{4}^{\dagger}*_{N}\mathcal{S}_{1}*_{N}\mathcal{C}_{6}^{\dagger}*_{N}\mathcal{G}
\\
\ \ \ \ \ \ *_{M}\mathcal{M}_{1}^{\eta^{*}}*_{M}\mathcal{C}_{6}^{\eta^{*}}*_{M}(\mathcal{\widehat{A}}_{4}^{\dagger})^{\eta^{*}}-\mathcal{\widehat{A}}_{4}^{\dagger}*_{N}\mathcal{C}_{6}*_{N}\mathcal{M}_{1}^{\dagger}
*_{N}\mathcal{G}*_{M}(\mathcal{\widehat{A}}_{4}^{\dagger})^{\eta^{*}},
\end{array}\\
\mathcal{E}_{33}=\mathcal{R}_{\mathcal{A}_{22}}*_{N}\mathcal{E}_{22}*_{N}(\mathcal{R}_{\mathcal{A}_{22}})^{\eta^{*}},\
\mathcal{A}_{44}=\begin{pmatrix}\mathcal{L}_{\mathcal{A}_{2}}& \mathcal{L}_{\mathcal{\widehat{A}}_{5}}\end{pmatrix},\
\mathcal{C}_{44}=\mathcal{\widehat{A}}_{5}^{\dagger}*_{N}\mathcal{C}_{7}\\
\begin{array}{l}
*_{N}\mathcal{L}_{\mathcal{M}_{2}},\mathcal{E}_{44}=\mathcal{\widehat{A}}_{5}^{\dagger}*_{N}\mathcal{F}*_{N}(\mathcal{\widehat{A}}_{5}^{\dagger})^{\eta^{*}}
-\mathcal{A}_{2}^{\dagger}*_{N}\mathcal{E}_{2}*_{N}(\mathcal{A}_{2}^{\dagger})^{\eta^{*}}
-\mathcal{\widehat{A}}_{5}^{\dagger}*_{N}\mathcal{S}_{2}*_{N}
\\
\mathcal{C}_{7}^{\dagger}*_{N}\mathcal{F}*_{N}\mathcal{M}_{2}^{\eta^{*}}*_{N}\mathcal{C}_{7}^{\eta^{*}}*_{N}(\mathcal{\widehat{A}}_{5}^{\dagger})^{\eta^{*}}-\mathcal{\widehat{A}}_{5}^{\dagger}*_{N}\mathcal{C}_{7}*_{N}\mathcal{M}_{2}^{\dagger}*_{N}\mathcal{F}
*_{M}(\mathcal{\widehat{A}}_{5}^{\dagger})^{\eta^{*}},
\end{array}\\
\mathcal{A}_{55}=\mathcal{R}_{\mathcal{A}_{44}}*_{N}\mathcal{C}_{44},\
\mathcal{E}_{55}=\mathcal{R}_{\mathcal{A}_{44}}*_{N}\mathcal{E}_{44}*_{M}(\mathcal{R}_{\mathcal{A}_{44}})^{\eta^{*}},\mathcal{A}_{66}=\begin{pmatrix}\mathcal{L}_{\mathcal{A}}& \mathcal{L}_{\mathcal{A}_{33}}\end{pmatrix},\\
\mathcal{B}_{66}=\begin{pmatrix}\mathcal{R}_{\mathcal{B}}\\ \mathcal{R}_{\mathcal{A}_{33}^{\eta^{*}}}\end{pmatrix},\
\mathcal{C}_{66}=\mathcal{A}^{\dagger}*_{N}\mathcal{S},\
\mathcal{D}_{66}=\mathcal{R}_{\mathcal{N}}*_{N}\mathcal{D}*_{N}\mathcal{B}^{\dagger},\\
\begin{array}{l}
\mathcal{E}_{66}=\mathcal{A}_{33}^{\dagger}*_{N}\mathcal{E}_{33}*_{N}(\mathcal{A}_{33}^{\dagger})^{\eta^{*}}
-\mathcal{A}^{\dagger}*_{N}\mathcal{E}*_{N}\mathcal{B}^{\dagger}
+\mathcal{A}^{\dagger}*_{N}\mathcal{S}*_{N}\mathcal{C}^{\dagger}*_{N}\mathcal{E}\\
\ \ \ \ \ \ *_{N}\mathcal{N}^{\dagger}*_{M}\mathcal{D}*_{N}\mathcal{B}^{\dagger}+\mathcal{A}^{\dagger}*_{N}\mathcal{C}*_{N}\mathcal{N}^{\dagger}*_{N}\mathcal{E}
*_{N}\mathcal{B}^{\dagger},
\end{array}\\
\mathcal{A}_{77}=\mathcal{R}_{\mathcal{A}_{66}}*_{N}\mathcal{C}_{66},\
\mathcal{B}_{77}=\mathcal{D}_{66}*_{M}\mathcal{L}_{\mathcal{B}_{66}},\
\mathcal{E}_{77}=\mathcal{R}_{\mathcal{A}_{66}}*_{N}\mathcal{E}_{66}*_{M}\mathcal{L}_{\mathcal{B}_{66}},\\
\mathcal{A}_{88}=\begin{pmatrix}\mathcal{L}_{\mathcal{M}}*_{N}\mathcal{L}_{\mathcal{S}}& \mathcal{L}_{\mathcal{A}_{55}}\end{pmatrix},\
\mathcal{B}_{88}=\begin{pmatrix}\mathcal{R}_{\mathcal{D}}\\ \mathcal{R}_{\mathcal{A}_{55}^{\eta^{*}}}\end{pmatrix},\
\mathcal{C}_{88}=\mathcal{L}_{\mathcal{M}},\
\mathcal{D}_{88}=\mathcal{R}_{\mathcal{N}},\\
\mathcal{E}_{88}=\mathcal{A}_{55}^{\dagger}*_{N}\mathcal{E}_{55}*_{N}(\mathcal{A}_{55}^{\dagger})^{\eta^{*}}
-\mathcal{M}^{\dagger}*_{N}\mathcal{E}*_{N}\mathcal{B}^{\dagger}
-\mathcal{S}^{\dagger}*_{N}\mathcal{S}*_{N}\mathcal{C}^{\dagger}*_{N}\mathcal{E}*_{M}\\
\mathcal{N}^{\dagger},\ \mathcal{A}_{99}=\mathcal{R}_{\mathcal{A}_{88}}*_{N}\mathcal{C}_{88},\
\mathcal{B}_{99}=\mathcal{D}_{88}*_{M}\mathcal{L}_{\mathcal{B}_{88}},
\end{gathered}
\end{align*}
\begin{align*}
\begin{gathered}
\mathcal{E}_{99}=\mathcal{R}_{\mathcal{A}_{88}}*_{N}\mathcal{E}_{88}*_{M}\mathcal{L}_{\mathcal{B}_{88}},
\ \mathcal{\widetilde{A}}=\begin{pmatrix}\mathcal{L}_{\mathcal{A}_{77}}& -\mathcal{L}_{\mathcal{A}_{99}}\end{pmatrix},\\
\mathcal{\widetilde{B}}=\begin{pmatrix}\mathcal{R}_{\mathcal{B}_{77}}\\ -\mathcal{R}_{\mathcal{B}_{99}}\end{pmatrix},\
\mathcal{\widetilde{E}}=\mathcal{A}_{77}^{\dagger}*_{N}\mathcal{E}_{77}*_{M}\mathcal{B}_{77}^{\dagger}
-\mathcal{A}_{99}^{\dagger}*_{N}\mathcal{E}_{99}*_{M}\mathcal{B}_{99}^{\dagger},\
\end{gathered}
\end{align*}
\end{subequations}
\end{small}
Then  \eqref{1.6} is solvable if and only if
\begin{small}
\begin{align*}
\begin{gathered}
\mathcal{R}_{\mathcal{A}_{3}}*_{N}\mathcal{E}_{3}=0,\
\mathcal{A}_{3}*_{N}\mathcal{E}_{3}=(\mathcal{A}_{3}*_{N}\mathcal{E}_{3})^{\eta^{*}},\
\mathcal{R}_{\mathcal{A}_{1}}*_{N}\mathcal{E}_{1}=0,\\
\mathcal{R}_{\mathcal{A}_{2}}*_{N}\mathcal{E}_{2}=0,\
\mathcal{R}_{\mathcal{\widehat{A}}_{4}}*_{N}\mathcal{G}*_{N}\mathcal{L}_{\mathcal{C}_{6}^{\eta^{*}}}=0,\
\mathcal{R}_{\mathcal{C}_{6}}*_{N}\mathcal{G}*_{M}\mathcal{L}_{\mathcal{\widehat{A}}_{4}^{\eta^{*}}}=0,\\
\mathcal{R}_{\mathcal{S}_{1}}*_{N}\mathcal{R}_{\mathcal{\widehat{A}}_{4}}*_{N}\mathcal{G}=0,\
\mathcal{G}*_{M}\mathcal{L}_{\mathcal{\widehat{A}}_{4}^{\eta^{*}}}*_{N}\mathcal{L}_{\mathcal{M}_{1}^{\eta^{*}}}=0,\\
\mathcal{R}_{\mathcal{\widehat{A}}_{5}}*_{N}\mathcal{F}*_{M}\mathcal{L}_{\mathcal{C}_{7}^{\eta^{*}}}=0,\ \mathcal{R}_{\mathcal{C}_{7}}*_{N}\mathcal{F}*_{N}\mathcal{L}_{\mathcal{\widehat{A}}_{5}^{\eta^{*}}}=0,\
\mathcal{R}_{\mathcal{S}_{2}}*_{N}\mathcal{R}_{\mathcal{\widehat{A}}_{5}}*_{N}\mathcal{F}=0,\\
\mathcal{F}*_{N}\mathcal{L}_{\mathcal{\widehat{A}}_{5}^{\eta^{*}}}*_{N}\mathcal{L}_{\mathcal{M}_{2}^{\eta^{*}}}=0,\
\mathcal{R}_{\mathcal{M}}*_{N}\mathcal{R}_{\mathcal{A}}*_{N}\mathcal{E}=0,\
\mathcal{R}_{\mathcal{A}}*_{N}\mathcal{E}*_{M}\mathcal{L}_{\mathcal{D}}=0,\\
\mathcal{E}*_{M}\mathcal{L}_{\mathcal{B}}*_{M}\mathcal{L}_{\mathcal{N}}=0,\
\mathcal{R}_{\mathcal{C}}*_{N}\mathcal{E}*_{M}\mathcal{L}_{\mathcal{B}}=0,\
\mathcal{R}_{\mathcal{A}_{33}}*_{N}\mathcal{E}_{33}=0,\\
\mathcal{E}_{33}*_{M}\mathcal{L}_{\mathcal{B}_{33}}=0,\
\mathcal{R}_{\mathcal{A}_{55}}*_{N}\mathcal{E}_{55}=0,\ \mathcal{E}_{55}*_{M}\mathcal{L}_{\mathcal{B}_{55}}=0,\
\mathcal{R}_{\mathcal{A}_{77}}*_{N}\mathcal{E}_{77}=0,\\
\mathcal{E}_{77}*_{M}\mathcal{L}_{\mathcal{B}_{77}}=0,\
\mathcal{R}_{\mathcal{A}_{99}}*_{N}\mathcal{E}_{99}=0,\ \mathcal{E}_{99}*_{M}\mathcal{L}_{\mathcal{B}_{99}}=0,\\
\mathcal{R}_{\mathcal{\widetilde{A}}}*_{N}\mathcal{\widetilde{E}}*_{M}\mathcal{L}_{\widetilde{B}}=0.
\end{gathered}
\end{align*}
\end{small}
Under these conditions, the general solution to \eqref{1.5} can be expressed as:
\begin{small}
\begin{align}
\label{1.6ssssk}
&(\mathcal{X}_{1},\mathcal{Y}_{1},\mathcal{X}_{3}, \mathcal{Y}_{3}, \mathcal{W})=\frac{1}{2}
\left(\mathcal{X}_{11}+\mathcal{X}_{12}^{\eta^{*}},
\mathcal{Y}_{11}+\mathcal{Y}_{12}^{\eta^{*}},
\mathcal{X}_{33}+\mathcal{X}_{33}^{\eta^{*}},
\mathcal{Y}_{33}+\mathcal{Y}_{33}^{\eta^{*}},
\mathcal{W}_{1}+\mathcal{W}_{1}^{\eta^{*}}\right)
\end{align}
\end{small}
where
\begin{small}
\begin{align*}
\begin{gathered}
\mathcal{X}_{11}=\mathcal{A}_{6}^{\dagger}*_{N}\mathcal{\grave{E}}_{1}-\mathcal{\widehat{T}}_{21}*_{M}\mathcal{A}_{6}^{\eta^{*}}+\mathcal{L}_{\mathcal{A}_{6}}*_{N}\mathcal{T}_{41},\\
\mathcal{X}_{12}=\mathcal{R}_{\mathcal{A}_{6}}*_{N}\mathcal{\grave{E}}_{1}*_{M}\mathcal{A}_{6}^{\eta^{*}}+\mathcal{A}_{6}*_{N}\mathcal{\widehat{T}}_{21}
+\mathcal{T}_{31}*_{M}\mathcal{R}_{\mathcal{A}_{6}^{\eta^{*}}},\\
\mathcal{Y}_{11}=\mathcal{A}_{8}^{\dagger}*_{N}\mathcal{\grave{E}}_{2}-\mathcal{\widehat{J}}_{21}*_{M}\mathcal{A}_{8}^{\eta^{*}}
+\mathcal{L}_{\mathcal{A}_{8}}*_{N}\mathcal{J}_{41},\\
\mathcal{Y}_{12}=\mathcal{R}_{\mathcal{A}_{8}}*_{N}\mathcal{\grave{E}}_{2}*_{M}\mathcal{A}_{8}^{\eta^{*}}+\mathcal{A}_{8}*_{N}\mathcal{\widehat{J}}_{21}
+\mathcal{J}_{31}*_{M}\mathcal{R}_{\mathcal{A}_{8}^{\eta^{*}}},\\
\mathcal{X}_{33}=\mathcal{A}_{1}^{\dagger}*_{N}\mathcal{E}_{1}*_{M}(\mathcal{A}_{1}^{\dagger})^{\eta^{*}}
+\mathcal{L}_{\mathcal{A}_{1}}*_{N}\mathcal{U}_{1}+\mathcal{U}_{2}*_{M}\mathcal{R}_{\mathcal{A}_{1}^{\eta^{*}}},\\
\mathcal{Y}_{33}=\mathcal{A}_{2}^{\dagger}*_{N}\mathcal{E}_{2}*_{M}(\mathcal{A}_{2}^{\dagger})^{\eta^{*}}
+\mathcal{L}_{\mathcal{A}_{2}}*_{N}\mathcal{U}_{3}+\mathcal{U}_{4}*_{M}\mathcal{R}_{\mathcal{A}_{2}^{\eta^{*}}},\\
\mathcal{W}_{1}=\mathcal{A}_{3}^{\dagger}*_{N}\mathcal{E}_{3}+\mathcal{L}_{\mathcal{A}_{3}}*_{N}\mathcal{E}_{3}^{\eta^{*}}*_{N}(\mathcal{A}_{3}^{\dagger})^{\eta^{*}}
+\mathcal{L}_{\mathcal{A}_{3}}*_{N}\mathcal{U}_{5}*_{M}\mathcal{R}_{\mathcal{A}_{3}^{\eta^{*}}},\\
\mathcal{\grave{E}}_{1}=\mathcal{E}_{9}-\mathcal{C}_{3}*_{N}\mathcal{X}_{3}*_{M}\mathcal{C}_{3}^{\eta^{*}}
-\mathcal{C}_{4}*_{N}\mathcal{W}*_{M}\mathcal{C}_{4}^{\eta^{*}},\\
\mathcal{\grave{E}}_{2}=\mathcal{E}_{10}-\mathcal{H}_{3}*_{N}\mathcal{Y}_{3}*_{M}\mathcal{H}_{3}^{\eta^{*}}
-\mathcal{H}_{4}*_{N}\mathcal{W}*_{M}\mathcal{H}_{4}^{\eta^{*}},\\
\begin{array}{l}
 \mathcal{U}_{1}=\begin{pmatrix}\mathcal{I}\ \  0\end{pmatrix}*_{N}(\mathcal{A}_{22}^{\dagger}*_{N}
(\mathcal{E}_{22}-\mathcal{C}_{22}*_{N}\mathcal{V}_{2}*_{N}\mathcal{C}_{22}^{\eta^{*}})
-\mathcal{A}_{22}^{\dagger}*_{N}\mathcal{H}_{12}*_{N}\mathcal{A}_{22}^{\eta^{*}},\\
 \ \ \ \ \ \ +\mathcal{L}_{\mathcal{A}_{22}}*_{N}\mathcal{H}_{11}),
 \end{array}\\
\begin{array}{l}
 \mathcal{U}_{2}=(\mathcal{R}_{\mathcal{A}_{22}}*_{N}(\mathcal{E}_{22}-\mathcal{C}_{22}*_{N}\mathcal{V}_{2}*_{M}\mathcal{C}_{22}^{\eta^{*}})
*_{N}(\mathcal{A}_{22}^{\dagger})^{\eta^{*}}+\mathcal{A}_{22}*_{N}\mathcal{A}_{22}^{\dagger}\\
\ \ \ \ \ \  *_{N}\mathcal{H}_{12}\mathcal{H}_{13}*_{M}\mathcal{R}_{\mathcal{A}_{22}^{\eta^{*}}})*_{M}\begin{pmatrix}\mathcal{I}\\ 0\end{pmatrix},
 \end{array}\\
 \begin{array}{l}
 \mathcal{U}_{3}=\begin{pmatrix}\mathcal{I}& 0\end{pmatrix}*_{N}(\mathcal{A}_{44}^{\dagger}*_{N}
(\mathcal{E}_{44}-\mathcal{C}_{44}*_{N}\mathcal{T}_{2}*_{N}\mathcal{C}_{44}^{\eta^{*}})
-\mathcal{A}_{44}^{\dagger}*_{N}\mathcal{H}_{22}*_{N}\mathcal{A}_{44}^{\eta^{*}}\\
\ \ \ \ \ \ +\mathcal{L}_{\mathcal{A}_{44}}*_{N}\mathcal{H}_{21}),
 \end{array}
 \end{gathered}
 \end{align*}
\begin{align*}
&\begin{array}{l}
 \mathcal{U}_{4}=(\mathcal{R}_{\mathcal{A}_{44}}*_{N}(\mathcal{E}_{44}-\mathcal{C}_{44}*_{N}\mathcal{T}_{2}*_{M}\mathcal{C}_{44}^{\eta^{*}})
*_{N}(\mathcal{A}_{44}^{\dagger})^{\eta^{*}}+\mathcal{A}_{44}*_{N}\mathcal{A}_{44}^{\dagger}*_{N}\\
\ \ \ \ \ \  \mathcal{H}_{22}\mathcal{H}_{23}*_{M}\mathcal{R}_{\mathcal{A}_{44}}^{\eta^{*}})*_{M}\begin{pmatrix}\mathcal{I}\\ 0\end{pmatrix},
 \end{array}\\
&\begin{array}{l}
 \mathcal{U}_{5}=\mathcal{M}_{1}^{\dagger}*_{N}\mathcal{G}*_{N}(\mathcal{C}_{6}^{\dagger})^{\eta^{*}}+\mathcal{S}_{1}^{\dagger}*_{N}\mathcal{S}_{1}
 *_{N}\mathcal{C}_{6}^{\dagger}
*_{N}\mathcal{G}*_{N}(\mathcal{M}_{1}^{\dagger})^{\eta^{*}}-\mathcal{L}_{\mathcal{M}_{1}}*_{N}\\
\ \ \ \ \ \ \mathcal{L}_{\mathcal{S}_{1}}*_{N}\mathcal{V}_{1}+\mathcal{L}_{\mathcal{M}_{1}^{\eta^{*}}}*_{N}\mathcal{V}_{2}*_{M}\mathcal{R}_{\mathcal{N}_{1}}
+\mathcal{V}_{3}*_{M}\mathcal{R}_{\mathcal{C}_{6}^{\eta^{*}}},
 \end{array}\\
&\begin{array}{l}
 \mathcal{V}_{1}=\begin{pmatrix}\mathcal{I}& 0\end{pmatrix}*_{N}(\mathcal{A}_{11}^{\dagger}*_{N}
(\mathcal{E}_{11}-\mathcal{L}_{\mathcal{M}_{1}}*_{N}\mathcal{V}_{2}*_{N}\mathcal{R}_{\mathcal{M}_{1}^{\eta^{*}}}-
\mathcal{L}_{\mathcal{M}_{2}}*_{N}\mathcal{T}_{2}*_{N}\\
\ \ \ \ \ \ \mathcal{R}_{\mathcal{M}_{2}^{\eta^{*}}})+\mathcal{W}_{1}*_{N}\mathcal{B}_{11}
+\mathcal{L}_{\mathcal{A}_{11}}*_{N}\mathcal{W}_{2}),
 \end{array}
 \end{align*}
\end{small}
and $\mathcal{V}_{2},$ $\mathcal{V}_{3},$ $\mathcal{T}_{2}$, $\mathcal{W}_{m}$ $(m=\overline{4,8}),$ $\mathcal{Q}_{1}$ and $\mathcal{Q}_{2}$ can be deduced from \eqref{system 22k}-\eqref{system 22t}
with $\mathcal{W}_{i}$, $\mathcal{K}_{i}$\, $\mathcal{H}_{jk}$ $(i,k=\overline{1,3},\ j=\overline{1,4}),$ $\mathcal{T}_{l1},\ \mathcal{\widehat{T}}_{21}$, $\mathcal{J}_{l1}$ and $\mathcal{\widehat{J}}_{21}$  $(l=3,4)$ are arbitrary.
\end{theorem}
\begin{proof} Consider the following  system of tensor equations:
\begin{small}
\begin{align}
\label{1.7aa}
\left\{
\begin{array}{rll}
\begin{gathered}
\mathcal{A}_{3}*_{N}\mathcal{W}_{1}=\mathcal{E}_{3},\ \mathcal{W}_{1}*_{N}\mathcal{A}_{3}^{\eta^{*}}=\mathcal{E}_{3}^{\eta^{*}},\\
\mathcal{A}_{1}*_{N}\mathcal{X}_{33}*_{N}\mathcal{A}_{1}^{\eta^{*}}=\mathcal{E}_{1},\
\mathcal{A}_{2}*_{N}\mathcal{Y}_{33}*_{N}\mathcal{A}_{2}^{\eta^{*}}=\mathcal{E}_{2},\\
\mathcal{A}_{6}*_{N}\mathcal{X}_{11}+\mathcal{X}_{12}*_{N}\mathcal{A}_{6}^{\eta^{*}}
+\mathcal{C}_{3}*_{N}\mathcal{X}_{33}*_{N}\mathcal{C}_{3}^{\eta^{*}}
+\mathcal{C}_{4}*_{N}\mathcal{W}_{1}*_{N}\mathcal{C}_{4}^{\eta^{*}}=\mathcal{E}_{9},\\
\mathcal{A}_{8}*_{N}\mathcal{Y}_{11}+\mathcal{Y}_{12}*_{N}\mathcal{A}_{8}^{\eta^{*}}
+\mathcal{H}_{3}*_{N}\mathcal{Y}_{3}*_{N}\mathcal{H}_{3}^{\eta^{*}}
+\mathcal{H}_{4}*_{N}\mathcal{W}_{1}*_{N}\mathcal{H}_{4}^{\eta^{*}}=\mathcal{E}_{10}.
\end{gathered}
\end{array}
  \right.
\end{align}
\end{small}
First, we show that  \eqref{1.7a} is solvable if and only if  \eqref{1.7aa} is solvable. Claim that $(\mathcal{X}_{1}, \mathcal{Y}_{1}, \mathcal{X}_{3}, \mathcal{Y}_{3}, \mathcal{W})$ is a solution to  \eqref{1.7a}, then it is evident that  $(\mathcal{X}_{11}, \mathcal{X}_{12}, \mathcal{Y}_{11}, \mathcal{Y}_{12}, \mathcal{X}_{33}, \mathcal{Y}_{33}, \mathcal{W}_{1})$ $=(\mathcal{X}_{1}, \mathcal{X}_{1}^{\eta^{*}}, \mathcal{Y}_{1}, \mathcal{Y}_{1}^{\eta^{*}},\\ \mathcal{X}_{3}, \mathcal{Y}_{3}, \mathcal{W})$ is a solution to  \eqref{1.7aa}. Conversely, if $(\mathcal{X}_{11}, \mathcal{X}_{12}, \mathcal{Y}_{11}, \mathcal{Y}_{12}, \mathcal{X}_{33}, \mathcal{Y}_{33}, \mathcal{W}_{1})$ is a solution to \eqref{1.7aa}. Now, we show that the formula  \eqref{1.6ssssk}
can be a solution to system \eqref{1.7a}. It is known that  $\mathcal{X}_{3},\mathcal{Y}_{3},$ and $\mathcal{W}$ are  $\eta$-Hermitian tensors. By Applying \eqref{1.6ssssk} on  \eqref{1.7a} yields:
\begin{small}
\begin{align*}
&\mathcal{A}_{1}*_{N}\mathcal{X}_{3}*_{N}\mathcal{A}_{1}^{\eta^{*}}
=\mathcal{A}_{1}*_{N}\left(\frac{\mathcal{X}_{33}+\mathcal{X}_{33}^{\eta^{*}}}{2}\right)*_{N}A_{1}^{\eta^{*}}\\
&=\frac{1}{2}\mathcal{A}_{1}*_{N}\mathcal{X}_{33}*_{N}\mathcal{A}_{1}^{\eta^{*}}
+\frac{1}{2}\left(\mathcal{A}_{1}*_{N}\mathcal{X}_{33}*_{N}\mathcal{A}_{1}^{\eta^{*}}\right)^{\eta^{*}}
=\mathcal{E}_{1}.
\end{align*}
\end{small}
Similarly, it can be verified that
\begin{small}
\begin{align*}
&\mathcal{A}_{2}*_{N}\mathcal{Y}_{3}*_{N}\mathcal{A}_{2}^{\eta^{*}}=\mathcal{E}_{2},\\
&\mathcal{A}_{3}*_{N}\mathcal{W}
=\mathcal{A}_{3}*_{N}\left(\frac{\mathcal{W}_{1}+\mathcal{W}_{1}^{\eta^{*}}}{2}\right)
=\frac{1}{2}\mathcal{A}_{3}*_{N}\mathcal{W}_{1}
+\frac{1}{2}\left(\mathcal{W}_{1}*_{N}\mathcal{A}_{3}^{\eta^{*}}\right)^{\eta^{*}}
=\mathcal{E}_{3},\\
&\mathcal{A}_{6}*_{N}\mathcal{X}_{1}+(\mathcal{A}_{6}*_{N}\mathcal{X}_{1})^{\eta^{*}}
+\mathcal{C}_{3}*_{N}\mathcal{X}_{3}*_{N}\mathcal{C}_{3}^{\eta^{*}}
+\mathcal{C}_{4}*_{N}\mathcal{W}*_{N}\mathcal{C}_{4}^{\eta^{*}}\\
&=\mathcal{A}_{6}*_{N}\left(\frac{\mathcal{X}_{11}+\mathcal{X}_{12}^{\eta^{*}}}{2}\right)
+\left(\mathcal{A}_{6}*_{N}\left(\frac{\mathcal{X}_{11}+\mathcal{X}_{12}^{\eta^{*}}}{2}\right)\right)^{\eta^{*}}\\
&\mathcal{C}_{3}*_{N}\left(\frac{\mathcal{X}_{33}+\mathcal{X}_{33}^{\eta^{*}}}{2}\right)*_{N}\mathcal{C}_{3}^{\eta^{*}}
+\mathcal{C}_{4}*_{N}\left(\frac{\mathcal{W}_{1}+\mathcal{W}_{1}^{\eta^{*}}}{2}\right)*_{N}\mathcal{C}_{4}^{\eta^{*}}\\
&=\frac{1}{2}\left[ \mathcal{A}_{6}*_{N}\mathcal{X}_{11}+\mathcal{X}_{12}*_{N}\mathcal{A}_{6}^{\eta^{*}}
+\mathcal{C}_{3}*_{N}\mathcal{X}_{33}*_{N}\mathcal{C}_{3}^{\eta^{*}}
+\mathcal{C}_{4}*_{N}\mathcal{W}_{1}*_{N}\mathcal{C}_{4}^{\eta^{*}}\right]\\
&+\frac{1}{2}\left[ \mathcal{A}_{6}*_{N}\mathcal{X}_{11}+\mathcal{X}_{12}*_{N}\mathcal{A}_{6}^{\eta^{*}}
+\mathcal{C}_{3}*_{N}\mathcal{X}_{33}*_{N}\mathcal{C}_{3}^{\eta^{*}}
+\mathcal{C}_{4}*_{N}\mathcal{W}_{1}*_{N}\mathcal{C}_{4}^{\eta^{*}}\right]^{\eta^{*}}=\mathcal{E}_{9}.
\end{align*}
\end{small}
Similarly, it can be found that
\begin{small}
\begin{align*}
\mathcal{A}_{8}*_{N}\mathcal{Y}_{1}+(\mathcal{A}_{8}*_{N}\mathcal{Y}_{1})^{\eta^{*}}
+\mathcal{H}_{3}*_{N}\mathcal{Y}_{3}*_{M}\mathcal{H}_{3}^{\eta^{*}}
+\mathcal{H}_{4}*_{N}\mathcal{W}*_{M}\mathcal{H}_{4}^{\eta^{*}}=\mathcal{E}_{10}.
 \end{align*}
\end{small}
 Therefore, \eqref{1.6ssssk} is a solution to  \eqref{1.7a}. Consequently, apply $Theorem$ $\ref{system 3.33AAAAA}$ on  \eqref{1.7aa}, we can establish the solvability conditions and the general solution to  \eqref{1.7a}.
\end{proof}

\section{\textbf{Conclusion}}
We derive the necessary and sufficient algebriac conditions for the existence of a solution to  \eqref{1.5}  in $Theorem$ \ref{system 3.33AAt}. We established an explicit formula of the general solution in terms of the Moore-Penrose inverses of some block-given tensors. An algorithm with a numerical example is investigated to compute the general solution to  \eqref{1.5}. As a particular case of  \eqref{1.5}, we discuss the solvability conditions and the general solution to \eqref{1.6} in $Theorem$ \ref{system 3.33AAAAA}. As an implementation of $Theorem$ \ref{system 3.33AAAAA}, we carry out the solvability conditions and an expression of the general solution to  \eqref{1.7a}, whenever  $\mathcal{X}_{3},$  $\mathcal{Y}_{3}$ and $\mathcal{W}$ are $\eta$-Hermitian tensors. All results are valid over an arbitrary division ring.

As a consequence the main findings in Section 3, we infer the solvability constraints of the two-sided linear matrix equation $A_{1}X_{1}B_{1}+A_{2}X_{2}B_{2}+A_{2}(C_{3}X_{3}D_{3}+C_{4}WD_{4})B_{1}$
    $=E_{1}$ can be characterized by Moore-Penrose inverses of some provided matrices and rank equalities. Consequently, we can obtain the solvability constraints and the general solution to the following system of two-sided and coupled matrix equations:
\begin{small}
\begin{align*}
\left\{
\begin{gathered}
A_{6}X_{1}B_{6}+A_{7}X_{2}B_{7}+A_{7}(C_{3}X_{3}D_{3}+C_{4}WD_{4})B_{6}=E_{9},\\
A_{8}Y_{1}B_{8}+A_{9}Y_{2}B_{9}+A_{9}(H_{3}Y_{3}J_{3}+H_{4}WD_{4})J_{8}=E_{10}.
\end{gathered}
\right.
\end{align*}
\end{small}
with respect to
\begin{small}
\begin{align*}
\left\{
\begin{gathered}
A_{1}X_{3}B_{1}=E_{1},\ A_{2}Y_{3}B_{2}=E_{2},\\
A_{4}X_{1}=E_{5},\ X_{2}B_{4}=E_{6},\\
A_{5}Y_{1}=E_{7},\ Y_{2}B_{5}=E_{8},\\
A_{3}W=E_{3},\ WB_{3}=E_{4}.
\end{gathered}
\right.
\end{align*}
\end{small}

\end{document}